\RequirePackage{fix-cm}
\documentclass[11pt,reqno]{amsart}

\makeatletter
\renewcommand\section{\@startsection{section}{1}%
  \z@{.7\linespacing}{.5\linespacing}%
  {\normalfont\scshape\centering}}
\makeatother
\usepackage{amsmath,amssymb,amsthm,geometry,needspace}
\usepackage[colorlinks,linkcolor=blue,citecolor=blue]{hyperref}

\numberwithin{equation}{section}

\newtheorem{theorem}{Theorem}[section]
\newtheorem{proposition}[theorem]{Proposition}
\newtheorem{lemma}[theorem]{Lemma}

\theoremstyle{definition}

\theoremstyle{remark}

\newcommand{\R}{\mathbb R}
\newcommand{\N}{\mathbb N}

\newcommand{\C}{\mathbb C}
\newcommand{\T}{\mathbb T}
\newcommand{\X}{X}

\newcommand{\abs}[1]{\left\lvert#1\right\rvert}

\allowdisplaybreaks[2]
\title[Hartman--Grobman Theorem for the Navier--Stokes Equation on $\mathbb T^3$]
{A Hartman--Grobman Theorem for the Navier--Stokes Equation on $\mathbb T^3$}

\author{Rongchang Liu}
\address[R. Liu]{School of Mathematics, Sichuan University,
Chengdu 610064, P.\ R.\ China}
\email[R.~Liu]{rcliu@scu.edu.cn}
\author{Kening Lu}
\address[K. Lu]{School of Mathematics, Sichuan University,
Chengdu 610064, P.\ R.\ China}
\email[K.~Lu]{keninglu@scu.edu.cn}
\author{Lin Shi}
\address[L. Shi]{School of Mathematical Sciences,
University of Electronic Science and Technology of China,
Chengdu 611731, P.\ R.\ China}
\email[L.~Shi]{shilinlavender@163.com}

\thanks{This work was supported by the Fundamental Research Funds for
the Central Universities and the National Natural Science Foundation
of China (Grant Nos.~12090010, 12090013 and 12471154).}


\begin{document}
\raggedbottom

\begin{abstract}
We prove a local Hartman–Grobman theorem at zero for the unforced Navier–Stokes equation on the three-dimensional torus within the strong topology of $H^s$ for every $s>7/2$. More precisely, we construct a
homeomorphism between open positively invariant neighborhoods of zero
that conjugates the Navier--Stokes semiflow to the heat semigroup.

We develop an infinite-dimensional boundary--time scheme built
on a nested hierarchy of exact asymptotic dynamical factors. Their fiber
geometry yields Sobolev block coordinates with autonomous
finite-dimensional factor systems, on which boundary--time topological
conjugacies can be constructed at every finite level and then assembled
on a common neighborhood into the desired $H^s$ conjugacy. The
infinite-dimensional assembly relies on new spectral estimates for the
divergence-free transport structure, which control derivative loss and
the blockwise tails of the finite-level conjugacies and their inverses
without an additional spectral-gap assumption.
\end{abstract}

\maketitle
\setcounter{tocdepth}{1}
\tableofcontents

\baselineskip13.6pt

\section{Introduction}\label{intro:section}
A fundamental problem in dynamical systems is whether the nonlinear dynamics near an equilibrium can be reduced, by a change of coordinates, to its linearization. Beginning with Poincar\'e \cite{Poincare}, this
question has played a central role in smooth and analytic dynamics,
through the work of Siegel, Sternberg, Arnold, Moser, Zehnder, and many
others \cite{Siegel,Arnold,Moser,Sternberg,Zehnder}. We consider this question for the unforced incompressible
Navier--Stokes equation
\begin{equation}\label{intro:velocity-equation}
 \partial_tu+(u\cdot\nabla)u+\nabla p=\nu\Delta u,
 \qquad \nabla\cdot u=0,\qquad \nu>0,
\end{equation}
on the three-dimensional torus $\T^3$. In the Sobolev phase space $X=H^s_\sigma(\T^3;\R^3)$ of divergence-free zero-mean
velocities with $s>7/2$, for sufficiently small initial data, the equation generates
a semiflow $S_t$ of global strong solutions, whose linearization at the
origin is the heat semigroup $e^{\nu t\Delta}$.

\vskip0.05in
Our main result, Theorem~\ref{app:main}, shows that these two semiflows are locally topologically
conjugate. More precisely, there are open neighborhoods $U,V$ of zero in $X$, positively invariant under $S_t$ and $e^{\nu t\Delta}$ respectively, and a homeomorphism $\mathcal H:U\to V$ in the strong $H^s$
topology such that
\begin{equation}\label{intro:goal}
 \mathcal H(0)=0,\qquad
 \mathcal H(S_tu)=e^{\nu t\Delta}\mathcal H(u),
 \qquad u\in U,\quad t\ge0.
\end{equation}
Standard energy estimates show that every sufficiently small solution is driven to zero but do not resolve the nonlinear coupling between Fourier modes. The above result is substantially stronger: after a single change of coordinates absorbing the coupling, the entire local Navier--Stokes dynamics becomes the heat flow. Thus the conjugacy identifies the local topological organization of the nonlinear flow, not merely the asymptotic decay of individual trajectories.

\vskip0.05in

For invertible dynamics, the classical Hartman--Grobman theorem gives such a topological conjugacy for
finite-dimensional $C^1$ flows near hyperbolic equilibria and for
$C^1$ diffeomorphisms near hyperbolic fixed points
\cite{Grobman,Hartman}. Palis and Pugh extended the diffeomorphism
theorem to Banach spaces \cite{Palis1968,Pugh1969}.
Further developments concern generalized hyperbolicity and
nonautonomous dynamics
\cite{BernardesMessaoudi,BernardesStability,BackesDragicevicGeneralized},
as well as semilinear hyperbolic evolution equations
\cite{HeinPruss2016,LuPintoXia}. We also mention that the regularity of a linearizing conjugacy is already a subtle issue in finite dimensions, as resonances show: the smoothness of the vector field need not be inherited by the conjugacy \cite{Hartman60,Sell2}. In the present work we restrict ourselves to topological conjugacy and leave open whether the conjugacy in Theorem~\ref{app:main} can be chosen $C^1$.

\vskip0.05in

For parabolic equations, smoothing prevents a direct application of the invertible-dynamics theorems since positive-time maps are not locally invertible on the original phase space. For noninvertible mappings, Aulbach and Garay obtained partial linearizations under additional spectral assumptions \cite{AulbachGaray1994}, while Lu \cite{Lu1991} and Bates--Lu \cite{BatesLu} established full phase-space conjugacies for several one-dimensional parabolic equations by exploiting growing spectral gaps to control nonlinear coupling. Our recent work \cite{LiuLuShi2D} uses planar spectral asymptotics to select bands whose boundary gaps compensate for derivative loss in two-dimensional semilinear parabolic equations. For nonlinearities without derivative loss, a different construction in \cite{LiuLuShiRD} uses fixed dyadic bands and requires neither lower bounds on consecutive spectral gaps nor uniform bounds on spectral multiplicities, allowing reaction--diffusion equations on tori of arbitrary dimension. These methods fail for \eqref{intro:velocity-equation} as the convection term loses one derivative, while the Stokes spectrum lacks the separation used in \cite{LiuLuShi2D} to recover it. 

\vskip0.05in

We overcome these difficulties through what we call an
\emph{infinite-dimensional boundary--time scheme}. The picture is simple in finite dimensions: for a contracting nonlinear
flow whose vector field points inward on a small sphere, a nonzero point
is determined by where its backward trajectory meets the boundary and by
the time elapsed from that intersection. Starting from the same boundary
point under the linearized flow and evolving for the same amount of time
then defines the conjugacy between the nonlinear flow and its
linearization.  Such a construction cannot be carried out directly for a parabolic equation, since the semiflow is not invertible. Our infinite-dimensional boundary--time scheme replaces the unavailable backward dynamics by exact finite-dimensional asymptotic factors, recovers the Sobolev geometry from the nested fibers of these factors, and uses the divergence-free transport structure to control the infinite tails of the resulting coordinate changes and their inverses. The boundary--time conjugacies constructed on the finite factors can then be assembled into a homeomorphism on a common neighborhood in the full phase space.

\vskip0.05in

The finite-dimensional factors underlying our construction are extracted
from the Foias--Saut asymptotic normalization
\cite{FoiasSaut1984a,FoiasSaut1984b,FoiasSaut1987}, further developed in
\cite{FoiasHoangOlsonZiane2009,FoiasHoangSaut2011,
FoiasHoangSaut2018,HoangMartinez}.  Here we regard its successive
asymptotic data as a nested family of exact dynamical quotients organized
by decay rate.  The basic idea is to retain long-time asymptotic
information rather than instantaneous low Fourier modes.  For a decay level $L$, two states are equivalent $u\sim_L v$ if 
\begin{align*}
 e^{\nu Lt}\|S_tu-S_tv\|_{H^s}\longrightarrow0, 
\end{align*}
that is,  their trajectories agree through
all asymptotic scales down to $e^{-\nu Lt}$ and may differ only in
faster-decaying components.  The relation is preserved by the semiflow,
since
\begin{align*}
 u\sim_L v
 \qquad\Longrightarrow\qquad
 S_\tau u\sim_L S_\tau v,
 \qquad \tau\ge0.
\end{align*}
Hence the quotient $X_L:=X/\!\sim_L$ (near zero) carries an exact
finite-dimensional factor dynamics induced by $S_t$, and the resulting
factor flows are compatible under the natural projections between
successive decay levels.  Geometrically, the
equivalence classes $[u]_L$ are strong-stable fibers, nested more finely as
$L$ increases. This exact factor property is what fails for an ordinary Fourier
truncation $P_N$ since in general, 
\begin{align*}
 P_{\leq N}u=P_{\leq N}v
 \qquad\not\Longrightarrow\qquad
 P_{\leq N}S_tu=P_{\leq N}S_tv,
\end{align*}
as  the future low modes still depend on the omitted high modes.
The Foias--Saut expansion provides concrete coordinates for these
factors through successive asymptotic coefficients
\begin{align*}
 u\longmapsto \bigl(W_1(u),W_2(u),\ldots\bigr), \quad \text{ so that } [u]_L=
 \{v:W_{\le L}(v)=W_{\le L}(u)\}, 
\end{align*}
where $W_j(u)$ records the $j$th layer of the long-time asymptotics of
$S_tu$ and plays the nonlinear role of spectral data at the corresponding
decay scale. As $L$ increases, progressively finer asymptotic information is retained,
and the full sequence $\bigl(W_1(u),W_2(u),\ldots\bigr)$ distinguishes
all initial states.  This complete asymptotic coding, however, is not yet a
coordinate system for the strong $H^s$ topology as  the high-level labels
may in principle compress or distort the Sobolev geometry  and
injectivity alone gives no control of reconstruction, of infinite tails,
or of the inverse map.

\vskip0.05in

To recover the Sobolev geometry, we choose the state of smallest \(H^s\) norm in each class by 
\begin{align*}
 \Pi_{\leq L}u
 :=
 \operatorname*{argmin}_{v\in [u]_L}
 \|v\|_{H^s}, 
\end{align*}
which is the nonlinear analogue of linear spectral truncation. 
For Navier--Stokes the fibers $[u]_L$ are curved so one needs to show
that these minimum-norm representatives remain unique and compatible
as $L$ increases. We resolve this issue by a curvature estimate admitting a useful dynamical explanation. Pairing the quadratic remainder in the difference
equation with an adjoint normal cancels the linear evolution, leaving
two forward-decaying factors against one backward-growing adjoint
factor.  This bounds the second fundamental form of the fibers.  On a
sufficiently small neighborhood, the squared $H^s$ norm is therefore
strictly convex along each fiber, giving a unique minimum-norm
representative without requiring large adjacent spectral gaps.  A
parameter-dependent Morse lemma near the unique minimum then 
introduces finite-dimensional blocks $z_j$ for which
\begin{align*}
 \|z_j\|^2
 =
 \|\Pi_{\leq L_j}u\|_{H^s}^2
 -
 \|\Pi_{\leq L_{j-1}}u\|_{H^s}^2, \quad L_j\uparrow\infty.
\end{align*}
Thus each block measures exactly the additional Sobolev cost of
resolving one more level of asymptotic information  and 
the squared block norms telescope to the $H^s$ norm as the asymptotic levels exhaust the state, 
\begin{align*}
 \sum_{k\le j}\|z_k\|^2
 =
 \|\Pi_{\leq L_j}u\|_{H^s}^2\to \|u\|_{H^s}^2,\quad j\to\infty. 
\end{align*}
A projection-error estimate of the form
\begin{align*}
 \|u-\Pi_{\leq L_j}u\|_{H^s}^2
 \lesssim \sum_{k>j}\|z_k\|^2
\end{align*}
then identifies small tails in the block variables with small
reconstruction errors in $H^s$, thereby recovering the strong Sobolev topology through a homeomorphic
nonlinear coordinate chart
\begin{align*}
 \Psi:u\longmapsto(z_1,z_2,\ldots)
\end{align*}
defined on a neighborhood in $H^s$.

At this stage the map $\Psi$ has recovered the Sobolev topology but has
not linearized the dynamics.  Its finite prefixes retain, however, the
exact factor structure developed above.  More precisely, for every
$j$, the block variables
\begin{align*}
 z_{\le j}:=(z_1,\ldots,z_j)
\end{align*}
contain exactly the same finite asymptotic information as the
Foias--Saut factor $W_{\le L_j}(u)$, but in coordinates adapted to the
Sobolev geometry; in particular, $z_{\le j}=K_j\bigl(W_{\le L_j}(u)\bigr)$ 
for a smooth finite-dimensional coordinate change $K_j$. 
Since these factors are nested and invariant under the semiflow, the
first $j$ blocks evolve autonomously.  In these coordinates the
Navier--Stokes dynamics therefore has the lower-triangular form
\begin{align*}
 \dot z_j=V_j(z_1,\ldots,z_j),\qquad j\ge1,
\end{align*}
where the differential equation is understood prefixwise that every finite
system is smooth, without requiring  differentiability of the full
infinite-dimensional map $\Psi$.

This triangular structure is what makes a factorwise boundary--time
construction possible.  Radial estimates inherited from the
Navier--Stokes dissipation show that, on a sufficiently small sphere
independent of the number of retained blocks, each finite-prefix vector
field points strictly inward.  Hence every nonzero point of a finite
factor can be followed backward to its unique intersection with the
boundary sphere. For the $j$th output block, we follow the first $j$ blocks
backward to their boundary sphere, retain the $j$th component
of the exit point, and evolve that component at the prescribed
reference rate for the same elapsed time. Different blocks
therefore generally use different exit times. This prefixwise
construction ensures that adding further blocks does not alter
the coordinates already constructed. This yields a compatible family of finite-dimensional conjugacies of
the Navier--Stokes factors to a common diagonal reference flow. Applying the construction also to the
corresponding finite-dimensional heat flow with this common target,
both dynamics are therefore represented at every finite level by the
same boundary--time coordinates.

The remaining issue is genuinely infinite-dimensional. Good
conjugacies at every finite level need not assemble into a map on the
full block coordinate space: their distortion may grow with the block index and
destroy square summability, either for the map or for its inverse.
Controlling this distortion is where the divergence-free transport
structure becomes essential.  After normalizing the variational
equation, we remove the nonresonant part of the skew transport and use
commutator estimates to show that the remaining near-resonant part has
only a small radial effect.  A new ingredient here is the resulting
quantitative spectral localization of the directions normal to the
factor fibers and of their fast-decaying tangent directions.
These estimates control the accumulation of errors across spectral
blocks and preserve square summability for both the conjugacies
and their inverses.

Consequently the finite boundary--time maps assemble into two
homeomorphisms in the strong Hilbert topology: one conjugates the
Navier--Stokes coordinate flow to the diagonal reference flow, and the
other conjugates the heat flow to the same reference flow. Their images
contain a common neighborhood of the origin. Composing the first map
with the inverse of the second, and returning through the Sobolev
coordinate map $\Psi$, yields the desired homeomorphism $\mathcal H$
satisfying \eqref{intro:goal}.

\section{Main result}\label{main:section}

We formulate the equation on the mean-zero solenoidal Sobolev space
and state the local conjugacy theorem.

Fix $\nu>0$ and $s>7/2$. On $\T^3=(\R/2\pi\mathbb Z)^3$, we consider
\begin{equation}\label{main:velocity-equation}
 \partial_tu+(u\cdot\nabla)u+\nabla p=\nu\Delta u,
 \qquad \nabla\cdot u=0,\qquad u(0)=u_0.
\end{equation}
We restrict the Navier--Stokes dynamics to the invariant subspace of zero-mean divergence-free velocities, thereby removing the constant neutral modes of the Stokes operator.
 For $q\in\R$, let
\begin{equation}\label{main:solenoidal-space}
 H^q_\sigma
 =
 \left\{
 u\in H^q(\T^3;\R^3):
 \nabla\cdot u=0,\quad
 \int_{\T^3}u(x)\,dx=0
 \right\}.
\end{equation}
We interpret the divergence condition and the zero-mean condition
in distributions; the latter means $\widehat u_0=0$.
Set
\begin{equation}\label{main:operators}
 X=H^s_\sigma,\qquad A=-\Delta,\qquad
 D_X(A)=H^{s+2}_\sigma,\qquad J_s=(I+A)^{s/2}.
\end{equation}
We have $D_X(A^{1/2})=H^{s+1}_\sigma$. The same Fourier multiplier
also defines a bounded map $A^{1/2}:X\to H^{s-1}_\sigma$.
Thus $A\ge I$ on $X$, and
$\sigma(A)=\{|k|^2:k\in\mathbb Z^3\setminus\{0\}\}\subset\mathbb N$.
We denote the complexification of $X$ by $X_{\mathbb C}$ and use
the same notation for complexified operators.

Let $\mathbb P$ be the Leray projection onto the zero-mean divergence-free vector fields. It commutes with $A$ and $J_s$ and is a contraction on every
Sobolev space. We use the inhomogeneous Sobolev norm
\begin{equation}\label{main:norm}
 \|u\|_q^2=\sum_{k\in\mathbb Z^3\setminus\{0\}}
               \langle k\rangle^{2q}|\widehat u_k|^2,
 \qquad \langle k\rangle=(1+|k|^2)^{1/2}.
\end{equation}
We write $\|\cdot\|=\|\cdot\|_s$ and denote the corresponding
inner product by $\langle\cdot,\cdot\rangle_s$. The notation
$B_r(Y)$ means the open ball of radius $r$ centered at zero in $Y$;
we omit $Y$ when it is $X$.
We use $C_s$ for positive constants depending only on $s$;
their values may increase from one estimate to the next.

Applying $\mathbb P$ to \eqref{main:velocity-equation}, we have
\begin{equation}\label{asy:equation}
 \dot u=-\nu Au-B(u,u),\qquad
 B(f,g)=\mathbb P((f\cdot\nabla)g),\qquad u(0)=u_0\in X.
\end{equation}
By the standard Sobolev product estimate,
 $B$ is a bounded bilinear map from
$X\times X\to H^{s-1}_\sigma$.
Conversely, every solution of \eqref{asy:equation} solves
\eqref{main:velocity-equation} with the unique zero-mean pressure recovered from $u$ by
\begin{equation}\label{main:pressure}
 -\Delta p=\sum_{i,j=1}^3\partial_i\partial_j(u_i u_j),
 \qquad \int_{\T^3}p\,dx=0.
\end{equation}  

Standard parabolic theory gives a local semiflow $S_tu_0=u(t)$ on
$X$; see \cite[Chapter~3]{Henry1981} and
\cite[Chapter~17]{TaylorPDE3}. In Section~\ref{asy:section}, we prove the
small-solution estimates used below. The linearization at zero is
$\dot u=-\nu Au$, with semigroup $e^{-\nu At}$.

\vskip0.1in
We state the main result.  
\nopagebreak[4]
\begin{theorem}\label{app:main}
There are open neighborhoods $U,V$ of zero in $X$ with the following
properties.
\begin{enumerate}
\renewcommand{\theenumi}{\roman{enumi}}
\renewcommand{\labelenumi}{(\theenumi)}
\item They satisfy $S_tU\subset U$ and $e^{-\nu At}V\subset V$ for
 every $t\ge0$.
\item There is a homeomorphism $\mathcal H:U\to V$ for the strong
 $H^s$ topology such that
\begin{equation}\label{app:conjugacy}
 \mathcal H(0)=0,\qquad
 \mathcal H(S_tu)=e^{-\nu At}\mathcal H(u),
 \qquad u\in U,\quad t\ge0.
\end{equation}
\end{enumerate}
\end{theorem}

Every global strong solution of the unforced equation with zero spatial
mean converges to zero in $H^s$ as $t\to\infty$; this is a classical
consequence of the long-time regularity and asymptotic theory
\cite{FoiasSaut1987}. Hence Theorem~\ref{app:main} describes the eventual
dynamics of every such solution. Indeed, for each global strong solution
$u(t)$ with $u(0)\in X$, there exists a time $T\ge0$, depending on the
solution, such that $u(T)\in U$. Positive invariance of $U$ and
\eqref{app:conjugacy} then give
\begin{equation}\label{app:eventual-conjugacy}
 u(t)=\mathcal H^{-1}\bigl(e^{-\nu A(t-T)}\mathcal H(u(T))\bigr),
 \qquad t\ge T.
\end{equation}
Thus the same local homeomorphism identifies the eventual dynamics of
all global strong solutions with the heat flow, preserving the original
time parameter.

\section{Asymptotic coordinates and nested fibers}
\label{asy:section}

In this section, we use the finite factors of the Foias--Saut
normalization and prove that they are submersions on a common
neighborhood and that their differentials separate tangent vectors.
Their level sets define nested invariant foliations, which provide
the geometric structure used in Section~\ref{geo:section} to construct
coordinates that put the Navier--Stokes semiflow into
lower-triangular form.

\subsection{Small solutions and the variational equation}

We establish small-data decay and variational estimates,
which provide the analytic foundation for the asymptotic
factors and the geometric constructions used in the conjugacy.

For real $u\in H^{s+1}_\sigma$, define
\begin{equation}\label{asy:linearized-operators}
 \mathcal N_u h=-B(u,h)-B(h,u),\qquad
 L_u=-\nu A+\mathcal N_u.
\end{equation}
Unless otherwise specified, $T^*$ denotes the adjoint of $T$ with respect to the $H^s$ inner product.
The normalized operator $J_s\mathcal N_uJ_s^{-1}$ acts from $H^1_\sigma$
to $L^2_\sigma$.


We collect standard Sobolev product and commutator estimates
in the form needed below; see \cite{RunstSickel} and
\cite[Chapter~13]{TaylorPDE3}.
\begin{lemma}\label{asy:products}
There exists $C_s>0$ such that:
\begin{enumerate}
\renewcommand{\theenumi}{\roman{enumi}}
\renewcommand{\labelenumi}{(\theenumi)}
\item For $f,g\in X$,
\begin{align}
 \|B(f,g)\|_{s-1}
 &\le C_s\|f\|_s\|g\|_s,
 \label{asy:bilinear}\\
 |\langle B(f,f),f\rangle_s|
 &\le C_s\|f\|_s^3
 \quad\text{for real smooth }f.
 \label{asy:energy-cancellation}
\end{align}
The map $f\mapsto\langle B(f,f),f\rangle_s$ has a unique
continuous extension to $X$.

\item For real $u\in H^{s+1}_\sigma$ and
$h\in H^{s+1}_\sigma$,
\begin{equation}\label{asy:adjoint-bound}
 \|\mathcal N_u h\|_s+\|\mathcal N_u^*h\|_s
 \le C_s\bigl(\|u\|_s\|A^{1/2}h\|_s
             +\|u\|_{s+1}\|h\|_s\bigr).
\end{equation}
Moreover,
\begin{equation}\label{asy:standardized}
 J_s\mathcal N_uJ_s^{-1}
 =-\mathbb P(u\cdot\nabla)\mathbb P+E_u,
 \qquad
 \|E_u\|_{\mathcal L(L^2_\sigma)}
 \le C_s\|u\|_{s+1}.
\end{equation}

\item For $q\in\{s-1,s\}$,
\begin{equation}\label{asy:linearized-loss}
 \|\mathcal N_u h\|_{q-1}
 \le C_s\|u\|_s\|h\|_q,
 \qquad u\in X,\quad h\in H^q_\sigma.
\end{equation}
\end{enumerate}
\end{lemma}

\begin{proof}
The standard product estimates give \eqref{asy:bilinear} and (iii).
Since $s>5/2$, using the commutator estimate, we have
\begin{equation}\label{asy:commutator-product}
 \|[J_s,u\cdot\nabla]h\|_0
 \le C_s\|u\|_s\|h\|_s.
\end{equation}
The transport cancellation then gives
\eqref{asy:energy-cancellation}; the continuous extension follows
by density from the corresponding bounded trilinear form.

For (ii), commuting $J_s$ through the transport term, we obtain
\[
 J_s[-B(u,\cdot)]J_s^{-1}
 =-\mathbb P(u\cdot\nabla)\mathbb P+R_u,
 \qquad \|R_u\|_{\mathcal L(L^2_\sigma)}\le C_s\|u\|_s.
\]
Together with $\|B(h,u)\|_s\le C_s\|h\|_s\|u\|_{s+1}$,
this proves \eqref{asy:standardized} and the bound for
$\mathcal N_u$ in \eqref{asy:adjoint-bound}. Taking the
$L^2$-adjoint, using the skew-symmetry of the transport term, and
conjugating back by $J_s$, we have the bound for $\mathcal N_u^*$.

\end{proof}

\Needspace{14\baselineskip}
We collect standard small-data estimates and analyticity
properties of the Navier--Stokes semiflow; see
\cite[Chapter~3]{Henry1981} and
\cite[Section~2.3 and Corollary~3.12]{FoiasHoangOlsonZiane2009}.
We record the bounds on fixed real and complex balls, with
explicit dependence on $\nu$.

\nopagebreak[4]
\begin{lemma}\label{asy:small-flow}
There exist constants $\rho_s,\epsilon_s,C_s>0$, depending
only on $s$, such that the following hold whenever
$0<r_s\le\rho_s\nu$.
\begin{enumerate}
\renewcommand{\theenumi}{\roman{enumi}}
\renewcommand{\labelenumi}{(\theenumi)}
\item The maps $S_t$ form a continuous semiflow on
$B_{4r_s}(X)$ and are real analytic for each $t\ge0$.
For $u_0\in B_{4r_s}(X)$ and $t\ge0$,
\begin{align}
 \|S_tu_0\|_s
 &\le C_s e^{-\nu t}\|u_0\|_s,
 \label{asy:decay}\\
 \int_t^\infty
 \bigl(\|S_qu_0\|_s^2+\|S_qu_0\|_{s+1}^2\bigr)\,dq
 &\le C_s\nu^{-1}\|S_tu_0\|_s^2,
 \label{asy:L2-coefficients}\\
 \int_t^\infty
 \bigl(\|S_qu_0\|_s+\|S_qu_0\|_{s+1}\bigr)\,dq
 &\le C_s\nu^{-1}\|S_tu_0\|_s.
 \label{asy:L1-coefficients}
\end{align}

\item The complexified maps
\[
 S_t:B_{5r_s}(X_{\mathbb C})
 \longrightarrow B_{5r_s}(X_{\mathbb C})
\]
are holomorphic for every $t\ge0$ and satisfy
\begin{equation}\label{asy:complex-decay}
 \|S_tu_0\|_s\le C_s e^{-\nu t}\|u_0\|_s,
 \qquad u_0\in B_{5r_s}(X_{\mathbb C}),\quad t\ge0.
\end{equation}
For every $m>s$, these maps are holomorphic into
$H^m_{\sigma,\mathbb C}$ for $t>0$, and
\[
 \sup_{t\ge\tau}\;
 \sup_{u_0\in B_{5r_s}(X_{\mathbb C})}
 \|S_tu_0\|_m<\infty
 \qquad\text{for every }\tau>0.
\]

\item The map
\[
 S_{1/\nu}:B_{\epsilon_s r_s}(H^1_{\sigma,\mathbb C})
 \longrightarrow B_{r_s}(H^s_{\sigma,\mathbb C})
\]
is holomorphic.
\end{enumerate}
\end{lemma}

\begin{proof}
By the analytic-semigroup estimate \cite{Pazy}, for $q\in\mathbb R$ and
$\theta\ge0$, we obtain
\begin{equation}\label{asy:heat-bound}
 \|e^{-\nu At}\|_{\mathcal L(H^{q-\theta}_\sigma,H^q_\sigma)}
 \le C\bigl(1+(\nu t)^{-\theta/2}\bigr)e^{-\nu t},
 \qquad t>0.
\end{equation}
Here $C$ depends only on $\theta$.
Since $D(A^{1/2})=H^s_\sigma$ on the base space $H^{s-1}_\sigma$
and the nonlinearity is a continuous quadratic polynomial by
\eqref{asy:bilinear}, local existence, uniqueness, continuation,
and analytic dependence follow from
\cite[Theorems~3.3.3--3.3.4 and Corollary~3.4.5]{Henry1981}.
The same local theory on the complexified spaces gives holomorphic
solution maps. We prove the quantitative bounds below.

We first prove (i). Using equation \eqref{asy:equation} and \eqref{asy:energy-cancellation}, we have
\[
 \frac12\frac d{dt}\|S_tu_0\|_s^2
 \le -\nu\|A^{1/2}S_tu_0\|_s^2+C_s\|S_tu_0\|_s^3.
\]
Choose $\rho_s>0$ so that $4C_s\rho_s\le1/2$.
While $\|S_tu_0\|_s\le4r_s$, the Poincar\'e inequality
absorbs the cubic term and gives
\begin{equation}\label{asy:energy-dissipation}
 \frac12\frac d{dt}\|S_tu_0\|_s^2
 \le-\frac{\nu}{2}\|A^{1/2}S_tu_0\|_s^2,
\end{equation}
and a continuity argument proves that $B_{4r_s}(X)$ is
positively invariant. The continuation theorem then gives
global existence on this ball.
The preceding inequality first gives
$\|S_tu_0\|_s\le e^{-\nu t/2}\|u_0\|_s$. For real data,
the sharper inequality
$(\log\|S_tu_0\|_s)'\le-\nu+C_s\|S_tu_0\|_s$
then gives \eqref{asy:decay} by integration. Integrating
\eqref{asy:energy-dissipation} from $t$ to $\infty$, we obtain
\eqref{asy:L2-coefficients}. 
Applying Cauchy--Schwarz and \eqref{asy:L2-coefficients} on
$[t+k/\nu,t+(k+1)/\nu]$, and \eqref{asy:decay} to $S_tu_0$,
we bound each integral of $\|S_qu_0\|_{s+1}$ by
$C_s\nu^{-1}e^{-k}\|S_tu_0\|_s$; summing over $k\ge0$ gives
\eqref{asy:L1-coefficients} and proves (i).

For (ii), duality, \eqref{asy:bilinear}, and the Poincar\'e
inequality give, also for complex $u$,
\[
 |\langle B(u,u),u\rangle_s|
 \le C_s\|u\|_s\|A^{1/2}u\|_s^2.
\]
After decreasing $\rho_s$, this gives invariance of
$B_{5r_s}(X_{\mathbb C})$ and decay at rate $3\nu/4$.
The mild equation, \eqref{asy:heat-bound}, and \eqref{asy:bilinear},
together with the preliminary decay estimate, give
\[
\begin{aligned}
 e^{\nu t}\|S_tu_0\|_s
 &\le \|u_0\|_s+C_s\|u_0\|_s^2
 \int_0^t\bigl[1+(\nu(t-q))^{-1/2}\bigr]e^{-\nu q/2}\,dq\\
 &\le \|u_0\|_s+C_s\nu^{-1}\|u_0\|_s^2.
\end{aligned}
\]
For the last inequality, rescale by $\nu$ and split the integral
at $t/2$ to bound it by $C/\nu$, uniformly in $t$.
Since $\|u_0\|_s<5r_s\le5\rho_s\nu$, this proves
\eqref{asy:complex-decay}; continuation gives global holomorphic
solution maps on this ball.

The standard positive-time smoothing estimates on the Sobolev scale
\cite[Section~3.5]{Henry1981} (see also \cite[Chapter~17]{TaylorPDE3}),
applied on the invariant, uniformly bounded $H^s$ ball, give
\[
 \sup_{v\in B_{5r_s}(X_{\mathbb C})}\|S_\tau v\|_m\le C,
 \qquad \tau>0,\quad m>s,
\]
where $C$ depends only on $\tau,m,s,\nu,r_s$.
These estimates also upgrade holomorphic dependence to $H^m$.
The identity $S_tu_0=S_\tau(S_{t-\tau}u_0)$ and ball invariance
give the asserted uniform bound for all $t\ge\tau$, proving (ii).

Finally, we prove (iii) for the complexified equation.
For $u_0\in H^1_{\sigma,\mathbb C}$, set
$\widetilde u(\theta)=\nu^{-1}S_{\theta/\nu}u_0$.
Then $\widetilde u$ solves the viscosity-one equation
with initial data $u_0/\nu$.
The Sobolev product estimate
$B:H^1_{\sigma,\mathbb C}\times H^1_{\sigma,\mathbb C}
\to H^{-3/4}_{\sigma,\mathbb C}$ \cite{RunstSickel}
allows us to apply the local theory and positive-time smoothing
cited above. On a sufficiently small $H^1$ ball, solutions exist
on $[0,1]$ and the time-one map is holomorphic into $H^s$.
Since it vanishes at zero, there are constants
$\delta_s,K_s>0$, depending only on $s$, such that
\[
 \|\widetilde u(1)\|_s
 \le K_s\|\widetilde u(0)\|_1
 \quad\text{whenever}\quad
 \|\widetilde u(0)\|_1\le\delta_s.
\]
Choose $\epsilon_s>0$ so that
$\epsilon_s\rho_s\le\delta_s$ and $K_s\epsilon_s\le1$.
If $\|u_0\|_1<\epsilon_s r_s$, then
$\|u_0/\nu\|_1\le\epsilon_s\rho_s$ and hence
\[
 \|S_{1/\nu}u_0\|_s
 =\nu\|\widetilde u(1)\|_s
 \le K_s\|u_0\|_1<K_s\epsilon_s r_s\le r_s.
\]
This proves (iii). 
\end{proof}

We use $D(t)$ to denote the time-dependent linear perturbation of
$-\nu A$; below, it represents the linearized term, its time-reversed
adjoint, or the linear term in the equation for the difference of
two solutions.

Let $0\le a<T\le\infty$, and  $\alpha,\beta$ be nonnegative
functions in $L^2_{\mathrm{loc}}((a,T))$. Suppose
$D(t):H^{s+1}_\sigma\to H^s_\sigma$ is linear and satisfies
\begin{equation}\label{asy:quotient-coefficients}
 \|D(t)h\|_s\le\alpha(t)\|A^{1/2}h\|_s+\beta(t)\|h\|_s.
\end{equation}
We consider a solution $v\in C([a,T);H^s_\sigma)$ of
\begin{equation}\label{asy:quotient-equation}
 v'=-\nu Av+D(t)v,
\end{equation}
which is smooth in time and space on $(a,T)$.

\vskip0.05in
We recall a standard consequence of the Dirichlet quotient
method: backward uniqueness and an exponential lower bound
for nonzero solutions.
\begin{lemma}\label{asy:nonvanishing}
Assume $\alpha,\beta\in L^2(a,T)$.
\begin{enumerate}
\renewcommand{\theenumi}{\roman{enumi}}
\renewcommand{\labelenumi}{(\theenumi)}
\item If $T<\infty$ and $v$ extends continuously to $T$, then
\[
 v(T)=0 \quad\Longrightarrow\quad
 v(t)=0 \quad\text{for all }t\in[a,T].
\]

\item If $T=\infty$, then for every $t_0>a$ with
$v(t_0)\ne0$, there exist constants $c,C>0$ such that
\begin{equation}\label{asy:exponential-lower-bound}
 \|v(t)\|_s\ge c e^{-C(t-t_0)},\qquad t\ge t_0.
\end{equation}
\end{enumerate}
\end{lemma}

\begin{proof}
We use the standard Dirichlet-quotient argument
\cite{AgmonNirenberg1967}. Wherever $v(t)\ne0$, set
\[
 q(t)=\frac{\|A^{1/2}v(t)\|_s^2}{\|v(t)\|_s^2}.
\]
The usual differentiation and Young's inequality, together with
\eqref{asy:quotient-coefficients}, give
\begin{equation}\label{asy:dirichlet}
 q'\le\frac{2}{\nu}\bigl(\alpha^2q+\beta^2\bigr),
\end{equation}
and
\begin{equation}\label{asy:log-norm}
 (\log\|v\|_s)'\ge-\nu q-\alpha\sqrt q-\beta.
\end{equation}
Fix $t_0\in(a,T)$ with $v(t_0)\ne0$.
Since $\alpha,\beta\in L^2(a,T)$, Gronwall's inequality applied to
\eqref{asy:dirichlet} gives $q\le K<\infty$ on every zero-free
interval starting at $t_0$, with $K$ independent of its right endpoint.
Integrating \eqref{asy:log-norm}, we obtain
\[
 \|v(t)\|_s\ge\|v(t_0)\|_s
 \exp\Big[-\nu K(t-t_0)
       -\sqrt K\int_{t_0}^t\alpha(r)\,dr
       -\int_{t_0}^t\beta(r)\,dr\Big].
\]
If $T<\infty$, Cauchy--Schwarz bounds both integrals uniformly.
The resulting positive lower bound excludes a first zero in
$(t_0,T]$, proving (i) by continuity.

If $T=\infty$, the same argument on finite intervals excludes every
zero after $t_0$. Cauchy--Schwarz and Young's inequality bound the
sum of the two integral terms by $(t-t_0)+C$,
with $C$ independent of $t$. The preceding lower bound then
gives \eqref{asy:exponential-lower-bound} with rate $\nu K+1$,
proving (ii).
\end{proof}

For $u_0\in B_{4r_s}$ and $t\ge a\ge0$, let
$U_{u_0}(t,a)=DS_{t-a}(S_au_0)$. Differentiating the mild equation gives
\begin{equation}\label{asy:variational-equation}
 \partial_tU_{u_0}(t,a)=L_{S_tu_0}U_{u_0}(t,a),\qquad U_{u_0}(a,a)=I,
\end{equation} where the operator $L_{u}$ is given by  \eqref{asy:linearized-operators}.

\vskip0.05in
We record standard smoothing, injectivity, and dense-range properties
of the variational evolution \cite[Sections~7.1 and~7.3]{Henry1981},
with the Sobolev estimate and endpoint conditions needed below.

\nopagebreak[4]
\begin{lemma}\label{asy:variational}
For $u_0\in B_{4r_s}$ and $t>a\ge0$, $U_{u_0}(t,a)$ extends to
a bounded operator from $H^{s-1}_\sigma$ to $X$, with
\begin{equation}\label{asy:smoothing}
 \|U_{u_0}(t,a)\|_{\mathcal L(H^{s-1}_\sigma,X)}
 \le C_s\bigl(\nu(t-a)\bigr)^{-1/2},
 \qquad 0<t-a\le\nu^{-1}.
\end{equation}
Its restriction to $X$ is injective and has dense range in $X$.
\end{lemma}

\begin{proof}
By \eqref{asy:linearized-loss}, we have
$\|\mathcal N_{S_tu_0}\|_{\mathcal L(H^\sigma,H^{\sigma-1})}
\le C_sr_s$ for $\sigma=s-1,s$.
After rescaling time by $\nu$, the perturbation is bounded by
$C_sr_s/\nu\le C_s\rho_s$.
The standard Volterra estimates \cite[Section~7.1]{Henry1981}
therefore give \eqref{asy:smoothing}, with $C_s$ independent of
$u_0,a,\nu$. They also give strong continuity and uniform boundedness
on finite time intervals in $H^{s-1}$ and $H^s$, including $a=0$.

The same evolution estimates \cite[Section~7.1]{Henry1981}
give operator-norm continuity in the initial time for $a<t$.
Their short-time estimate, in the dual $H^s$ pairing, gives
\[
 \|(U_{u_0}(t,t-\tau)^*-I)g\|_s
 \le C_s\bigl(\sqrt{\nu\tau}+r_s\tau\bigr)\|g\|_{s+1},
 \qquad g\in H^{s+1}_\sigma,
\]
for $0<\tau\le\min\{t-a,\nu^{-1}\}$.
Density and uniform boundedness therefore give
$\tau\mapsto U_{u_0}(t,t-\tau)^*g\in C([0,t-a];X)$ for every $g\in X$.

Injectivity follows from Lemma~\ref{asy:nonvanishing}(i),
using \eqref{asy:adjoint-bound}, \eqref{asy:L2-coefficients},
and continuity at the initial time.
For density of the range, we use the standard adjoint argument
\cite[Section~7.3]{Henry1981}.
For $g\in X$, set $z(\tau)=U_{u_0}(t,t-\tau)^*g$.
The adjoint evolution identity gives, on interior intervals,
\begin{equation}\label{asy:reverse-adjoint}
 z'=-\nu Az+\mathcal N_{S_{t-\tau}u_0}^*z,\qquad z(0)=g.
\end{equation}
By \eqref{asy:adjoint-bound}, the coefficients in
\eqref{asy:quotient-coefficients} can be taken as
$\alpha(\tau)=C_s\|S_{t-\tau}u_0\|_s$ and
$\beta(\tau)=C_s\|S_{t-\tau}u_0\|_{s+1}$.
Both belong to $L^2(0,t-a)$ by \eqref{asy:L2-coefficients},
also when $a=0$. Thus interior regularity, the endpoint continuity
above, and Lemma~\ref{asy:nonvanishing}(i) give
$U_{u_0}(t,a)^*g=z(t-a)=0\Rightarrow g=0$.
Consequently, $\ker U_{u_0}(t,a)^*=\{0\}$, proving the dense-range claim.
\end{proof}

\subsection{Asymptotic factors and remainder estimates}

We recall the Foias--Saut asymptotic factors and normal form
\cite[Theorems~4--6]{FoiasSaut1987}, using the periodic formulation
in \cite[Section~2.2]{FoiasHoangOlsonZiane2009}.
We prove uniform estimates for the remainders and their
initial-data derivatives on a common complex ball, providing
the analytic basis for the invariant foliations and
lower-triangular coordinates used in the conjugacy construction.

For $n\in\mathbb N$ and $L\in\mathbb R$, we set
\begin{equation}\label{asy:spectral-projections}
 P_n=\mathbf1_{\{A=n\}},\qquad
 P_{\le L}=\mathbf1_{\{A\le L\}},\qquad P_{>L}=I-P_{\le L}.
\end{equation}
Thus $P_n=0$ when $n\notin\sigma(A)$ and $P_{\le L}=0$ when $L<1$.

We use the normalization of Foias and Saut
\cite[Theorems~4--5]{FoiasSaut1987}. In the periodic setting,
their phase space is $H^1_\sigma$, and the set $\mathcal R$ of
initial data generating global regular solutions is open and
contains zero. Their normalization $W=(W_n)_{n\ge1}$ is analytic
and injective on $\mathcal R$, with values in the product of the
eigenspaces endowed with the componentwise topology. Since $X\hookrightarrow H^1_\sigma$,
we decrease $r_s$ if necessary so that $B_{4r_s}(X)\subset\mathcal R$.
We use its polynomial representation from
\cite[Section~2.2, (2.13)--(2.17)]{FoiasHoangOlsonZiane2009},
with viscosity $\nu$ restored by $u(t)=\nu v(\nu t)$, and set
$q_0\equiv0$.
For $n\ge1$ and $\xi_j\in P_jX_{\C}$, $1\le j\le n$,
we write $\xi_{\le n}=(\xi_1,\ldots,\xi_n)$ and define the
polynomial $q_n(t;\xi_{\le n})$ recursively by
\begin{align}
 q_n'+\nu(A-nI)q_n
       &=-\sum_{i=0}^{n}B(q_i,q_{n-i}),\label{asy:recursion}\\
 P_nq_n(0;\xi_{\le n})&=\xi_n.\label{asy:normalizing}
\end{align}
Since $q_0=0$, the right-hand side of \eqref{asy:recursion}
depends only on $q_1,\ldots,q_{n-1}$ and vanishes when $n=1$.
The existence and uniqueness of these polynomial solutions, their
finite Fourier support, and their polynomial dependence on
$\xi_{\le n}$ follow from
\cite[Section~2.2, (2.13)--(2.16)]{FoiasHoangOlsonZiane2009}.
For real initial data, these are the coefficients of the
Foias--Saut asymptotic expansion
\cite[Theorem~4]{FoiasSaut1987}.

We prepare a quantitative linear extraction lemma for the proof
of Proposition~\ref{asy:expansion}.

Fix $m\in\N$, $0<\theta<1$, a finite-dimensional $A$-invariant
Fourier subspace $E$, and $d\in\N\cup\{0\}$.
For a finite-dimensional subspace $F\subset X$ and
$j\in\N\cup\{0\}$, we write $\operatorname{Pol}_j(F)$ for the
space of polynomials $r(t)=\sum_{\ell=0}^j r_\ell t^\ell$ with
$r_\ell\in F$, equipped with the coefficient norm
\[
 \|r\|_{F,j}:=\sum_{\ell=0}^j\|r_\ell\|_s.
\]
Let $p\in\operatorname{Pol}_d(E)$.
We denote by $\mathcal G_{m,\theta}$ the space of continuous
functions $g:[0,\infty)\to H^{s-1}_\sigma$ with finite norm
\[
 \|g\|_{\mathcal G_{m,\theta}}
 :=\sup_{t\ge0}e^{\nu(m+\theta)t}\|g(t)\|_{s-1}.
\]
For a forcing remainder $g\in\mathcal G_{m,\theta}$ and an initial
value $v_0\in H^s_\sigma$, we set
\begin{equation}\label{asy:extraction-topology}
 \|(p,g,v_0)\|
 :=\|p\|_{E,d}+\|g\|_{\mathcal G_{m,\theta}}+\|v_0\|_s.
\end{equation}
We consider
\begin{equation}\label{asy:extraction-equation}
 v'+\nu Av=e^{-\nu mt}p(t)+g(t),\qquad v(0)=v_0.
\end{equation}
We call a triple $(p,g,v_0)$ admissible if the mild solution of
\eqref{asy:extraction-equation} satisfies
$\|v(t)\|_s\le Ce^{-\nu(m-1+\theta)t}$ when $m>1$, and is bounded
when $m=1$. 

By linearity of \eqref{asy:extraction-equation}, the admissible
triples form a linear subspace
\[
 Y\subset
 \operatorname{Pol}_d(E)\times\mathcal G_{m,\theta}\times H^s_\sigma,
\]
which we equip with the norm \eqref{asy:extraction-topology}.

\vskip0.05in
\nopagebreak[4]
The next lemma records the linear extraction step in the
Foias--Saut expansion \cite[Section~4]{FoiasSaut1987}, using
\cite[Lemma~4.2]{HoangMartinez2018}.
We verify the Sobolev estimates and their uniformity in $E$.
\begin{lemma}\label{asy:linear-extraction}
For each admissible triple, there is a unique polynomial
$q(t)=\sum_{j=0}^{d+1}q_jt^j$, with $q_j\in E+P_mX$,
such that
\[
 q'+\nu(A-m)q=p,\qquad
 \|v(t)-e^{-\nu mt}q(t)\|_s\le C\|(p,g,v_0)\|e^{-\nu(m+\theta)t}.
\]
The map
\[
 \mathcal T:Y\to\operatorname{Pol}_{d+1}(E+P_mX),
 \qquad (p,g,v_0)\mapsto q,
\]
is a bounded linear operator. Here $C$ depends only on $m,\theta,d,s,\nu$, and is independent
of $E$, $t$, and the admissible triple.
\end{lemma}

\begin{proof}
On each eigenspace $P_nX$, apply \cite[Lemma~4.2]{HoangMartinez2018}
to $e^{\nu mt}P_nv(t)$, with coefficient $\nu(n-m)$,
polynomial forcing $P_np$, and remainder $e^{\nu mt}P_ng$,
whose decay rate is $\nu\theta$ in $P_nX$.
For $n<m$, admissibility gives the required condition
$e^{\nu nt}P_nv(t)\to0$. The resulting polynomial is given by
\begin{equation}\label{asy:polynomial-solution}
 (I-P_m)q
 =\sum_{k=0}^d(-1)^k[\nu(A-mI)]^{-k-1}(I-P_m)p^{(k)},
\end{equation}
and
\begin{equation}\label{asy:resonant-polynomial}
 P_mq(t)=\xi+\int_0^t P_mp(r)\,dr,\qquad
 \xi=P_mv_0+\int_0^\infty e^{\nu mr}P_mg(r)\,dr.
\end{equation}
These formulas give
$q\in\operatorname{Pol}_{d+1}(E+P_mX)$ and
$q'+\nu(A-mI)q=p$.

We verify the stated norm bounds. Since
$|\lambda-m|\ge1$ for $\lambda\in\sigma(A)\setminus\{m\}$ and
$\|P_mh\|_s\le\sqrt{1+m}\|h\|_{s-1}$, we have
\[
 \|\xi\|_s\le\|v_0\|_s
       +\frac{\sqrt{1+m}}{\nu\theta}\|g\|_{\mathcal G_{m,\theta}}.
\]
Thus \eqref{asy:polynomial-solution}--\eqref{asy:resonant-polynomial}
define a linear map satisfying
\begin{equation}\label{asy:extraction-operator-bound}
 \|\mathcal T(p,g,v_0)\|_{E+P_mX,d+1}
 \le C\|(p,g,v_0)\|,
\end{equation}
with $C$ independent of $E$.

For $z=v-e^{-\nu mt}q$, the same linear extraction lemma gives
\[
 P_{\le m}z(t)
 =-\int_t^\infty e^{-\nu A(t-r)}P_{\le m}g(r)\,dr.
\]
We estimate this finite-dimensional integral in $H^s$.
On $P_{>m}X$, use the variation-of-constants formula and the
heat estimate \eqref{asy:heat-bound}, with spectral decay rate
$\nu(m+1)$. Since $0<\theta<1$, these bounds give
\[
 \|z(t)\|_s\le C e^{-\nu(m+\theta)t}
 \bigl(\|v_0\|_s+\|q(0)\|_s+\|g\|_{\mathcal G_{m,\theta}}\bigr).
\]
Together with \eqref{asy:extraction-operator-bound}, this proves
the remainder estimate. Uniqueness follows because an
exponentially decaying polynomial vanishes identically.
\end{proof}

For these factors, we write
$W_{\le n}=\sum_{j=1}^n W_j$ and identify this sum with the tuple
$(W_1,\ldots,W_n)$ in the spectral direct sum.

\vskip0.05in
\Needspace{6\baselineskip}
The next proposition recalls the classical asymptotic expansion and
normalization \cite[Theorems~4--5 and Section~6,
Proposition~2]{FoiasSaut1987}. We verify the uniform remainder
and derivative estimates on a common complex neighborhood
in the Sobolev norms used below.

\nopagebreak[4]
\begin{proposition}\label{asy:expansion}
The Foias--Saut factors admit bounded holomorphic extensions
\[
 W_n:B_{4r_s}(X_{\C})\to P_nX_{\C},\qquad n\ge1,
\]
on a common ball whose radius is independent of $n$.
They have the following properties.
\begin{enumerate}
\renewcommand{\theenumi}{\roman{enumi}}
\renewcommand{\labelenumi}{(\theenumi)}
\item For every $N\ge1$, $u_0\in B_{4r_s}(X_{\C})$, and $t\ge0$,
\begin{equation}\label{asy:expansion-eq}
 S_tu_0=\sum_{n=1}^N e^{-\nu nt}q_n(t;W_{\le n}(u_0))
       +\mathcal R_N(t,u_0).
\end{equation}

\item For every $N\ge1$, there exists $C_N>0$, depending only
on $N,s,\nu,r_s$, such that
\begin{equation}\label{asy:remainder-C1}
 \sup_{u_0\in B_{2r_s}(X_{\C})}
 \bigl(
 \|\mathcal R_N(t,u_0)\|_s+
 \|D_{u_0}\mathcal R_N(t,u_0)\|_{\mathcal L(X_{\C},X_{\C})}
 \bigr)
 \le C_Ne^{-\nu(N+1/4)t},
 \qquad t\ge0,
\end{equation}
and
\begin{equation}\label{asy:remainder-C1-smoothing}
 \sup_{u_0\in B_{2r_s}(X_{\C})}
 \bigl(
 \|\mathcal R_N(t,u_0)\|_{s+1}+
 \|D_{u_0}\mathcal R_N(t,u_0)\|_{\mathcal L(X_{\C},H^{s+1}_{\sigma,\C})}
 \bigr)
 \le C_Ne^{-\nu(N+1/4)t},
 \qquad t\ge\nu^{-1}.
\end{equation}

\item For every $n\ge1$,
\begin{equation}\label{asy:derivative-at-zero}
 W_n(0)=0,\qquad DW_n(0)=P_n.
\end{equation}
\end{enumerate}
\end{proposition}
\begin{proof}
For real initial data, the expansion and analytic factors follow from
\cite[Theorems~4--5]{FoiasSaut1987}; (iii) follows from the zero
solution and \cite[Section~6, Proposition~2]{FoiasSaut1987}.
We verify the extensions and uniform remainder estimates on the
fixed complex ball, then obtain the derivative bounds by Cauchy's
estimate. Uniqueness identifies the constructed factors with the
real Foias--Saut factors, so (iii) extends by complexification.

\emph{Step 1.} We apply the classical extraction scheme
\cite[Section~4]{FoiasSaut1987}, in the form of
Lemma~\ref{asy:linear-extraction} with $\theta=1/4$, to prove
bounded holomorphy of the factors on $B_{4r_s}(X_{\C})$ and
\begin{equation}\label{asy:remainder-uniform}
 \sup_{u_0\in B_{4r_s}(X_{\C})}\|\mathcal R_N(t,u_0)\|_s
 \le C_Ne^{-\nu(N+1/4)t},\qquad t\ge0.
\end{equation}
For $N=1$, \eqref{asy:bilinear} and \eqref{asy:complex-decay} give
\[
 \sup_{u_0\in B_{4r_s}(X_{\C})}
 \|B(S_tu_0,S_tu_0)\|_{s-1}\le Ce^{-2\nu t}.
\]
Lemma~\ref{asy:linear-extraction}, with $m=1$, $p=0$,
$g(t)=-B(S_tu_0,S_tu_0)$, and $v_0=u_0$, yields
$q_1=W_1\in P_1X_{\C}$ and \eqref{asy:remainder-uniform}.
The uniformly convergent integral in \eqref{asy:resonant-polynomial}
defines a bounded holomorphic map $W_1$.

Assume the assertion at order $N$. Write
$q_j(t)=q_j(t;W_{\le j}(u_0))$ and
$U_N(t,u_0)=\sum_{j=1}^N e^{-\nu jt}q_j(t)$.
Equations \eqref{asy:equation} and \eqref{asy:recursion} give
\begin{equation}\label{asy:remainder-equation}
 \mathcal R_N'+\nu A\mathcal R_N
 =-B(S_tu_0,\mathcal R_N)-B(\mathcal R_N,U_N)+F_N,
\end{equation}
where
\begin{equation}\label{asy:remainder-forcing}
 F_N(t,u_0)=-\sum_{\substack{1\le i,j\le N\\i+j>N}}
 e^{-\nu(i+j)t}B(q_i(t),q_j(t)).
\end{equation}
To extract the next coefficient, put
\begin{align}
 p_{N+1}(t,u_0)&=-\sum_{i=1}^{N}B(q_i(t),q_{N+1-i}(t)),\notag\\
 g_{N+1}(t,u_0)&=-B(S_tu_0,\mathcal R_N)-B(\mathcal R_N,U_N)
                 +F_N-e^{-\nu(N+1)t}p_{N+1}.
 \label{asy:next-extraction-data}
\end{align}
Then \eqref{asy:remainder-equation} has the form required by
Lemma~\ref{asy:linear-extraction} with $m=N+1$.
By \eqref{asy:complex-decay} and \eqref{asy:remainder-uniform},
$\|S_tu_0\|_s+\|U_N(t,u_0)\|_s\le C_Ne^{-\nu t}$.
The bilinear estimate controls the first two terms of $g_{N+1}$;
all remaining terms have total index at least $N+2$, so their
exponential decay absorbs the polynomial factors. Thus
\begin{equation}\label{asy:next-forcing-bound}
 \sup_{u_0\in B_{4r_s}(X_{\C})}\|g_{N+1}(t,u_0)\|_{s-1}
 \le C_Ne^{-\nu(N+1+1/4)t}.
\end{equation}
The induction hypothesis bounds the coefficients of $p_{N+1}$
in a fixed finite Fourier subspace and bounds $\mathcal R_N(0,u_0)$.
Estimate~\eqref{asy:remainder-uniform} also supplies the admissibility
condition for $m=N+1$. Lemma~\ref{asy:linear-extraction} therefore yields
\[
 \mathcal R_N=e^{-\nu(N+1)t}q_{N+1}+\mathcal R_{N+1},\qquad
 W_{N+1}(u_0)=P_{N+1}q_{N+1}(0),
\]
a uniform bound for $W_{N+1}$, and
\eqref{asy:remainder-uniform} at order $N+1$.
Uniqueness of the polynomial solutions
\cite[Section~2.2, (2.13)--(2.16)]{FoiasHoangOlsonZiane2009}
identifies this polynomial with
$q_{N+1}(t;W_{\le N+1}(u_0))$.
The data are holomorphic in $u_0$, and
\eqref{asy:next-forcing-bound} makes the integral in
\eqref{asy:resonant-polynomial} uniformly convergent, so $W_{N+1}$
is holomorphic. Only the constants depend on the induction order;
the complex ball remains unchanged. This proves (i) and
\eqref{asy:remainder-uniform} on that ball.

\emph{Step 2.} We apply standard semigroup smoothing
\cite[Section~3.5]{Henry1981} to \eqref{asy:remainder-equation},
keeping track of the decay rate $\beta_N=\nu(N+1/4)$.
The finite Fourier support and bounded coefficients of $q_1,\ldots,q_N$
give uniform bounds for $U_N$ in every fixed Sobolev norm.
For $\sigma=s,s+1/2$, the Sobolev multiplier estimate
\cite{RunstSickel,TaylorPDE3} and \eqref{asy:complex-decay} yield
\begin{equation}\label{asy:remainder-coefficients}
 \|B(S_tu_0,v)+B(v,U_N(t,u_0))\|_{\sigma-1}
 \le C_N\|v\|_\sigma.
\end{equation}
Since every term in \eqref{asy:remainder-forcing} has total
index at least $N+1$, we also have
\begin{equation}\label{asy:remainder-forcing-bound}
 \|F_N(t,u_0)\|_{\sigma-1}\le C_Ne^{-\beta_Nt},
 \qquad \sigma=s,s+1/2.
\end{equation}
Both estimates are uniform for $u_0\in B_{4r_s}(X_{\C})$.

For $t\ge\nu^{-1}$, put $\delta=(4\nu)^{-1}$.
Apply the variation-of-constants formula first on $[t-2\delta,t]$
with $\sigma=s$, and then on $[t-\delta,t]$ with $\sigma=s+1/2$.
By \eqref{asy:heat-bound}, each step gains half a derivative with
the integrable kernel $(t-r)^{-3/4}$. Starting from
\eqref{asy:remainder-uniform}, the first step bounds the
$H^{s+1/2}$ norm on $[t-\delta,t]$; the second gives
\begin{equation}\label{asy:remainder-one-bound}
 \sup_{u_0\in B_{4r_s}(X_{\C})}\|\mathcal R_N(t,u_0)\|_{s+1}
 \le C_Ne^{-\beta_Nt},\qquad t\ge\nu^{-1}.
\end{equation}
The decay rate is preserved because
$e^{-\beta_Nr}\le e^{2\beta_N\delta}e^{-\beta_Nt}$ on these
fixed-length intervals, and the extra factor is absorbed into $C_N$.

\emph{Step 3.} The remainder is holomorphic into $H^s$ and,
by Lemma~\ref{asy:small-flow}(ii), into $H^{s+1}$ for $t>0$.
Cauchy's estimate on the complex lines of radius $r_s$ centered
in $B_{2r_s}(X_{\C})$ gives
\[
 \sup_{u_0\in B_{2r_s}(X_{\C})}
 \|D_{u_0}\mathcal R_N(t,u_0)\|_{\mathcal L(X_{\C},H^q_{\sigma,\C})}
 \le r_s^{-1}\sup_{v_0\in B_{4r_s}(X_{\C})}
                    \|\mathcal R_N(t,v_0)\|_q,
 \qquad q=s,s+1.
\]
Using \eqref{asy:remainder-uniform} for $q=s$, $t\ge0$, and
\eqref{asy:remainder-one-bound} for $q=s+1$, $t\ge\nu^{-1}$,
we obtain (ii) after enlarging $C_N$.
\end{proof}

For $L\ge1$, we restrict $W_{\le L}=\sum_{n\le L}W_n$ to
$B_{2r_s}(X)$, with values in $P_{\le L}X$.
For $\xi\in P_{\le L}X$, write $\xi_j=P_j\xi$ and
$\xi_{\le n}=(\xi_1,\ldots,\xi_n)$, and set
$p_n(\xi)=q_n(0;\xi_{\le n})$ for $n\le L$.
Thus $p_n$ depends only on the first $n$ components of $\xi$
and satisfies $P_np_n(\xi)=\xi_n$.
We consider the triangular polynomial system
\begin{equation}\label{asy:normal-ODE}
 \dot\xi_n=-\nu n\xi_n
       -P_n\sum_{i+j=n}B(p_i(\xi),p_j(\xi)),\qquad n\le L.
\end{equation}

The next proposition follows from the Foias--Saut normalization
\cite[Theorems~5--6]{FoiasSaut1987}, its periodic formulation
\cite[Section~2.2]{FoiasHoangOlsonZiane2009}, and the compact
embedding $H^s_\sigma\hookrightarrow H^1_\sigma$.

\nopagebreak[4]
\begin{proposition}\label{asy:finite-flow}
The following statements hold.
\begin{enumerate}
\renewcommand{\theenumi}{\roman{enumi}}
\renewcommand{\labelenumi}{(\theenumi)}
\item For every $L\ge1$, system~\eqref{asy:normal-ODE}
generates a global flow
\[
 \Phi_L^t:P_{\le L}X\to P_{\le L}X,\qquad t\in\mathbb R,
\]
and
\begin{equation}\label{asy:covariance}
 W_{\le L}(S_tu)=\Phi_L^t(W_{\le L}(u)),
 \qquad u\in B_{2r_s}(X),\quad t\ge0.
\end{equation}

\item There exists $\delta_s>0$, independent of $L$, such that
each $W_{\le L}$ admits a real analytic extension
\[
 \widetilde W_{\le L}:
 B_{\delta_s}(H^1_\sigma)\to P_{\le L}X,
\]
with
\[
 \widetilde W_{\le L}(u)=W_{\le L}(u),
 \qquad u\in B_{\delta_s}(H^1_\sigma)\cap B_{2r_s}(X).
\]

\item For every $L\ge1$ and every sequence
$u_k,u\in B_{\delta_s}(H^1_\sigma)\cap B_{2r_s}(X)$,
\[
 u_k\rightharpoonup u\quad\text{in }H^s_\sigma
 \quad\Longrightarrow\quad
 \|W_{\le L}(u_k)-W_{\le L}(u)\|_s\to0.
\]
\end{enumerate}
\end{proposition}

\begin{proof}
For (i), \cite[Section~2.2, (2.17) and the formula following it]
{FoiasHoangOlsonZiane2009} gives, after restoring $\nu$,
\begin{equation}\label{asy:finite-flow-formula}
 (\Phi_L^t\xi)_n=e^{-\nu nt}P_nq_n(t;\xi_{\le n}),
 \qquad n\le L,\quad t\in\mathbb R.
\end{equation}
This formula defines solutions for all real times, and uniqueness
gives the flow property. The covariance
\eqref{asy:covariance} is the finite-dimensional projection of
\cite[Theorem~6]{FoiasSaut1987}.

For (ii), choose $\delta_s>0$ such that
$B_{\delta_s}(H^1_\sigma)\subset\mathcal R$.
The analytic normalization of
\cite[Theorem~5]{FoiasSaut1987} gives
\begin{equation}\label{asy:H1-extension}
 \widetilde W_{\le L}(u)=\sum_{n\le L}W_n(u),
 \qquad u\in B_{\delta_s}(H^1_\sigma).
\end{equation}
The domain is independent of $L$. Uniqueness of the asymptotic
coefficients gives agreement with $W_{\le L}$ on the common domain.

We prove (iii). If $u_k\rightharpoonup u$ in $H^s_\sigma$,
the compact embedding $H^s_\sigma\hookrightarrow H^1_\sigma$
gives $u_k\to u$ in $H^1_\sigma$; see \cite{RunstSickel}.
The continuity of $\widetilde W_{\le L}$ and its agreement
with $W_{\le L}$ therefore imply
$\|W_{\le L}(u_k)-W_{\le L}(u)\|_s\to0$.
\end{proof}

\subsection{Level sets of the asymptotic factors}

We prove that the finite asymptotic factors are submersions
on a common neighborhood and that their differentials separate
tangent vectors. Their level sets form nested invariant
foliations, whose tangent and normal spaces provide the basis
for the minimum-norm projections and lower-triangular coordinates
constructed in Section~\ref{geo:section}.

For every $L$, set
\begin{equation}\label{asy:flag-spaces}
 F_{>L}(u)=\ker DW_{\le L}(u),\qquad E_{\le L}(u)=F_{>L}(u)^\perp.
\end{equation}
For fixed $L$, we say that
$E_{\le L}(S_tu)$ converges to $P_{\le L}X$ in the $H^{s+2}$ graph topology if,
for all sufficiently large $t$, there is a linear map
$K_L(t):P_{\le L}X\to P_{>L}H^{s+2}_\sigma$ such that
\begin{equation}\label{asy:normal-graph-definition}
 E_{\le L}(S_tu)=\{h+K_L(t)h:h\in P_{\le L}X\},\qquad
 \|K_L(t)\|_{\mathcal L(P_{\le L}X,P_{>L}H^{s+2}_\sigma)}\to0\quad(t\to\infty).
\end{equation}

\vskip0.05in
\Needspace{6\baselineskip}
The next proposition collects consequences of the Foias--Saut
normalization \cite[Theorems~5--6]{FoiasSaut1987} and establishes
the geometric properties of its finite factors needed below.

\nopagebreak[4]
\begin{proposition}\label{asy:factor-flags}
On one common ball $B_r(\X)$, with $r>0$ independent of $L$, the following assertions hold.
\begin{enumerate}
\renewcommand{\theenumi}{\roman{enumi}}
\renewcommand{\labelenumi}{(\theenumi)}
\item Every $W_{\le L}$ is a submersion, that is,
$DW_{\le L}(u):X\to P_{\le L}X$ is onto for every $u\in B_r(X)$. The spaces in
\eqref{asy:flag-spaces} form analytic bundles of codimension and
dimension $\dim P_{\le L}X$, respectively.
\item The families $\{W_n\}_{n\ge1}$ and $\{DW_n(u)\}_{n\ge1}$
satisfy
\begin{equation}\label{asy:completeness}
\begin{aligned}
 {}[\,W_n(u)=W_n(\widetilde u)\ \text{for all }n\,]
 &\ \Longrightarrow\ u=\widetilde u,\\
 {}[\,DW_n(u)h=0\ \text{for all }n\,]
 &\ \Longrightarrow\ h=0,\qquad h\in X.
\end{aligned}
\end{equation}
Consequently,
\[
 \bigcap_{L\ge1}F_{>L}(u)=\{0\},
 \qquad \overline{\bigcup_{L\ge1}E_{\le L}(u)}=X.
\]
\item The bundles satisfy
\begin{equation}\label{asy:normal-covariance}
 DS_t(u)^*E_{\le L}(S_tu)=E_{\le L}(u),\qquad
 DS_t(u)F_{>L}(u)\subset F_{>L}(S_tu).
\end{equation}
For each fixed $L$, $E_{\le L}(u)\subset H^{2s-1}_\sigma$,
and a basis can be chosen locally analytically with values in $H^{2s-1}_\sigma$. In particular,
\begin{equation}\label{asy:normal-terminal}
 E_{\le L}(S_tu)\longrightarrow P_{\le L}X\quad\text{in the }H^{s+2}\text{ graph topology as }t\to\infty.
\end{equation}
\item If $u\ne \widetilde u$ and $W_{\le L}(u)=W_{\le L}(\widetilde u)$, there are an eigenvalue $n>L$ and $h\in P_nX\setminus\{0\}$ such that
\begin{equation}\label{asy:leaf-asymptotic}
 e^{\nu nt}(S_t\widetilde u-S_tu)\longrightarrow h\quad\text{in }H^{s+1}_\sigma\quad\text{as }t\to\infty.
\end{equation}
If $0\ne h_0\in F_{>L}(u)$, then for some eigenvalue $n>L$ and some $h\in P_nX\setminus\{0\}$,
\begin{equation}\label{asy:tangent-asymptotic}
 e^{\nu nt}DS_t(u)h_0\longrightarrow h\quad\text{in }H^{s+1}_\sigma\quad\text{as }t\to\infty.
\end{equation}
\end{enumerate}
\end{proposition}

\begin{proof}
Separation of points in (ii) follows from the injectivity of the
normalization \cite[Theorem~5]{FoiasSaut1987} and its identification
with our factors. We prove the remaining assertions using the
remainder estimates of Proposition~\ref{asy:expansion}, the covariance
in Proposition~\ref{asy:finite-flow}, and the variational evolution
of Lemma~\ref{asy:variational}.

We choose $0<r<\min\{2r_s,\delta_s\}$, with $\delta_s$ from
Proposition~\ref{asy:finite-flow}(ii). Then
$B_r(X)\subset B_{\delta_s}(H^1_\sigma)$, and
\eqref{asy:energy-dissipation} makes $B_r(X)$ positively invariant.
We use this ball for every $L$.

We first prove (i). Fix $L$.
By \eqref{asy:derivative-at-zero} and continuity,
$DW_{\le L}$ is onto on a neighborhood of zero that may depend
on $L$. For any $u\in B_r(X)$, use \eqref{asy:decay} to choose
$T>0$ so that $S_Tu$ belongs to this neighborhood.
Differentiating
\eqref{asy:covariance}, we obtain
\begin{equation}\label{asy:differentiated-covariance}
 DW_{\le L}(S_tu)DS_t(u)
 =D\Phi_L^t(W_{\le L}(u))DW_{\le L}(u),\qquad t\ge0.
\end{equation}
By Lemma~\ref{asy:variational}, $DS_T(u)$ has dense range.
The left-hand side of \eqref{asy:differentiated-covariance}
at $t=T$ therefore has dense range in the finite-dimensional
space $P_{\le L}X$, and hence is onto.
Since $D\Phi_L^T$ is invertible by
Proposition~\ref{asy:finite-flow}(i), we conclude that $DW_{\le L}(u)$
is onto. The time $T$ may depend on $L$ and $u$; the ball does not.

Let $\mathcal P_{\le L}(u):X\to E_{\le L}(u)$ denote the orthogonal
projection with respect to the $H^s$ inner product. Since
$DW_{\le L}(u)$ is onto, we have
\begin{equation}\label{asy:normal-projection}
 \mathcal P_{\le L}(u)=DW_{\le L}(u)^*
 [DW_{\le L}(u)DW_{\le L}(u)^*]^{-1}DW_{\le L}(u).
\end{equation}
The standard projection formula \cite[Chapter~5]{Brezis2011}
depends real analytically on $u$. Its range and kernel therefore
form analytic bundles, with
$\dim E_{\le L}(u)=\operatorname{codim}F_{>L}(u)=\dim P_{\le L}X$.
The analytic implicit-function theorem
\cite[Chapter~4]{Deimling1985} identifies the level sets as
analytic submanifolds with tangent spaces $F_{>L}(u)$.
This proves (i).

We now prove separation of tangent vectors in (ii).
If $DW_n(u)h=0$ for every $n$, differentiation of the expansion
gives $DS_t(u)h=D_u\mathcal R_N(t,u)h$ for every $N$.
For $h\ne0$, Lemma~\ref{asy:variational} and
Lemma~\ref{asy:nonvanishing}(ii) give an exponential lower bound
for this nonzero variational solution. Estimate
\eqref{asy:remainder-C1}, with $N$ sufficiently large, contradicts
that bound as $t\to\infty$. Thus $h=0$, completing
\eqref{asy:completeness}.

By \eqref{asy:flag-spaces} and \eqref{asy:completeness},
we have
\[
 \bigcap_{L\ge1}F_{>L}(u)
 =\bigcap_{n\ge1}\ker DW_n(u)=\{0\}.
\]
Since the spaces $E_{\le L}(u)$ increase, the Hilbert-space
orthogonal-complement identity \cite[Chapter~5]{Brezis2011} gives
$\overline{\bigcup_{L\ge1}E_{\le L}(u)}=X$.

For (iii), the covariance identities follow by taking kernels
and adjoint ranges in \eqref{asy:differentiated-covariance}.
Indeed, taking adjoints gives
\begin{equation}\label{asy:adjoint-factor-covariance}
 DS_t(u)^*DW_{\le L}(S_tu)^*
 =DW_{\le L}(u)^*D\Phi_L^t(W_{\le L}(u))^*.
\end{equation}
Since $E_{\le L}(v)=\operatorname{Ran}DW_{\le L}(v)^*$ and
$D\Phi_L^t$ is invertible, these two relations give
\eqref{asy:normal-covariance}.

We prove the asserted Sobolev regularity. Fix $L$ and an
$H^s$-orthonormal basis $e_1,\ldots,e_d$
of $P_{\le L}X$. By Proposition~\ref{asy:finite-flow}(ii)
and the Riesz representation theorem, there are
$g_j(u)\in H^1_\sigma$, depending real analytically on
$u$, such that
\[
 \langle DW_{\le L}(u)h,e_j\rangle_s
 =\langle h,g_j(u)\rangle_1
 =\langle h,(I+A)^{1-s}g_j(u)\rangle_s.
\]
Hence
\begin{equation}\label{asy:normal-riesz}
 DW_{\le L}(u)^*e_j=(I+A)^{1-s}g_j(u)\in H^{2s-1}_\sigma.
\end{equation}
Since $(I+A)^{1-s}:H^1_\sigma\to H^{2s-1}_\sigma$
is bounded and $DW_{\le L}(u)^*$ is injective, these vectors
form a basis of $E_{\le L}(u)$ depending real analytically
on $u$ in $H^{2s-1}_\sigma$.

Set $R_L(t)=DW_{\le L}(S_tu)^*$.
By \eqref{asy:normal-riesz}, \eqref{asy:decay}, and
\eqref{asy:derivative-at-zero},
$R_L(t)\to I|_{P_{\le L}X}$ in $\mathcal L(P_{\le L}X,H^{2s-1}_\sigma)$.
Since $P_{\le L}X$ is finite-dimensional and $2s-1>s+2$, the map
$P_{\le L}R_L(t)$ is invertible for large $t$, and
\begin{equation}\label{asy:normal-graph-map}
 K_L(t)=P_{>L}R_L(t)[P_{\le L}R_L(t)]^{-1}
 \longrightarrow0\quad\text{in }\mathcal L(P_{\le L}X,P_{>L}H^{s+2}_\sigma).
\end{equation}
Its graph is $\operatorname{Ran}R_L(t)=E_{\le L}(S_tu)$, proving
\eqref{asy:normal-terminal}.

Finally, we prove (iv) by extracting the first nonzero factor.
Spectral asymptotics for differences of solutions are classical;
see \cite[proof of Theorem~3]{FoiasSaut1987}.
Here we verify both limits in $H^{s+1}$ using
Proposition~\ref{asy:expansion}.
The polynomial recursion
\cite[Section~2.2, (2.13)--(2.16)]{FoiasHoangOlsonZiane2009} gives
\[
 q_n(t;\xi_{<n},\xi_n)
 =\xi_n+q_n(t;\xi_{<n},0).
\]
Suppose $u\ne \widetilde u$ and $W_{\le L}(u)=W_{\le L}(\widetilde u)$.
By \eqref{asy:completeness}, there is a first index
\[
 n=\min\{j:W_j(\widetilde u)\ne W_j(u)\}>L,\qquad
 h=W_n(\widetilde u)-W_n(u)\in P_nX\setminus\{0\}.
\]
The expansion coefficients below order $n$ coincide, and their
difference at order $n$ is $h$ by the preceding identity.
Subtracting \eqref{asy:expansion-eq} at order $n$ and applying
\eqref{asy:remainder-C1-smoothing}, we obtain
\[
 \|e^{\nu nt}(S_t\widetilde u-S_tu)-h\|_{s+1}
 \le 2C_ne^{-\nu t/4},\qquad t\ge\nu^{-1}.
\]
This proves \eqref{asy:leaf-asymptotic}.

Similarly, for $0\ne h_0\in F_{>L}(u)$, choose
\[
 n=\min\{j:DW_j(u)h_0\ne0\}>L,\qquad
 h=DW_n(u)h_0\in P_nX\setminus\{0\},
\]
using \eqref{asy:completeness}. The chain rule and the same
identity show that the derivatives of the expansion coefficients
vanish below order $n$ and equal $h$ at order $n$.
Differentiating \eqref{asy:expansion-eq} at that order and using
\eqref{asy:remainder-C1-smoothing}, we obtain
\[
 \|e^{\nu nt}DS_t(u)h_0-h\|_{s+1}
 \le C_ne^{-\nu t/4}\|h_0\|_s,\qquad t\ge\nu^{-1}.
\]
This proves \eqref{asy:tangent-asymptotic} and completes the proof.
\end{proof}

The fibers also have an intrinsic dynamical characterization:
\[
 W_{\le L}(u)=W_{\le L}(\widetilde u)
 \quad\Longleftrightarrow\quad
 e^{\nu Lt}\|S_tu-S_t\widetilde u\|_s\longrightarrow0.
\]
Indeed, for distinct $u,\widetilde u$, the first unequal factor has index $n$,
and the proof of (iv) gives a nonzero leading term at rate $\nu n$.
Thus the limit holds exactly when $n>L$.
The zero fibers $\{u\in B_r(X):W_{\le L}(u)=0\}$ are the local
$H^s$ restrictions of the classical nonlinear spectral manifolds
of Foias and Saut \cite{FoiasSaut1984a,FoiasSaut1984b}; see also
\cite[Corollary~4]{FoiasSaut1987}.

On the common positively invariant ball $B_r(X)$, write $u\sim_L v$
when $W_{\le L}(u)=W_{\le L}(v)$. By the preceding characterization,
this is the asymptotic equivalence relation used in the introduction.
Since $W_{\le L}$ is a submersion, it is an open map and
induces a homeomorphism
\[
 X_L:=B_r(X)/\!\sim_L\ \longrightarrow\ W_{\le L}(B_r(X)),
 \qquad [u]_L\longmapsto W_{\le L}(u),
\]
where $X_L$ carries the quotient topology. Covariance
\eqref{asy:covariance} makes $[u]_L\mapsto[S_tu]_L$ well defined
for $t\ge0$ and identifies this quotient dynamics with the restriction
of $\Phi_L^t$ to the open set $W_{\le L}(B_r(X))\subset P_{\le L}X$.

For use in Section~\ref{geo:section}, let $u,\widetilde u\in B_r(X)$
and set $w(t)=S_t\widetilde u-S_tu$. Subtracting their equations gives
\begin{equation}\label{asy:difference-equation}
 w'=-\nu Aw-B(S_tu,w)-B(w,S_t\widetilde u),
\end{equation}
and Lemma~\ref{asy:products}(ii) gives
\begin{equation}\label{asy:difference-coefficients}
 \|B(S_tu,w)+B(w,S_t\widetilde u)\|_s
 \le C_s\bigl(\|S_tu\|_s\|A^{1/2}w\|_s
                  +\|S_t\widetilde u\|_{s+1}\|w\|_s\bigr).
\end{equation}

\section{Geometry of the factor fibers and Sobolev coordinates}\label{geo:section}

In this section, we use minimum-norm projections and a
parameter-dependent Morse construction to obtain a homeomorphism
$\Psi$ in the strong $H^s$ topology. In these coordinates, the
Navier--Stokes semiflow has a lower-triangular form with autonomous
finite-dimensional prefixes. We also derive the radial equations
used in the subsequent blockwise construction of the conjugacy.

\vskip0.05in
Throughout this section, we work on the common ball supplied by
Section~\ref{asy:section}; we decrease its radius finitely many times, always
independently of the spectral cutoff.

\subsection{Geometry of the level sets}

We control the normal component of chords and the curvature
of the level sets, with constants independent of the cutoff $L$.
These estimates provide the geometric basis for the minimum-norm
projections used to construct lower-triangular coordinates.

For $v\in H^{s+1}_\sigma$ and $f\in H^{s-1}_\sigma$, we use the continuous
extension of the real $H^s$ inner product:
\begin{equation}\label{geo:extended-pairing}
 \langle v,f\rangle_s
 =\operatorname{Re}\sum_{m\ne0}\langle m\rangle^{2s}
       \overline{\widehat v_m}\cdot\widehat f_m,
 \qquad
 |\langle v,f\rangle_s|\le\|v\|_{s+1}\|f\|_{s-1}.
\end{equation}
Since the Leray projection is orthogonal in every Sobolev space,
\eqref{asy:bilinear} and \eqref{geo:extended-pairing} give the mixed estimate
\begin{equation}\label{geo:mixed-pairing}
 |\langle v,B(h,k)\rangle_s|
 \le C_s\|v\|_{s+1}\|h\|_s\|k\|_s,
 \qquad v\in H^{s+1}_\sigma,\quad h,k\in X.
\end{equation}
Thus the pairing remains meaningful even when $B(h,k)$ has only
$H^{s-1}$ regularity.
For $L\ge1$ with $P_{\le L}X\ne\{0\}$, we set
\begin{equation}\label{geo:spectral-endpoints}
 \Lambda_L=\max\sigma(A|_{P_{\le L}\X}),\qquad
 \mu_L=\min\sigma(A|_{P_{>L}\X}).
\end{equation}
Thus $1\le\Lambda_L<\mu_L$.
We use the orthogonal projection $\mathcal P_{\le L}(u):X\to E_{\le L}(u)$
from \eqref{asy:normal-projection} and define the complementary
orthogonal projection onto $F_{>L}(u)$ by
\begin{equation}\label{geo:orthogonal-projections}
 \mathcal P_{>L}(u)=I-\mathcal P_{\le L}(u):X\to F_{>L}(u).
\end{equation}
We use $L_u$ from \eqref{asy:linearized-operators} and set
\begin{equation}\label{geo:rate-constants}
 a=1-C_s(r/\nu)^2,\qquad b=1+C_s(r/\nu)^2,
 \qquad C_0=\exp(C_sr/\nu),
\end{equation} where $r>0$ is independent of $L$ and is chosen so that
$B_{4r}(X)$ lies in the common ball of
Proposition~\ref{asy:factor-flags}.
\vskip0.05in
\Needspace{6\baselineskip}
The next lemma gives decay estimates along the factor fibers,
growth estimates for normal covectors, and smoothing bounds,
uniformly in the spectral cutoff.

\nopagebreak[4]
\begin{lemma}\label{geo:variational} There exists $\delta_{\rm geo}>0$ such that, if $r/\nu\le\delta_{\rm geo}$, then
the following estimates  hold
uniformly for $u_0,\widetilde u_0\in B_{4r}(X)$ and $L$.
\begin{enumerate}
\renewcommand{\theenumi}{\roman{enumi}}
\renewcommand{\labelenumi}{(\theenumi)}
\item For $t\ge0$ and $h\in F_{>L}(u_0)$,
\begin{equation}\label{geo:fast-rate}
 \|DS_t(u_0)h\|_s
 \le C_0e^{-a\nu\mu_Lt}\|h\|_s.
\end{equation}
If $W_{\le L}(\widetilde u_0)=W_{\le L}(u_0)$, then
\begin{equation}\label{geo:leaf-rate}
 \|S_t\widetilde u_0-S_tu_0\|_s
 \le C_0e^{-a\nu\mu_Lt}\|\widetilde u_0-u_0\|_s,
 \qquad t\ge0.
\end{equation}
\item For every $v_0\in E_{\le L}(u_0)$, there exists
$v\in C([0,\infty);H^{s+2}_\sigma)\cap C^1((0,\infty);X)$
satisfying
\begin{equation}\label{geo:normal-equation}
 v'=-L_{S_tu_0}^*v,\qquad v(0)=v_0,\qquad
 v(t)\in E_{\le L}(S_tu_0),
\end{equation}
and, for $t\ge0$,
\begin{align}
 \|A^{1/2}v(t)\|_s^2
 &\le b\Lambda_L\|v(t)\|_s^2,\label{geo:normal-rayleigh}\\
 \|v(t)\|_s
 &\le C_0e^{b\nu\Lambda_Lt}\|v_0\|_s.\label{geo:normal-rate}
\end{align}
\item For $t>0$, the composition $DS_t(u_0)A^{1/2}$,
initially defined on $D_X(A^{1/2})$, has a bounded extension to $X$.
With this convention, we have
\begin{align}
 \|DS_t(u_0)A^{1/2}\|_{\mathcal L(X)}
 &\le C_s(\nu t)^{-1/2}e^{-\nu t/4},\label{geo:negative-smoothing}\\
 \|A^{1/2}\mathcal P_{>L}(u_0)DS_t(u_0)^*\|_{\mathcal L(X)}
 &\le C_s\bigl((\nu t)^{-1/2}+\sqrt{\Lambda_L}\bigr)
 e^{-a\nu\mu_Lt/2}.\label{geo:fast-adjoint-smoothing}
\end{align}
\end{enumerate}
\end{lemma}

\begin{proof}
By \eqref{asy:L2-coefficients} and \eqref{asy:L1-coefficients}, we obtain
\begin{equation}\label{geo:coefficient-integrals}
\begin{aligned}
 \int_t^\infty\bigl(\|S_qu_0\|_s^2+\|S_qu_0\|_{s+1}^2\bigr)\,dq
 &\le C_sr^2/\nu,\\
 \int_0^\infty\bigl(\|S_qu_0\|_s+\|S_qu_0\|_{s+1}\bigr)\,dq
 &\le C_sr/\nu.
\end{aligned}
\end{equation}
By \eqref{asy:adjoint-bound}, we have
\begin{equation}\label{geo:relative-operator}
 \|\mathcal N_{S_tu_0}h\|_s+\|\mathcal N_{S_tu_0}^*h\|_s
 \le C_s\bigl(\|S_tu_0\|_s\|A^{1/2}h\|_s
             +\|S_tu_0\|_{s+1}\|h\|_s\bigr).
\end{equation}

We first prove (i). We consider
$w(t)=DS_t(u_0)h_0$ with $0\ne h_0\in F_{>L}(u_0)$, or
$w(t)=S_t\widetilde u_0-S_tu_0$ with $u_0\ne \widetilde u_0$ and
$W_{\le L}(u_0)=W_{\le L}(\widetilde u_0)$.
By Lemma~\ref{asy:variational} in the first case and
Lemma~\ref{asy:nonvanishing}(i),
\eqref{asy:difference-coefficients}, and
\eqref{geo:coefficient-integrals} in the second, we have $w(t)\ne0$.
For $t>0$, we set
\[
 q(t)=\frac{\|A^{1/2}w(t)\|_s^2}{\|w(t)\|_s^2}.
\]
Applying \eqref{asy:dirichlet} with
\eqref{geo:relative-operator} gives, in the variational case,
\begin{equation}\label{geo:forward-q}
 q'\le\frac{C_s}{\nu}
 \bigl(\|S_tu_0\|_s^2q+\|S_tu_0\|_{s+1}^2\bigr).
\end{equation}
Similarly, using \eqref{asy:difference-equation} and
\eqref{asy:difference-coefficients}, we obtain
\begin{equation}\label{geo:difference-q}
 q'\le\frac{C_s}{\nu}\Bigl[
 (\|S_tu_0\|_s^2+\|S_t\widetilde u_0\|_s^2)q
 +\|S_tu_0\|_{s+1}^2+\|S_t\widetilde u_0\|_{s+1}^2\Bigr].
\end{equation}
By \eqref{asy:tangent-asymptotic} and
\eqref{asy:leaf-asymptotic}, respectively, the leading term of
$w(t)$ is a nonzero eigenvector with eigenvalue $n>L$,
with convergence in $H^{s+1}_\sigma$. Hence
$q(t)\to n\ge\mu_L$ as $t\to\infty$.
We apply Gronwall's inequality to \eqref{geo:forward-q} or
\eqref{geo:difference-q} and use \eqref{geo:coefficient-integrals}
for both initial points to obtain
\[
 q(T)\le e^\delta\bigl(q(t)+\delta\bigr),
 \qquad T\ge t>0,\qquad \delta=K_s(r/\nu)^2,
\]
where $K_s$ is independent of $L$. Letting $T\to\infty$,
we deduce
\begin{equation}\label{geo:fast-rayleigh-lower}
 q(t)\ge e^{-\delta}n-\delta
 \ge(1-2\delta)\mu_L\ge a\mu_L.
\end{equation}
Here we use $e^{-\delta}\ge1-\delta$, $n\ge\mu_L\ge1$,
and choose $C_s\ge2K_s$ in \eqref{geo:rate-constants}.

By the commutator and product estimates in the proof of
Lemma~\ref{asy:products}(ii), we obtain
\begin{equation}\label{geo:symmetric-energy}
 |\langle w,\mathcal N_{S_tu_0}w\rangle_s|
 \le C_s\bigl(\|S_tu_0\|_s+\|S_tu_0\|_{s+1}\bigr)\|w\|_s^2.
\end{equation}
For $w(t)=DS_t(u_0)h_0$, we use
\eqref{asy:variational-equation}, \eqref{geo:fast-rayleigh-lower},
and \eqref{geo:symmetric-energy} to obtain
\[
 \frac{d}{dt}\log\|w(t)\|_s
 =-\nu q(t)+
 \frac{\langle w,\mathcal N_{S_tu_0}w\rangle_s}{\|w\|_s^2}
 \le-a\nu\mu_L+C_s(\|S_tu_0\|_s+\|S_tu_0\|_{s+1}).
\]
We integrate on $[\varepsilon,t]$, use
\eqref{geo:coefficient-integrals}, and let $\varepsilon\downarrow0$
to obtain
\[
 \|DS_t(u_0)h_0\|_s
 \le e^{C_sr/\nu-a\nu\mu_Lt}\|h_0\|_s.
\]
By \eqref{geo:rate-constants}, this proves \eqref{geo:fast-rate}.

For $w(t)=S_t\widetilde u_0-S_tu_0$, the same symmetric-energy estimate
holds with the coefficient in \eqref{geo:symmetric-energy}
replaced by
$C_s\sum_{z\in\{u_0,\widetilde u_0\}}(\|S_tz\|_s+\|S_tz\|_{s+1})$.
Integrating with \eqref{geo:fast-rayleigh-lower} and
\eqref{geo:coefficient-integrals} proves \eqref{geo:leaf-rate}.
The zero initial data and $t=0$ follow by continuity.

We next prove (ii). By Proposition~\ref{asy:factor-flags}(i),
we choose the unique $z_0\in P_{\le L}X$ such that
$DW_{\le L}(u_0)^*z_0=v_0$. Using the invertibility in
Proposition~\ref{asy:finite-flow}(i), we define
\begin{equation}\label{geo:normal-continuation}
 v(t)=DW_{\le L}(S_tu_0)^*
 \bigl[D\Phi_L^t(W_{\le L}(u_0))^{-1}\bigr]^*z_0.
\end{equation}
By \eqref{asy:flag-spaces} and
Proposition~\ref{asy:factor-flags}(iii), we have
$v(0)=v_0$, $v(t)\in E_{\le L}(S_tu_0)$, and
$v\in C([0,\infty);H^{s+2}_\sigma)\cap C^1((0,\infty);X)$.
Using \eqref{asy:differentiated-covariance}, we obtain
\begin{equation}\label{geo:normal-pairing-invariance}
 DS_t(u_0)^*v(t)=v_0.
\end{equation}
We pair \eqref{geo:normal-pairing-invariance} with $h\in X$,
differentiate in $t$, and use \eqref{asy:variational-equation}:
\[
 0=\frac d{dt}\langle v(t),DS_t(u_0)h\rangle_s
 =\langle v'(t)+L_{S_tu_0}^*v(t),DS_t(u_0)h\rangle_s.
\]
By the dense range in Lemma~\ref{asy:variational}, we obtain
$v'=-L_{S_tu_0}^*v$, proving \eqref{geo:normal-equation}.

We may assume $v_0\ne0$; then
\eqref{geo:normal-pairing-invariance} gives $v(t)\ne0$.
We set $q(t)=\|A^{1/2}v(t)\|_s^2/\|v(t)\|_s^2$.
Since $v$ is continuous in $D_X(A)$ and differentiable in $X$,
self-adjointness gives, by a difference quotient,
\begin{equation}\label{geo:quadratic-derivative}
 \frac d{dt}\|A^{1/2}v(t)\|_s^2
 =2\langle Av(t),v'(t)\rangle_s,\qquad t>0.
\end{equation}
Using \eqref{geo:normal-equation},
\eqref{asy:linearized-operators}, and
\eqref{geo:quadratic-derivative}, we compute
\[
 q'=2\nu\frac{\|(A-qI)v\|_s^2}{\|v\|_s^2}
 -2\frac{\langle(A-qI)v,\mathcal N_{S_tu_0}^*v\rangle_s}
          {\|v\|_s^2}
 \ge-\frac{\|\mathcal N_{S_tu_0}^*v\|_s^2}{\nu\|v\|_s^2},
\]
where we apply Young's inequality. By
\eqref{geo:relative-operator}, we deduce
\begin{equation}\label{geo:backward-q}
 q'\ge-\frac{C_s}{\nu}
 \bigl(\|S_tu_0\|_s^2q+\|S_tu_0\|_{s+1}^2\bigr).
\end{equation}
By the graph convergence in \eqref{asy:normal-terminal} and
\eqref{geo:spectral-endpoints}, we have
\[
 \limsup_{t\to\infty}q(t)
 \le\limsup_{t\to\infty}\sup_{0\ne z\in E_{\le L}(S_tu_0)}
       \frac{\|A^{1/2}z\|_s^2}{\|z\|_s^2}
 =\Lambda_L.
\]
We integrate \eqref{geo:backward-q} on $[t,T]$ with an
integrating factor and use \eqref{geo:coefficient-integrals}:
\[
 q(t)\le e^\delta\bigl(q(T)+\delta\bigr),
 \qquad T\ge t>0,\qquad \delta=K_s(r/\nu)^2.
\]
Letting $T\to\infty$, we obtain
\[
 q(t)\le e^\delta(\Lambda_L+\delta)\le b\Lambda_L,
\]
where we use $\Lambda_L\ge1$, take $r/\nu$ sufficiently small,
and enlarge $C_s$ in \eqref{geo:rate-constants} if necessary.
This proves \eqref{geo:normal-rayleigh} for $t>0$.

Since $\langle v,\mathcal N_u^*v\rangle_s
=\langle v,\mathcal N_uv\rangle_s$, equations
\eqref{geo:normal-equation}, \eqref{geo:normal-rayleigh}, and
\eqref{geo:symmetric-energy} give
\[
 \frac d{dt}\log\|v(t)\|_s
 \le b\nu\Lambda_L+C_s(\|S_tu_0\|_s+\|S_tu_0\|_{s+1}).
\]
Integration using \eqref{geo:coefficient-integrals} and
continuity at zero yields \eqref{geo:normal-rate} with
$C_0$ from \eqref{geo:rate-constants}.
We extend \eqref{geo:normal-rayleigh} to $t=0$ by continuity
in $H^{s+2}_\sigma$. For $v_0=0$, we take $v\equiv0$.

Finally, we prove (iii) using the short-time smoothing estimate
\eqref{asy:smoothing} and an energy bound on the full evolution.
For $w(t)=DS_t(u_0)h$, \eqref{geo:symmetric-energy} and $A\ge I$ give
\[
 \frac12\frac d{dt}\|w(t)\|_s^2
 \le\bigl[-\nu+C_s(\|S_tu_0\|_s+\|S_tu_0\|_{s+1})\bigr]
       \|w(t)\|_s^2.
\]
Integrating with \eqref{geo:coefficient-integrals} gives
\begin{equation}\label{geo:full-derivative-decay}
 \|DS_t(u_0)\|_{\mathcal L(X)}
 \le e^{C_sr/\nu}e^{-\nu t}
 \le C_se^{-\nu t},\qquad t\ge0.
\end{equation}

Since $A^{1/2}:X\to H^{s-1}_\sigma$ is bounded,
Lemma~\ref{asy:variational} defines the bounded extension
$\mathcal K(t)=U_{u_0}(t,0)A^{1/2}$ of $DS_t(u_0)A^{1/2}$.
For $0<t\le\nu^{-1}$, \eqref{asy:smoothing} gives
\[
 \|\mathcal K(t)\|_{\mathcal L(X)}\le C_s(\nu t)^{-1/2}.
\]
For $t\ge\tau:=\nu^{-1}$, the evolution property gives
\[
 \mathcal K(t)=DS_{t-\tau}(S_\tau u_0)\mathcal K(\tau).
\]
This identity holds first on $D_X(A^{1/2})$ and then on $X$
by boundedness and density. Using
\eqref{geo:full-derivative-decay} at $S_\tau u_0$ and the
short-time bound, we obtain
\[
 \|\mathcal K(t)\|_{\mathcal L(X)}
 \le C_se^{-\nu(t-\tau)}
 \le C_s(\nu t)^{-1/2}e^{-\nu t/4},\qquad t\ge\tau.
\]
After enlarging $C_s$, the same estimate holds for $0<t\le\tau$.
This proves \eqref{geo:negative-smoothing}.

We now prove \eqref{geo:fast-adjoint-smoothing}.
By \eqref{geo:normal-rayleigh} at $t=0$ and self-adjointness,
we have
\[
 \|\mathcal P_{\le L}(u_0)A^{1/2}\|_{\mathcal L(X)}
 =\|A^{1/2}\mathcal P_{\le L}(u_0)\|_{\mathcal L(X)}
 \le\sqrt{b\Lambda_L}.
\]
For $h\in X$, we have $A^{1/2}h\in H^{s-1}_\sigma$ and define
\[
 \mathcal P_{>L}(u_0)A^{1/2}h
 :=A^{1/2}h-\mathcal P_{\le L}(u_0)A^{1/2}h
 \in H^{s-1}_\sigma.
\]
By Lemma~\ref{asy:variational}, for $t>0$ we have
\[
 DS_t(u_0)\mathcal P_{>L}(u_0)A^{1/2}
 =\mathcal K(t)-DS_t(u_0)\mathcal P_{\le L}(u_0)A^{1/2}
 \in\mathcal L(X).
\]
We apply \eqref{geo:negative-smoothing} and
\eqref{geo:full-derivative-decay} to obtain
\[
 \|DS_{t/2}(u_0)\mathcal P_{>L}(u_0)A^{1/2}\|_{\mathcal L(X)}
 \le C_s\bigl((\nu t)^{-1/2}+\sqrt{\Lambda_L}\bigr).
\]
By \eqref{asy:normal-covariance}, we have
\[
 DS_{t/2}(u_0)\mathcal P_{>L}(u_0)A^{1/2}h
 \in F_{>L}(S_{t/2}u_0),
 \qquad h\in D_X(A^{1/2}).
\]
Boundedness and density extend this inclusion to every $h\in X$,
since $F_{>L}(S_{t/2}u_0)$ is closed.
We write
$DS_t(u_0)=DS_{t/2}(S_{t/2}u_0)DS_{t/2}(u_0)$ and apply
\eqref{geo:fast-rate} to obtain
\begin{equation}\label{geo:fast-smoothing-factor}
\begin{aligned}
 &\|DS_t(u_0)\mathcal P_{>L}(u_0)A^{1/2}\|_{\mathcal L(X)}\\
 &\qquad\le C_0e^{-a\nu\mu_Lt/2}
 \|DS_{t/2}(u_0)\mathcal P_{>L}(u_0)A^{1/2}\|_{\mathcal L(X)}\\
 &\qquad\le C_s\bigl((\nu t)^{-1/2}+\sqrt{\Lambda_L}\bigr)
 e^{-a\nu\mu_Lt/2}.
\end{aligned}
\end{equation}
We take adjoints in $X$ to obtain
\eqref{geo:fast-adjoint-smoothing}.
\end{proof}

In the following lemma, we estimate the normal component of $\widetilde u-u$ for points $u,\widetilde u$
on the same factor fiber.
\nopagebreak[4]
\begin{lemma}\label{geo:chord}
Suppose $u,\widetilde u\in B_{4r}(\X)$ and $W_{\le L}(u)=W_{\le L}(\widetilde u)$.
Then
\begin{equation}\label{geo:chord-bound}
 \abs{\langle v,\widetilde u-u\rangle_{s}}
 \le \frac{C_s}{\nu\sqrt{\mu_L}}\|v\|_{s}\|\widetilde u-u\|_{s}^2,
 \qquad v\in E_{\le L}(u).
\end{equation}
\end{lemma}

\begin{proof}
We may assume $u\ne \widetilde u$ and $v\ne0$. We set $w(t)=S_t\widetilde u-S_tu$
and use Lemma~\ref{geo:variational}(ii) to choose $v(t)$ satisfying
\eqref{geo:normal-equation} with $v(0)=v$.
By \eqref{asy:difference-equation} and
\eqref{asy:linearized-operators}, we have
\[
 w'=L_{S_tu}w-B(w,w).
\]
We pair this equation with $v(t)$ and use
\eqref{geo:normal-equation} to obtain
\begin{equation}\label{geo:chord-pairing-derivative}
 \frac{d}{dt}\langle v(t),w(t)\rangle_s
 =-\langle v(t),B(w(t),w(t))\rangle_s.
\end{equation}

Fix $0<\epsilon<\mu_L-\Lambda_L$.
By \eqref{asy:normal-terminal}, the normal Rayleigh quotient
is at most $\Lambda_L+\epsilon$ for all sufficiently large times.
Integrating the normal energy equation
\eqref{geo:normal-equation} using \eqref{geo:symmetric-energy}
and \eqref{geo:coefficient-integrals}, and enlarging the constant
to cover the remaining finite interval, gives the first bound below.
The second follows directly from \eqref{asy:leaf-asymptotic},
since its leading eigenvalue is at least $\mu_L$:
\begin{equation}\label{geo:fixed-cut-terminal-bounds}
 \|v(t)\|_s\le C_{L,\epsilon}e^{\nu(\Lambda_L+\epsilon)t},
 \qquad
 \|w(t)\|_s\le C_L e^{-\nu\mu_Lt},\qquad t\ge0.
\end{equation}
Here the constants may depend on $u,\widetilde u,v$ and the fixed cutoff $L$.
Since $\mu_L-\Lambda_L-\epsilon>0$, these bounds imply
$\langle v(t),w(t)\rangle_s\to0$ as $t\to\infty$.
We integrate \eqref{geo:chord-pairing-derivative} and use
$v(0)=v$ and $w(0)=\widetilde u-u$ to obtain
\begin{equation}\label{geo:chord-integral}
 \langle v,\widetilde u-u\rangle_s
 =\int_0^\infty
 \langle v(t),B(w(t),w(t))\rangle_s\,dt.
\end{equation}

By \eqref{geo:mixed-pairing}, \eqref{geo:normal-rayleigh},
and $A\ge I$, we have
\[
 |\langle v(t),B(w(t),w(t))\rangle_s|
 \le C_s\sqrt{\Lambda_L}\,\|v(t)\|_s\|w(t)\|_s^2.
\]
We decrease $r/\nu$ so that $2a-b\ge1/2$. We then apply
\eqref{geo:leaf-rate} and \eqref{geo:normal-rate} in
\eqref{geo:chord-integral} to obtain
\[
\begin{aligned}
 |\langle v,\widetilde u-u\rangle_s|
 &\le C_s\sqrt{\Lambda_L}\,\|v\|_s\|\widetilde u-u\|_s^2
       \int_0^\infty e^{-\nu(2a\mu_L-b\Lambda_L)t}\,dt\\
 &=\frac{C_s\sqrt{\Lambda_L}}{\nu(2a\mu_L-b\Lambda_L)}
       \|v\|_s\|\widetilde u-u\|_s^2\\
 &\le\frac{C_s}{\nu\sqrt{\mu_L}}\|v\|_s\|\widetilde u-u\|_s^2,
\end{aligned}
\]
where we use $\Lambda_L<\mu_L$ and
$2a\mu_L-b\Lambda_L\ge(2a-b)\mu_L\ge\mu_L/2$.
This proves \eqref{geo:chord-bound}.
\end{proof}

By Proposition~\ref{asy:factor-flags}, we write the $W_{\le L}$
level through $u$ locally as
\[
 \{u+h+g(h):h\in U\subset F_{>L}(u)\},
\]
where $U$ is a neighborhood of $0$ and $g:U\to E_{\le L}(u)$ is
real analytic with $g(0)=0$ and $Dg(0)=0$.
We define its second fundamental form at $u$ by
\[
 \mathfrak I_L(u)(h,k)=D^2g(0)[h,k]\in E_{\le L}(u),
 \qquad h,k\in F_{>L}(u).
\]
Equivalently, for any $C^2$ curve $\gamma$ in this level with
$\gamma(0)=u$ and $\gamma'(0)=h$, we have
$\mathfrak I_L(u)(h,h)=\mathcal P_{\le L}(u)\gamma''(0)$.
For $v\in E_{\le L}(u)$, we define the shape operator
$\mathfrak S_v:F_{>L}(u)\to F_{>L}(u)$ by
\begin{equation}\label{geo:shape-convention}
 \langle \mathfrak S_vh,k\rangle_s
 =\langle v,\mathfrak I_L(u)(h,k)\rangle_s,
 \qquad h,k\in F_{>L}(u).
\end{equation}
By the Riesz representation theorem and symmetry of
$\mathfrak I_L(u)$, this defines a unique bounded self-adjoint operator.

\vskip0.05in
The next lemma provides uniform curvature estimates for the
factor fibers in the Sobolev and energy norms.
\nopagebreak[4]
\begin{lemma}\label{geo:shape}
For $u\in B_{4r}(\X)$, $h,k\in F_{>L}(u)$, and $v\in E_{\le L}(u)$,
\begin{align}
 \|\mathfrak I_L(u)(h,k)\|_{s}
 &\le \frac{C_s}{\nu\sqrt{\mu_L}}\|h\|_{s}\|k\|_{s},
                                                        \label{geo:II-bound}\\
 \|A^{1/2}\mathfrak S_vh\|_{s}
 &\le \frac{C_s}{\nu}\|v\|_{s}\|h\|_{s}.
                                                        \label{geo:energy-shape}
\end{align}
\end{lemma}

\begin{proof}
\emph{Step 1.} We first prove \eqref{geo:II-bound}.
Apply \eqref{geo:chord-bound} to
$\widetilde u=u+\tau h+g(\tau h)$. Since
$g(\tau h)=\frac12\tau^2\mathfrak I_L(u)(h,h)+o(\tau^2)$
and $v\perp F_{>L}(u)$, division by $\tau^2$ and passage to the
limit give
\[
 |\langle v,\mathfrak I_L(u)(h,h)\rangle_s|
 \le\frac{C_s}{\nu\sqrt{\mu_L}}\|v\|_s\|h\|_s^2.
\]
Polarization and duality in $E_{\le L}(u)$ yield \eqref{geo:II-bound}.

\emph{Step 2.} We derive an integral representation of
$\mathfrak S_v$. We first take
$h,k\in F_{>L}(u)\cap H^{s+2}_\sigma$, set
$h(t)=DS_t(u)h$ and $k(t)=DS_t(u)k$, and choose $v(t)$
by Lemma~\ref{geo:variational}(ii), with $v(0)=v$.
Using the local graph $g$ defined above, we set
\[
 \Gamma(\alpha,\beta)
 =u+\alpha h+\beta k+g(\alpha h+\beta k)
\]
for sufficiently small $\alpha,\beta\in\mathbb R$, and define
\[
 U(t)=\left.\partial_\alpha\partial_\beta
 S_t\Gamma(\alpha,\beta)\right|_{(0,0)}.
\]
Then $U(0)=\mathfrak I_L(u)(h,k)$.
By differentiating \eqref{asy:equation} twice and using
\eqref{asy:linearized-operators}, we obtain
\[
 U'=L_{S_tu}U-B(h(t),k(t))-B(k(t),h(t)).
\]
Using \eqref{geo:normal-equation}, we obtain
\[
 \frac d{dt}\langle v(t),U(t)\rangle_s
 =-\langle v(t),B(h(t),k(t))+B(k(t),h(t))\rangle_s.
\]
By \eqref{asy:covariance}, $S_t\Gamma(\alpha,\beta)$ lies in the
factor fiber through $S_tu$. Hence
\[
 \langle v(t),U(t)\rangle_s
 =\langle v(t),\mathfrak I_L(S_tu)(h(t),k(t))\rangle_s.
\]
Using \eqref{geo:II-bound}, \eqref{geo:fast-rate}, and
\eqref{geo:normal-rate}, we obtain
\[
 |\langle v(t),U(t)\rangle_s|
 \le\frac{C_s}{\nu\sqrt{\mu_L}}
 e^{-\nu(2a\mu_L-b\Lambda_L)t}\|v\|_s\|h\|_s\|k\|_s
 \to0\quad\text{as }t\to\infty,
\]
since $2a>b$ and $\Lambda_L<\mu_L$.
We integrate and use \eqref{geo:shape-convention} to obtain
\begin{equation}\label{geo:shape-integral}
 \langle\mathfrak S_vh,k\rangle_s
 =\int_0^\infty
 \langle v(t),B(h(t),k(t))+B(k(t),h(t))\rangle_s\,dt.
\end{equation}

\emph{Step 3.} We prove \eqref{geo:energy-shape}.
By \eqref{geo:mixed-pairing} and the Riesz representation theorem,
we define $g(t)\in X$ by
\begin{equation}\label{geo:shape-riesz}
 \langle g(t),z\rangle_s
 =\langle v(t),B(h(t),z)+B(z,h(t))\rangle_s,
 \qquad z\in X.
\end{equation}
Using \eqref{geo:fast-rate},
\eqref{geo:normal-rayleigh}--\eqref{geo:normal-rate}, and $A\ge I$,
we obtain
\[
 \|g(t)\|_s
 \le C_s\sqrt{\Lambda_L}
 e^{-\nu(a\mu_L-b\Lambda_L)t}\|h\|_s\|v\|_s.
\]
We decrease $r/\nu$ so that $3a/2-b\ge1/4$.
By \eqref{geo:fast-adjoint-smoothing} and $\Lambda_L<\mu_L$,
we have
\[
\begin{aligned}
 &\int_0^\infty
 \|A^{1/2}\mathcal P_{>L}(u)DS_t(u)^*g(t)\|_s\,dt\\
 &\quad\le C_s\sqrt{\Lambda_L}\|h\|_s\|v\|_s
 \int_0^\infty\bigl((\nu t)^{-1/2}+\sqrt{\Lambda_L}\bigr)
 e^{-\nu\mu_Lt/4}\,dt\\
 &\quad\le\frac{C_s}{\nu}
 \biggl(\sqrt{\frac{\Lambda_L}{\mu_L}}+
 \frac{\Lambda_L}{\mu_L}\biggr)\|h\|_s\|v\|_s
 \le\frac{C_s}{\nu}\|h\|_s\|v\|_s.
\end{aligned}
\]
Since $A\ge I$, the integral below converges in the graph norm
of $A^{1/2}$ on $X$. By \eqref{geo:shape-integral} and
\eqref{geo:shape-riesz}, we obtain
\begin{equation}\label{geo:shape-operator-integral}
 \mathfrak S_vh
 =\int_0^\infty\mathcal P_{>L}(u)DS_t(u)^*g(t)\,dt.
\end{equation}
The preceding estimate proves \eqref{geo:energy-shape}
for smooth $h$.
By Proposition~\ref{asy:factor-flags}(iii),
$\mathcal P_{>L}(u)$ preserves $H^{s+2}_\sigma$.
Projecting smooth approximations by $\mathcal P_{>L}(u)$ therefore
shows that $F_{>L}(u)\cap H^{s+2}_\sigma$ is dense in $F_{>L}(u)$.
This justifies the operator identity above from its pairings
with smooth tangent vectors. The bound \eqref{geo:II-bound}
and closedness of $A^{1/2}$ then extend
\eqref{geo:energy-shape} to every $h\in F_{>L}(u)$.
\end{proof}

\subsection{Minimum-norm points on factor fibers}

We select the unique minimum-norm point on each factor fiber and
prove that these selections define compatible analytic maps
converging to the identity in $H^s$. We then compare these representatives
with the normal projections of the original points to obtain the norm
and energy estimates needed for the coordinate construction and radial dynamics.

For
$u\in B_{2r}$, consider the minimum problem
\begin{equation}\label{geo:min-definition}
 \Pi_{\le L}u=\operatorname*{arg\,min}
 \{\|\widetilde u\|_{s}:\widetilde u\in\overline B_{3r},\ W_{\le L}(\widetilde u)=W_{\le L}(u)\}.
\end{equation}
We choose $r$ so that
$\epsilon_L=C_sr/(\nu\sqrt{\mu_L})<1/2$ for every $L$.

\vskip0.05in

\Needspace{6\baselineskip}
The next proposition shows that the minimum-norm representatives
define nested analytic maps converging to the identity in $H^s$.
\nopagebreak[4]
\begin{proposition}\label{geo:minima}
The following assertions hold.
\begin{enumerate}
\renewcommand{\theenumi}{\roman{enumi}}
\renewcommand{\labelenumi}{(\theenumi)}
\item Problem~\eqref{geo:min-definition} has a unique solution for every
$u\in B_{2r}$.
The map $\Pi_{\le L}$ is analytic. If $v=\Pi_{\le L}u$, then
\begin{equation}\label{geo:pythagoras}
 (1-\epsilon_L)\|u-v\|_{s}^2
 \le\|u\|_{s}^2-\|v\|_{s}^2
 \le(1+\epsilon_L)\|u-v\|_{s}^2.
\end{equation}
\item For $L\le M$ with $P_{\le L}X\ne\{0\}$, we have
\begin{equation}\label{geo:nested}
 \Pi_{\le L}\Pi_{\le M}=\Pi_{\le M}\Pi_{\le L}=\Pi_{\le L}.
\end{equation}
\item As $L\to\infty$, we have
\begin{equation}\label{geo:minimum-limit}
 \Pi_{\le L}u\to u\quad\text{in }H^s,
\end{equation}
uniformly for $u$ in every compact subset of $B_{2r}(X)$
in the $H^s$ topology.
\end{enumerate}
\end{proposition}

\begin{proof}
(i) We fix $u\in B_{2r}$. The constraint is weakly closed in
$\overline B_{3r}$ by Proposition~\ref{asy:finite-flow}(iii).
The direct method in a Hilbert space \cite[Chapters~3 and~5]{Brezis2011} therefore
gives a minimizer $v$ with $\|v\|_s\le\|u\|_s<2r$.

We next prove uniqueness and \eqref{geo:pythagoras}.
For $h\in F_{>L}(v)$, we use Proposition~\ref{asy:factor-flags}(i)
to choose a $C^1$ curve $\gamma$ in the fiber with
$\gamma(0)=v$ and $\gamma'(0)=h$.
Since $v$ is an interior minimizer, we obtain
\[
 0=\left.\frac{d}{d\tau}\frac12\|\gamma(\tau)\|_s^2
 \right|_{\tau=0}=\langle v,h\rangle_s.
\]
Thus
\begin{equation}\label{geo:minimum-stationarity}
 \langle v,h\rangle_s=0\quad(h\in F_{>L}(v)),
 \qquad v\in F_{>L}(v)^\perp=E_{\le L}(v).
\end{equation}
For every admissible $\widetilde u$, we apply \eqref{geo:chord-bound}
to the normal vector $v$ and use $\|v\|_s<2r$ to obtain
\[
 2|\langle v,\widetilde u-v\rangle_s|
 \le\frac{2C_s\|v\|_s}{\nu\sqrt{\mu_L}}\|\widetilde u-v\|_s^2
 \le\epsilon_L\|\widetilde u-v\|_s^2,
\]
where we choose $C_s$ in $\epsilon_L$ large enough and then
$r$ small enough that $\epsilon_L<1/2$.
Using
\[
 \|\widetilde u\|_s^2-\|v\|_s^2
 =\|\widetilde u-v\|_s^2+2\langle v,\widetilde u-v\rangle_s,
\]
we obtain
\begin{equation}\label{geo:minimum-coercivity}
 (1-\epsilon_L)\|\widetilde u-v\|_s^2
 \le\|\widetilde u\|_s^2-\|v\|_s^2
 \le(1+\epsilon_L)\|\widetilde u-v\|_s^2.
\end{equation}
We take $\widetilde u=u$ to prove \eqref{geo:pythagoras}.
If $\widetilde u$ is another minimizer, we use
\eqref{geo:minimum-coercivity} to obtain
$(1-\epsilon_L)\|\widetilde u-v\|_s^2\le0$, hence $\widetilde u=v$.
The preceding estimates use only $v\in E_{\le L}(v)$ and
$\|v\|_s<2r$. Hence every point satisfying these conditions
is the unique minimizer on its fiber.

We finally prove analyticity of $\Pi_{\le L}$. Fix $u_0\in B_{2r}$,
set $v=\Pi_{\le L}u_0$ and $\xi_0=W_{\le L}(v)$, and use the
analytic implicit-function theorem \cite[Chapter~4]{Deimling1985}
to write nearby fibers as
\[
 v+\zeta+g(\zeta,\xi),\qquad
 \zeta\in F_{>L}(v),\quad g(\zeta,\xi)\in E_{\le L}(v),
\]
where $W_{\le L}(v+\zeta+g(\zeta,\xi))=\xi$,
$g(0,\xi_0)=D_{\zeta}g(0,\xi_0)=0$.
For $f(\zeta,\xi)=\frac12\|v+\zeta+g(\zeta,\xi)\|_s^2$, stationarity and
\eqref{geo:shape-convention} give
\[
 \nabla_{\zeta} f(0,\xi_0)=0,\qquad
 D_{\zeta}^2f(0,\xi_0)[h,k]
 =\langle(I+\mathfrak S_v)h,k\rangle_s.
\]
By \eqref{geo:II-bound}, $\|\mathfrak S_v\|\le\epsilon_L<1/2$.
The Hessian is invertible, so the same implicit-function theorem
gives an analytic critical point $m_L(\xi)$ near $v$.
By \eqref{geo:minimum-coercivity}, this point is the unique
minimizer on its fiber. Hence
$\Pi_{\le L}u=m_L(W_{\le L}(u))$ near $u_0$.

(ii) The points $u$ and $\Pi_{\le M}u$ have the same $W_{\le L}$
value, so uniqueness gives $\Pi_{\le L}\Pi_{\le M}u=\Pi_{\le L}u$.
The $W_{\le M}$ fiber through $\Pi_{\le L}u$ is contained in its
$W_{\le L}$ fiber. Hence $\Pi_{\le L}u$ also minimizes the norm on
this smaller fiber, giving $\Pi_{\le M}\Pi_{\le L}u=\Pi_{\le L}u$.

(iii) The norms $\|\Pi_{\le L}u\|_s$ increase and are bounded by
$\|u\|_s$. Every weak cluster point $\widetilde u$ of $\Pi_{\le L}u$ satisfies
$W_{\le K}(\widetilde u)=W_{\le K}(u)$ for each fixed $K$, by
Proposition~\ref{asy:finite-flow}(iii) and nesting.
Completeness \eqref{asy:completeness} gives $\widetilde u=u$.
Thus $\Pi_{\le L}u\rightharpoonup u$, and weak lower semicontinuity gives
$\|\Pi_{\le L}u\|_s\to\|u\|_s$. The Hilbert-space norm criterion
\cite[Chapters~3 and~5]{Brezis2011} yields
\eqref{geo:minimum-limit}.

We now fix an $H^s$-compact set $K\subset B_{2r}$.
By (i), \eqref{geo:min-definition}, and \eqref{geo:nested},
the functions $d_L(u)=\|u\|_s^2-\|\Pi_{\le L}u\|_s^2$ are continuous
on $K$ and decrease pointwise to zero.
We apply Dini's theorem and \eqref{geo:pythagoras} to obtain
\[
 \sup_{u\in K}\|u-\Pi_{\le L}u\|_s^2
 \le2\sup_{u\in K}d_L(u)\to0
 \qquad\text{as }L\to\infty.
\]
\end{proof}

We set $M_{\le L}=\Pi_{\le L}(B_{2r}(X))$.
By \eqref{geo:minimum-stationarity}, every $v\in M_{\le L}$
satisfies $v\in E_{\le L}(v)$.

\vskip0.05in
\Needspace{6\baselineskip}
In the next lemma, we  compare the normal projection of a point with
the minimum-norm representative of its fiber in the Sobolev
and energy norms.
\nopagebreak[4]
\begin{lemma}\label{geo:projection-comparison}
There exists $C_s>0$, independent of $L$, such that for every
$w\in B_r(X)$, with
\[
 v=\Pi_{\le L}w,\qquad p=\mathcal P_{\le L}(w)w,\qquad
 \rho=\bigl(\|w\|_s^2-\|v\|_s^2\bigr)^{1/2},
\]
the following estimates hold:
\begin{enumerate}
\renewcommand{\theenumi}{\roman{enumi}}
\renewcommand{\labelenumi}{(\theenumi)}
\item
\begin{equation}\label{geo:radial-normal-norm}
 \bigl|\|p\|_s^2-\|v\|_s^2\bigr|
 \le \frac{C_sr}{\nu\sqrt{\mu_L}}\rho^2.
\end{equation}
\item
\begin{align}
 \|A^{1/2}(p-v)\|_s
 &\le \frac{C_sr}{\nu}\rho,
 \label{geo:radial-normal-energy}\\
 \bigl|\|A^{1/2}p\|_s^2-\|A^{1/2}v\|_s^2\bigr|
 &\le \frac{C_sr^2\sqrt{\Lambda_L}}{\nu}\rho.
 \label{geo:radial-normal-heat}
\end{align}
\end{enumerate}
\end{lemma}

\begin{proof}
We follow the negative gradient flow of the squared norm along
the factor fiber and integrate the variation of its normal component.
By \eqref{geo:minimum-stationarity} and
Proposition~\ref{asy:factor-flags}(iii), we have
$v\in E_{\le L}(v)\subset H^{s+2}_\sigma$ and
$p\in E_{\le L}(w)\subset H^{s+2}_\sigma$, so the energy norms
in \eqref{geo:radial-normal-energy}--\eqref{geo:radial-normal-heat}
are well defined.

On the fiber through $w$, we set
$f(\widetilde u)=\frac12\|\widetilde u\|_s^2$, $p(\widetilde u)=\mathcal P_{\le L}(\widetilde u)\widetilde u$, and
$q(\widetilde u)=\mathcal P_{>L}(\widetilde u)\widetilde u$. By \eqref{geo:orthogonal-projections} and
\eqref{geo:shape-convention},
\[
 \nabla_{\rm fiber}f(\widetilde u)=q(\widetilde u),\qquad
 \operatorname{Hess}_{\rm fiber}f(\widetilde u)[h,h]
 =\|h\|_s^2+\langle p(\widetilde u),\mathfrak I_L(\widetilde u)(h,h)\rangle_s.
\]
We apply \eqref{geo:II-bound} to obtain
\begin{equation}\label{geo:fiber-hessian-bounds}
 \frac12\|h\|_s^2
 \le\operatorname{Hess}_{\rm fiber}f(\widetilde u)[h,h]
 \le\frac32\|h\|_s^2,\qquad h\in F_{>L}(\widetilde u),\quad \widetilde u\in B_r.
\end{equation}
We solve
\begin{equation}\label{geo:minimum-gradient-flow}
 \widetilde u'=-q(\widetilde u),\qquad \widetilde u(0)=w,
\end{equation}
and write $p(t)=p(\widetilde u(t))$, $q(t)=q(\widetilde u(t))$. By
\eqref{geo:fiber-hessian-bounds},
\begin{equation}\label{geo:gradient-decay}
 \frac d{dt}f(\widetilde u(t))=-\|q(t)\|_s^2,\qquad
 -3\|q(t)\|_s^2\le\frac d{dt}\|q(t)\|_s^2\le-\|q(t)\|_s^2.
\end{equation}
Since $q(\widetilde u)\in F_{>L}(\widetilde u)$, the flow stays in the fiber through $w$.
Moreover, \eqref{geo:gradient-decay} gives
\[
 \|\widetilde u(t)\|_s\le\|w\|_s<r,\qquad
 \|\widetilde u'(t)\|_s=\|q(t)\|_s\le e^{-t/2}\|q(0)\|_s.
\]
The bounded speed gives a strong limit at any finite maximal time,
which lies inside $B_r$ and permits continuation.
Thus the flow exists for all $t\ge0$, and the integrable speed
implies $\widetilde u(t)\to \widetilde u_\infty$ in $H^s$, with $q(\widetilde u_\infty)=0$.
The stationary-point characterization in the proof of
Proposition~\ref{geo:minima}(i) identifies $\widetilde u_\infty=v$;
hence $p(t)\to v$ as well.
Integrating \eqref{geo:gradient-decay}, we obtain
\begin{equation}\label{geo:gradient-length}
 \int_0^\infty\|q(t)\|_s^2\,dt=\rho^2/2,\qquad
 \int_0^\infty\|q(t)\|_s\,dt
 \le2\|q(0)\|_s\le\sqrt6\,\rho.
\end{equation}
Here the lower inequality in \eqref{geo:gradient-decay} gives
$\|q(0)\|_s^2\le3\rho^2/2$.

By \eqref{geo:shape-convention}, for $h\in F_{>L}(\widetilde u)$ we have
\[
 D\mathcal P_{\le L}(\widetilde u)[h]p(\widetilde u)=-\mathfrak S_{p(\widetilde u)}h,\qquad
 D\mathcal P_{\le L}(\widetilde u)[h]q(\widetilde u)=-\mathfrak I_L(\widetilde u)(h,q(\widetilde u)).
\]
Along \eqref{geo:minimum-gradient-flow}, we therefore obtain
\begin{equation}\label{geo:normal-gradient-derivative}
 p'(t)=\mathfrak S_{p(t)}q(t)
          +\mathfrak I_L(\widetilde u(t))(q(t),q(t)).
\end{equation}
We first estimate the squared norm of $p(t)$.
By tangential-normal orthogonality and \eqref{geo:II-bound},
\[
 \left|\frac d{dt}\|p(t)\|_s^2\right|
 =2|\langle p(t),\mathfrak I_L(\widetilde u(t))(q(t),q(t))\rangle_s|
 \le\frac{C_sr}{\nu\sqrt{\mu_L}}\|q(t)\|_s^2.
\]
We integrate using \eqref{geo:gradient-length}, $p(0)=p$, and
$p(t)\to v$ to prove \eqref{geo:radial-normal-norm}.

We next estimate $p'(t)$ in the energy norm.
By \eqref{geo:energy-shape}, \eqref{geo:normal-rayleigh},
\eqref{geo:II-bound}, and $\Lambda_L<\mu_L$, we have
\[
 \|A^{1/2}p'(t)\|_s
 \le\frac{C_s}{\nu}\|p(t)\|_s\|q(t)\|_s
   +\frac{C_s\sqrt{\Lambda_L}}{\nu\sqrt{\mu_L}}\|q(t)\|_s^2
 \le\frac{C_sr}{\nu}\|q(t)\|_s.
\]
We integrate in the graph norm of $A^{1/2}$ and apply
\eqref{geo:gradient-length} to prove \eqref{geo:radial-normal-energy}.
Finally, \eqref{geo:normal-rayleigh} gives
$\|A^{1/2}p\|_s+\|A^{1/2}v\|_s\le C_sr\sqrt{\Lambda_L}$.
Together with \eqref{geo:radial-normal-energy} and the
inequality $|\|u\|_s^2-\|v\|_s^2|
\le(\|u\|_s+\|v\|_s)\|u-v\|_s$, this proves
\eqref{geo:radial-normal-heat}.
\end{proof}

\subsection{Construction of compatible Sobolev coordinates}\label{geo:coordinates-section}
We use the nested minimum-norm sections to construct a coordinate
homeomorphism $\Psi$ in the strong $H^s$ topology that puts the
Navier--Stokes semiflow into lower-triangular form. We also derive
the block radial equations needed for the subsequent estimates
and blockwise construction of conjugacies.

We choose $L_j\in\sigma(A)$ with $L_{j+1}>L_j$ and
$L_j\to\infty$, and let
$P_{\le L_j}$ denote the spectral projection of $A$ onto the
sum of all eigenspaces with eigenvalues at most $L_j$.
We take $L_0=0$ and use the empty-prefix conventions
$P_{\le L_0}=0$, $F_{>L_0}(u)=X$, and $E_{\le L_0}(u)=\{0\}$.
We set
\begin{equation}\label{geo:block-spaces}
 M_{\le L_0}=\{0\},\qquad \Pi_{\le L_0}=0,\qquad
 Z_j=(P_{\le L_j}-P_{\le L_{j-1}})X.
\end{equation}
We equip $Z_j$ with the induced $H^s$ inner product and set
$\mathcal Z=\bigoplus_{j\ge1}Z_j$, with
$\|z\|_{\mathcal Z}^2=\sum_j\|z_j\|_{Z_j}^2$.
For $v\in M_{\le L_{j-1}}\cap B_r$, define the fiber
\[
 \mathcal B_j(v)=\{w\in M_{\le L_j}\cap B_r:\Pi_{\le L_{j-1}}w=v\}.
\]
For consecutive sections, set
\begin{equation}\label{geo:coordinate-domain}
 \mathcal D_j=\{(v,w):v\in M_{\le L_{j-1}}\cap B_r,
                    \ w\in\mathcal B_j(v)\}.
\end{equation}

\vskip0.05in
Next, we construct compatible coordinates whose squared block norms equal the successive energy differences.

\nopagebreak[4]
\begin{lemma}\label{geo:finite-disk}
There are smooth maps $c_j:\mathcal D_j\to Z_j$ with the following properties.
\begin{enumerate}
\renewcommand{\theenumi}{\roman{enumi}}
\renewcommand{\labelenumi}{(\theenumi)}
\item For every $v\in M_{\le L_{j-1}}\cap B_r$, the map
$w\mapsto c_j(v,w)$ is a diffeomorphism from $\mathcal B_j(v)$
onto the ball in $Z_j$ of radius $(r^2-\|v\|_{s}^2)^{1/2}$, and
\begin{equation}\label{geo:finite-energy}
 \|c_j(v,w)\|_{Z_j}^2=\|w\|_{s}^2-\|v\|_{s}^2.
\end{equation}
\item The finite coordinate map
\[
 w\longmapsto
 \bigl(c_i(\Pi_{\le L_{i-1}}w,\Pi_{\le L_i}w)\bigr)_{i\le j}
\]
is a diffeomorphism from $M_{\le L_j}\cap B_r$ onto
$B_r(\bigoplus_{i\le j}Z_i)$.
\end{enumerate}
\end{lemma}

\begin{proof}
We argue by induction on $j$, starting with $M_{\le L_0}=\{0\}$.
We prove (i) in three steps and then assemble the finite coordinates.

\emph{Step 1.} We show that each fiber energy has compact sublevel
sets and a unique nondegenerate critical point.
By Proposition~\ref{asy:finite-flow}(iii), the finite factor
$W_{\le L}$ maps the weakly compact ball $\overline B_\rho$
into a compact finite-dimensional set for $0<\rho<2r$.
Uniqueness in Proposition~\ref{geo:minima}(i) makes the local analytic
minimum sections agree on overlaps, defining a continuous section
$m_L$ on $W_{\le L}(B_{2r})$, with $\Pi_{\le L}=m_L\circ W_{\le L}$.
Since $\|\Pi_{\le L}u\|_s\le\|u\|_s$ and $\Pi_{\le L}$ fixes $M_{\le L}$, we obtain
\begin{equation}\label{geo:minimum-compactness}
 M_{\le L}\cap\overline B_\rho
 =\Pi_{\le L}(\overline B_\rho)\quad\text{is compact in }H^s.
\end{equation}
We verify the smooth fiber structure explicitly.
The identities $W_{\le L}\circ m_L=I$ and
$\Pi_{\le L}=m_L\circ W_{\le L}$ give
\[
 DW_{\le L}(m_L(\xi))Dm_L(\xi)=I,\qquad
 \ker D\Pi_{\le L}(u)=F_{>L}(u).
\]
Thus $m_L$ parametrizes $M_{\le L}$ locally, and
$\operatorname{Ran}D\Pi_{\le L}(w)=T_wM_{\le L}$ for $w\in M_{\le L}$.
For $j\ge2$, compatibility of the factors gives
\[
 \Pi_{\le L_{j-1}}\bigl(m_{L_j}(\xi)\bigr)
 =m_{L_{j-1}}\bigl(P_{\le L_{j-1}}\xi\bigr).
\]
In these local coordinates, the restriction of $\Pi_{\le L_{j-1}}$
to $M_{\le L_j}$ is the surjective linear truncation
$P_{\le L_{j-1}}:P_{\le L_j}X\to P_{\le L_{j-1}}X$.
It is therefore a submersion, and $\mathcal B_j(v)$ is a smooth
fiber of dimension $\dim Z_j$. For $j=1$, the same conclusion
holds because $M_{\le L_0}=\{0\}$.
For $h\in F_{>L_{j-1}}(w)$, decomposing
$h=(I-D\Pi_{\le L_j}(w))h+D\Pi_{\le L_j}(w)h$ and using
$\ker D\Pi_{\le L_j}(w)=F_{>L_j}(w)$ and \eqref{geo:nested}, we have 
\begin{equation}\label{geo:fiber-tangent-split}
 F_{>L_{j-1}}(w)=F_{>L_j}(w)\oplus T_w\mathcal B_j(v).
\end{equation}
We set $f_v(w)=\|w\|_s^2-\|v\|_s^2$ on $\mathcal B_j(v)$.
By \eqref{geo:minimum-compactness}, its sublevel sets below
$r^2-\|v\|_s^2$ are compact. Since $w\in M_{\le L_j}$,
\eqref{geo:minimum-stationarity} gives $w\perp_s F_{>L_j}(w)$.
Thus \eqref{geo:fiber-tangent-split} implies
\[
 df_v(w)=0\quad\Longrightarrow\quad w\in E_{\le L_{j-1}}(w).
\]
The stationary-point characterization in the proof of
Proposition~\ref{geo:minima}(i) then gives $w=v$;
for $j=1$, this follows from $E_{\le L_0}=\{0\}$.
We apply \eqref{geo:pythagoras} along curves through $v$ to obtain
\begin{equation}\label{geo:fiber-minimum-hessian}
 D^2f_v(v)[h,h]\ge2(1-\epsilon_{L_{j-1}})\|h\|_s^2,
 \qquad h\in T_v\mathcal B_j(v),
\end{equation}
where $\epsilon_{L_0}=0$. Thus $v$ is the unique critical point
and a nondegenerate minimum.

\emph{Step 2.} We construct local Morse coordinates depending
smoothly on the base point $v$.
By induction, $M_{\le L_{j-1}}\cap B_r$ is diffeomorphic to a ball.
To identify the vertical tangent space at zero, differentiate
$\Pi_{\le L}u=\mathcal P_{\le L}(\Pi_{\le L}u)\Pi_{\le L}u$ and
$W_{\le L}\Pi_{\le L}=W_{\le L}$ at zero.
By \eqref{asy:derivative-at-zero} and \eqref{asy:normal-projection},
we obtain
\[
 D\Pi_{\le L}(0)=P_{\le L}D\Pi_{\le L}(0)=P_{\le L},\qquad
 T_0\mathcal B_j(0)=Z_j.
\]
Parallel transport of a basis of $Z_j$ along radial segments
in the base coordinates gives a smooth frame of
$v\mapsto T_v\mathcal B_j(v)$; smooth parameter dependence
follows from the finite-dimensional ODE theory
\cite[Chapter~5]{HartmanODE}.
The fiberwise exponential map then gives a smooth chart
$w=\theta(v,\zeta)$, $\theta(v,0)=v$, for $\|\zeta\|_{Z_j}<\delta(v)$,
where $\delta(v)>0$ may vary smoothly with the base point.
Using Taylor's formula, we have
\[
 f_v(w)=\langle G(v,\zeta)\zeta,\zeta\rangle_{Z_j},\qquad
 G(v,\zeta)=\int_0^1(1-t)
 D_{\zeta}^2\bigl[f_v\circ\theta(v,\cdot)\bigr](t\zeta)\,dt.
\]
By \eqref{geo:fiber-minimum-hessian}, we have that $G(v,\zeta)>0$ near $\zeta=0$.
We set
\begin{equation}\label{geo:local-morse-coordinate}
 \kappa_v(w)=G(v,\zeta)^{1/2}\zeta,\qquad
 \|\kappa_v(w)\|_{Z_j}^2=f_v(w).
\end{equation}
We apply the inverse-function theorem with parameters to obtain
a local inverse $\kappa_v^{-1}$ from a neighborhood of $0$ in $Z_j$
onto a neighborhood of $v$ in $\mathcal B_j(v)$, depending smoothly
on $v$.

\emph{Step 3.} We extend $\kappa_v$ to the entire fiber
$\mathcal B_j(v)$ by a flow that scales its energy. We use the induced $H^s$ inner product on
$T_w\mathcal B_j(v)$ and define the fiber gradient by
\[
 \langle\nabla^{\rm v}f_v(w),h\rangle_s=df_v(w)[h],
 \qquad h\in T_w\mathcal B_j(v).
\]
We define two tangent vector fields:
\[
 Y_0(v,w)=[D\kappa_v(w)]^{-1}\kappa_v(w)
 \quad\text{near }w=v,\qquad
 Y_1(v,w)=\frac{2f_v(w)\nabla^{\rm v}f_v(w)}
 {\|\nabla^{\rm v}f_v(w)\|_s^2}
 \quad(w\ne v).
\]
The uniqueness of the critical point makes $Y_1$ well defined.
By \eqref{geo:local-morse-coordinate}, both vector fields satisfy
$df_v[Y_i]=2f_v$.
We choose a smooth cutoff $\chi(v,w)\in[0,1]$ equal to $1$
near $w=v$ and supported in the domain of $\kappa_v$, and set
\[
 Y(v,w)=\chi(v,w)Y_0(v,w)+(1-\chi(v,w))Y_1(v,w),
 \qquad Y(v,v)=0.
\]
We denote the flow of $Y(v,\cdot)$ by $\phi_t^v$.
Since $df_v[Y]=2f_v$, we obtain
\begin{equation}\label{geo:coordinate-scaling-flow}
 Yf_v=2f_v,\qquad
 f_v(\phi_{-t}^v w)=e^{-2t}f_v(w).
\end{equation}
By \eqref{geo:minimum-compactness}, the negative orbit remains
in a compact sublevel set. The continuation theorem for finite-dimensional ODEs
\cite[Chapter~2]{HartmanODE} extends $\phi_{-t}^v w$
to all $t\ge0$. By \eqref{geo:pythagoras} and
\eqref{geo:coordinate-scaling-flow},
\[
 \|\phi_{-t}^v w-v\|_s^2
 \le\frac{e^{-2t}f_v(w)}{1-\epsilon_{L_{j-1}}}\longrightarrow0.
\]
For sufficiently large $t$, we define
\begin{equation}\label{geo:global-fiber-coordinate}
 c_j(v,w)=e^t\kappa_v(\phi_{-t}^v w).
\end{equation}
Once the negative orbit lies in the neighborhood where $\chi=1$,
we have $D\kappa_vY=\kappa_v$ and hence
\[
 \frac d{dt}\bigl[e^t\kappa_v(\phi_{-t}^v w)\bigr]=0.
\]
Thus \eqref{geo:global-fiber-coordinate} is independent of the
chosen sufficiently large $t$. By \eqref{geo:local-morse-coordinate}
and \eqref{geo:coordinate-scaling-flow}, we obtain
$\|c_j(v,w)\|_{Z_j}^2=f_v(w)$, which proves
\eqref{geo:finite-energy}.

For a vector $z\in Z_j$ with
$\|z\|_{Z_j}^2<r^2-\|v\|_s^2$, we choose $t$ sufficiently large
that $\kappa_v^{-1}(e^{-t}z)$ lies in a coordinate ball where
$Y=Y_0$, and set
\[
 w=\phi_t^v\bigl(\kappa_v^{-1}(e^{-t}z)\bigr).
\]
By \eqref{geo:local-morse-coordinate} and
\eqref{geo:coordinate-scaling-flow}, for $0\le q\le t$ we have
\[
 f_v\bigl(\phi_q^v(\kappa_v^{-1}(e^{-t}z))\bigr)
 =e^{2(q-t)}\|z\|_{Z_j}^2\le\|z\|_{Z_j}^2.
\]
We use \eqref{geo:minimum-compactness} and the ODE continuation
theorem to obtain existence up to time $t$. The flow identity
and \eqref{geo:global-fiber-coordinate} show that this formula
inverts $w\mapsto c_j(v,w)$.
For each fixed $(v,w)$, choose $t$ so that $\phi_{-t}^v w$
lies in a smaller Morse coordinate ball where $Y=Y_0$.
Continuous dependence of finite-time flows ensures that the same
$t$ works in a neighborhood of $(v,w)$; subsequent negative
trajectories stay in this coordinate ball because their Morse
coordinates contract radially. The inverse formula likewise uses
one fixed $t$ locally in $(v,z)$.
Smooth dependence of finite-time flows on initial data and parameters
\cite[Chapter~5]{HartmanODE} now shows that $c_j$ and its fiber
inverse are smooth, proving (i).

\emph{Step 4.} We prove (ii) by assembling the fiber coordinates.
By induction and \eqref{geo:nested}, the first $j-1$ blocks
determine $v=\Pi_{\le L_{j-1}}w$, and (i) recovers $w$ smoothly
from $v$ and the last block. Moreover, \eqref{geo:finite-energy} gives
\[
 \begin{aligned}
 \sum_{i\le j}\|c_i(\Pi_{\le L_{i-1}}w,\Pi_{\le L_i}w)\|_{Z_i}^2
 &=\sum_{i\le j}
 \bigl(\|\Pi_{\le L_i}w\|_s^2-\|\Pi_{\le L_{i-1}}w\|_s^2\bigr)\\
 &=\|w\|_s^2.
 \end{aligned}
\]
Conversely, for any block vector of norm less than $r$, the
inductively recovered $v$ leaves exactly the radius
$(r^2-\|v\|_s^2)^{1/2}$ allowed by (i) for the last block.
Thus the assembled map is a diffeomorphism onto the stated ball,
proving (ii).
\end{proof}

Using Lemma~\ref{geo:finite-disk}, define
\begin{equation}\label{geo:coordinate-map}
 \Psi_j(u)=c_j(\Pi_{\le L_{j-1}}u,\Pi_{\le L_j}u),\qquad
 \Psi(u)=(\Psi_j(u))_{j\ge1}.
\end{equation}
We denote by $\pi_{\le j}$ the orthogonal projection onto the first
$j$ blocks in the Hilbert sum of
\eqref{geo:block-spaces}.

\vskip0.05in
\Needspace{6\baselineskip}
The next theorem assembles these compatible coordinates into a
norm-preserving homeomorphism in the strong $H^s$ topology and
transforms the semiflow into a lower triangular system.
\nopagebreak[4]
\begin{theorem}\label{geo:chart}
The map $\Psi$ defined by \eqref{geo:coordinate-map}
has the following properties.
\begin{enumerate}
\renewcommand{\theenumi}{\roman{enumi}}
\renewcommand{\labelenumi}{(\theenumi)}
\item
The map
\[
 \Psi:B_r(X)\longrightarrow B_r(\mathcal Z)
\]
is a homeomorphism with respect to the norms
$\|\cdot\|_s$ on $X$ and $\|\cdot\|_{\mathcal Z}$ on
$\mathcal Z$.
\item
For every $j\ge1$, the map
\[
 u\longmapsto(\Psi_1(u),\ldots,\Psi_j(u))
 \quad\text{from }B_r(X)\text{ to }\bigoplus_{i=1}^j Z_i
\]
is smooth. Moreover,
\begin{align}
 \|\Psi(u)\|_{\mathcal Z}^2
 &=\|u\|_s^2,
 \label{geo:isometry-origin}\\
 \Psi\Pi_{\le L_j}
 &=\pi_{\le j}\Psi,
 \label{geo:chart-prefix}\\
 \|u-\Pi_{\le L_j}u\|_s^2
 &\le2\sum_{i>j}\|\Psi_i(u)\|_{Z_i}^2.
 \label{geo:chart-tail}
\end{align}
\item
For every $j\ge1$, there exists a smooth vector field
\[
 V_j:B_r\Bigl(\bigoplus_{i=1}^j Z_i\Bigr)\longrightarrow Z_j
\]
such that, for $z(t)=\Psi(S_tu)$,
\[
 \dot z_j(t)=V_j(z_1(t),\ldots,z_j(t)).
\]
\end{enumerate}
\end{theorem}

\begin{proof}
We first prove (ii). By \eqref{geo:finite-energy} and \eqref{geo:minimum-limit}, we have
\begin{equation}\label{geo:coordinate-prefix-energy}
 \sum_{i\le j}\|\Psi_i(u)\|_{Z_i}^2
 =\|\Pi_{\le L_j}u\|_s^2\longrightarrow\|u\|_s^2.
\end{equation}
This proves \eqref{geo:isometry-origin}.
The nesting identity \eqref{geo:nested} and the definition
\eqref{geo:coordinate-map} give \eqref{geo:chart-prefix}.
Combining \eqref{geo:pythagoras} with
\eqref{geo:coordinate-prefix-energy} gives \eqref{geo:chart-tail}.
Smoothness of each finite prefix follows from
Lemma~\ref{geo:finite-disk}(i) and
Proposition~\ref{geo:minima}(i), completing (ii).

We next prove (i), starting with surjectivity.
For $z\in B_r(\mathcal Z)$, applying Lemma~\ref{geo:finite-disk}(ii),
we  obtain the unique $w_j\in M_{\le L_j}\cap B_r$ satisfying
$\Psi_i(w_j)=z_i$ for $i\le j$.
By \eqref{geo:chart-prefix} and the injectivity in
Lemma~\ref{geo:finite-disk}(ii), we have
\[
 \Pi_{\le L_j}w_m=w_j,\qquad m\ge j.
\]
Using \eqref{geo:pythagoras} and \eqref{geo:coordinate-prefix-energy}, we get
\begin{equation}\label{geo:finite-inverse-cauchy}
 \|w_m-w_j\|_s^2
 \le2\bigl(\|w_m\|_s^2-\|w_j\|_s^2\bigr)
 =2\sum_{j<i\le m}\|z_i\|_{Z_i}^2\longrightarrow0
 \quad(m\ge j\to\infty).
\end{equation}
Thus $w_j\to u$ in $H^s$, with
$\|u\|_s=\|z\|_{\mathcal Z}<r$ and $\Psi_i(u)=z_i$ for every $i$.
This proves surjectivity. If $\Psi(u)=\Psi(\widetilde u)$,
Lemma~\ref{geo:finite-disk}(ii) and \eqref{geo:chart-prefix}
give $\Pi_{\le L_j}u=\Pi_{\le L_j}\widetilde u$ for every $j$; letting $j\to\infty$
in \eqref{geo:minimum-limit} gives $u=\widetilde u$.

We prove continuity of $\Psi$ and $\Psi^{-1}$.
If $u_n\to u$ in $H^s$, using the continuity in (ii) and
\eqref{geo:isometry-origin}, we obtain
\[
 \Psi(u_n)\rightharpoonup\Psi(u),\qquad
 \|\Psi(u_n)\|_{\mathcal Z}=\|u_n\|_s
 \longrightarrow\|u\|_s=\|\Psi(u)\|_{\mathcal Z}.
\]
Weak convergence together with convergence of norms yields
$\Psi(u_n)\to\Psi(u)$ in $\mathcal Z$; see
\cite[Chapter~5]{Brezis2011}.
Conversely, let $z_n\to z$ in $\mathcal Z$, and set
$u_n=\Psi^{-1}z_n$, $u=\Psi^{-1}z$.
By Lemma~\ref{geo:finite-disk}(ii) and \eqref{geo:chart-prefix}, we obtain that
$\Pi_{\le L_j}u_n\to\Pi_{\le L_j}u$ for each $j$.
Then, applying \eqref{geo:chart-tail} and the triangle inequality, we have
\[
 \limsup_{n\to\infty}\|u_n-u\|_s
 \le2\sqrt2\big(\sum_{i>j}\|z_i\|_{Z_i}^2\big)^{1/2}
 \longrightarrow0\qquad(j\to\infty).
\]
Thus $\Psi^{-1}$ is continuous, which completes the proof of (i).

We finally prove (iii). Write $\Psi_{\le j}=(\Psi_1,\ldots,\Psi_j)$
and let $m_{L_j}$ be the analytic minimum-norm section constructed
in the proof of Proposition~\ref{geo:minima}(i).
By that proposition and Lemma~\ref{geo:finite-disk}(ii), the map
\[
 K_j=\Psi_{\le j}\circ m_{L_j}:
 W_{\le L_j}(B_r(X))
 \longrightarrow B_r\Bigl(\bigoplus_{i\le j}Z_i\Bigr)
\]
is a smooth diffeomorphism, and \eqref{geo:chart-prefix} gives
$\Psi_{\le j}=K_j\circ W_{\le L_j}$.
For $z(t)=\Psi(S_tu)$, covariance \eqref{asy:covariance} therefore yields
\[
 z_{\le j}(t)
 =K_j\bigl(\Phi_{L_j}^t(K_j^{-1}(z_{\le j}(0)))\bigr),
 \qquad z_{\le j}=(z_1,\ldots,z_j).
\]
By Proposition~\ref{asy:finite-flow}(i), this defines a smooth
finite-dimensional semiflow: the level-$L_j$ quotient dynamics
expressed in the coordinates $K_j$. The last component of its generator
defines $V_j$, proving (iii).
\end{proof}

The norm identity \eqref{geo:isometry-origin} and tail estimate
\eqref{geo:chart-tail} address the possible distortion of Sobolev
geometry noted in the introduction: the coordinates preserve distance
to zero and control reconstruction errors by coordinate tails.

We now identify the radial dynamics of these coordinates.
We first verify that the local parametrizations of $M_{\le L}$ are analytic
with values in $H^{2s-1}_\sigma\subset H^{s+2}_\sigma$.
By Proposition~\ref{asy:factor-flags}(iii), we choose a local basis
$e_1(v),\ldots,e_d(v)$ of $E_{\le L}(v)$ depending analytically on $v$
with values in $H^{2s-1}_\sigma$.
Since $v\in E_{\le L}(v)$ by \eqref{geo:minimum-stationarity}, we write
\[
 v=\sum_{i=1}^d \alpha_i(v)e_i(v),\qquad
 \sum_{j=1}^d\langle e_j(v),e_i(v)\rangle_s \alpha_j(v)
 =\langle v,e_i(v)\rangle_s,\quad 1\le i\le d.
\]
We invert the Gram matrix to obtain analytic coefficients $\alpha_i(v)$.
Composing this representation with the local analytic parametrizations
from Proposition~\ref{geo:minima}(i) proves the assertion.
We set $V(u)=-\nu Au-B(u,u)$ on $H^{s+2}_\sigma$ and define
\begin{equation}\label{geo:finite-vector-field}
 V_L(v)=D\Pi_{\le L}(v)V(v),\qquad v\in M_{\le L}\cap B_r.
\end{equation}
\vskip0.05in
In the next proposition, we give the energy identities for the minimum-norm representatives
and their coordinate blocks.

\nopagebreak[4]
\begin{proposition}\label{geo:radial-factor}
Let $u\in B_r(X)$.
\begin{enumerate}
\renewcommand{\theenumi}{\roman{enumi}}
\renewcommand{\labelenumi}{(\theenumi)}
\item
For every $L\ge1$, the curve $v(t)=\Pi_{\le L}(S_tu)$ satisfies
\[
 \dot v(t)=V_L(v(t)),\qquad v(0)=\Pi_{\le L}u,
\]
where $V_L$ is defined by \eqref{geo:finite-vector-field}.
Moreover, for $t\ge0$,
\begin{equation}\label{geo:factor-radial}
\begin{aligned}
 \frac12\frac d{dt}\|v(t)\|_s^2
 &=-\nu\|A^{1/2}v(t)\|_s^2
   -\langle v(t),B(v(t),v(t))\rangle_s\\
 &\le-\frac{\nu}{2}\|A^{1/2}v(t)\|_s^2.
\end{aligned}
\end{equation}
\item
For every $j\ge1$, we set
\[
 v(t)=\Pi_{\le L_{j-1}}(S_tu),\qquad
 w(t)=\Pi_{\le L_j}(S_tu),\qquad
 z_j(t)=\Psi_j(S_tu).
\]
Then, for $t\ge0$,
\begin{equation}\label{geo:block-radial}
\begin{aligned}
 \frac12\frac d{dt}\|z_j(t)\|_{Z_j}^2
 &=-\nu\bigl(\|A^{1/2}w(t)\|_s^2
             -\|A^{1/2}v(t)\|_s^2\bigr)\\
 &\quad-\langle w(t),B(w(t),w(t))\rangle_s
       +\langle v(t),B(v(t),v(t))\rangle_s.
\end{aligned}
\end{equation}
\end{enumerate}
\end{proposition}

\begin{proof}
We first prove (i). By \eqref{asy:covariance} and uniqueness
in Proposition~\ref{geo:minima}(i),
\[
 \Pi_{\le L}S_t\Pi_{\le L}u=\Pi_{\le L}S_tu.
\]
Thus $v(t)=\Pi_{\le L}S_tu$ evolves under the semiflow
$v\mapsto\Pi_{\le L}S_tv$ on $M_{\le L}$.
The regularity established above allows differentiation at $t=0$;
by \eqref{asy:equation}, its generator is
$D\Pi_{\le L}(v)V(v)=V_L(v)$.

Differentiating $W_{\le L}\Pi_{\le L}=W_{\le L}$ gives
$(I-D\Pi_{\le L}(v))X\subset F_{>L}(v)$.
Since $v\perp F_{>L}(v)$ by \eqref{geo:minimum-stationarity},
\[
 \langle v,V_L(v)\rangle_s=\langle v,V(v)\rangle_s,
\]
which proves the identity in \eqref{geo:factor-radial}.
By \eqref{asy:energy-cancellation}, $\|v\|_s<r$, and $A\ge I$,
\[
 |\langle v,B(v,v)\rangle_s|
 \le C_s\|v\|_s^3\le C_sr\|A^{1/2}v\|_s^2.
\]
Choosing $C_sr\le\nu/2$ proves the dissipation inequality.

Finally, \eqref{geo:coordinate-map} and \eqref{geo:finite-energy} give
$\|z_j(t)\|_{Z_j}^2=\|w(t)\|_s^2-\|v(t)\|_s^2$.
Differentiating and applying (i) proves \eqref{geo:block-radial},
with $v(t)=0$ when $j=1$ by \eqref{geo:block-spaces}.
\end{proof}

\section{Spectral estimates for the normal and fast spaces}
\label{mic:section}
In this section, we show that normal vectors have small
high-frequency components and fast vectors have small
low-frequency components in the $H^s$ norm. These estimates
yield the block interaction bounds and radial decay estimates
in Section~\ref{app:estimates-section} needed to assemble
the finite-dimensional conjugacies and their inverses
in the strong $H^s$ topology.

\subsection{Removal of nonresonant transport and remainder estimates}

We construct a Fourier change of variables that removes
nonresonant transport from the linearized equation and estimate
both the transformation and the resulting remainder. These
bounds provide the basis for the spectral localization estimates
used in the conjugacy construction.

We use $J_s$ from \eqref{main:operators} to identify $X$
isometrically with $L^2_\sigma$.
Since $J_s$ commutes with $A$,
we keep the notation $A$ after this identification. We denote the normalized $L^2$ norm and inner product by
$\|\cdot\|_0$ and $\langle\cdot,\cdot\rangle_0$.
We retain $\|\cdot\|_s$ for physical velocity fields and write
$\|\cdot\|_{\mathcal L(Y)}$ for operator norms on a Hilbert space $Y$.
For operators on the normalized space $L^2_\sigma$, adjoints and
orthogonal complements are taken in the $L^2$ inner product.
For the original spaces $E_{\le L}(u_0),F_{>L}(u_0)$, we retain the $H^s$ convention.
We fix a small initial value $u_0$ and write $u(t)=S_tu_0$.

We set $K_u=-\mathbb P(u\cdot\nabla)\mathbb P$ and
$E(t)=E_{u(t)}$, where $E_u$ is defined by
\eqref{asy:standardized}. Then the normalized variational equation is
\begin{equation}\label{mic:normalized-equation}
 v_t=(-\nu A+K_u+E(t))v.
\end{equation}

We verify the skew-symmetry of the transport term and bound the remaining normalized coefficients.

\nopagebreak[4]
\begin{lemma}\label{mic:normalized}
The operator $K_u:H^1_\sigma\to L^2_\sigma$ satisfies the skew-symmetry identity
\begin{equation}\label{mic:transport-skew-form}
 \langle K_uf,g\rangle_{0}=-\langle f,K_ug\rangle_{0},
 \qquad f,g\in H^1_\sigma,
\end{equation}
and
\begin{equation}\label{mic:E-bound}
 \|E(t)\|_{\mathcal L(L^2_\sigma)}\le C_s\|S_tu_0\|_{s+1}.
\end{equation}
Moreover,
\begin{equation}\label{mic:velocity-bounds}
 \|\partial_tS_tu_0\|_{s-2}\le C_s\bigl(\nu\|S_tu_0\|_s+\|S_tu_0\|_s^2\bigr).
\end{equation}
\end{lemma}
\begin{proof}
Since $u$ is divergence free, integration by parts gives
$\langle K_uf,g\rangle_0=-\langle f,K_ug\rangle_0$ for smooth
solenoidal $f,g$. The bound
$\|K_uf\|_0\le\|u\|_{L^\infty}\|\nabla f\|_0$
extends the identity to $H^1_\sigma$ by density; see
\cite[Chapter~17]{TaylorPDE3}.
Equation~\eqref{asy:standardized} gives \eqref{mic:E-bound}, while
\eqref{asy:equation} and \eqref{asy:bilinear} give
\[
 \|u_t\|_{s-2}\le\nu\|Au\|_{s-2}+\|B(u,u)\|_{s-2}
 \le C_s(\nu\|u\|_s+\|u\|_s^2).
\]
This proves \eqref{mic:velocity-bounds}.
\end{proof}

We define the operators used in Lemma~\ref{mic:gauge} through
 their Fourier coefficients. For a real divergence-free field $u\in H^s$,
we write
\[
 u(x)=\sum_{\ell\in\mathbb Z^3}u_\ell e^{i\ell\cdot x},
 \qquad
 u_\ell=\frac1{(2\pi)^3}\int_{\T^3}u(x)e^{-i\ell\cdot x}\,dx.
\]
We use $\langle k\rangle$ from \eqref{main:norm}.
For a vector-valued trigonometric polynomial $v$, we have
\[
 \widehat{(-u\cdot\nabla)v}_{m}
 =\sum_{k+\ell=m}(-iu_\ell\cdot k)\widehat v_k.
\]
Thus $-iu_\ell\cdot k$ is the scalar coefficient coupling the
input mode $k$ to the output mode $k+\ell$.
We choose a real even function $\chi\in C^\infty(\mathbb R)$ with
$0\le\chi\le1$, equal to $1$ on $[-1,1]$ and to $0$ outside
$[-2,2]$. For $k,\ell\in\mathbb Z^3$ and $0<\eta\le1$, we set
\begin{equation}\label{mic:symbols}
\begin{split}
 a_\ell(k)&=-iu_\ell\cdot k,\qquad
 d_\ell(k)=|k+\ell|^2-|k|^2,\qquad
 Q_\ell(k)=\langle k+\ell/2\rangle,\\
 s_\ell(k)&=\frac{a_\ell(k)}{\nu d_\ell(k)}
 \left[1-\chi\left(\frac{d_\ell(k)}{\eta Q_\ell(k)}\right)\right].
\end{split}
\end{equation}
Here $d_\ell(k)$ is the difference of the heat eigenvalues of
these two modes. The cutoff $\chi$ selects interactions near
$d_\ell(k)=0$; $1-\chi$ selects the nonresonant part.
We set $s_\ell(k)=0$ when $d_\ell(k)=0$.

We initially define $S^0$ and $K^0_{\rm nr}$ on vector-valued
trigonometric polynomials by
\[
 \widehat{S^0v}_m=\sum_{k+\ell=m}s_\ell(k)\widehat v_k,
 \qquad
 \widehat{K^0_{\rm nr}v}_m
 =\sum_{k+\ell=m}a_\ell(k)
 \left[1-\chi\left(\frac{d_\ell(k)}{\eta Q_\ell(k)}\right)\right]
 \widehat v_k.
\]
We set $K^0_u=-u\cdot\nabla$ and $K^0_\eta=K^0_u-K^0_{\rm nr}$.
We then apply the Leray projection to obtain operators on
divergence-free fields:
\begin{equation}\label{mic:compressed-gauge}
 S=\mathbb P S^0\mathbb P,\qquad
 K_{\rm nr}=\mathbb P K^0_{\rm nr}\mathbb P,\qquad
 K_\eta=\mathbb P K^0_\eta\mathbb P=K_u-K_{\rm nr}.
\end{equation}
We also write $S=S_\eta(u)$ to display its dependence on $u$
and $\eta$. The identity $\nu d_\ell(k)s_\ell(k)
=a_\ell(k)[1-\chi(d_\ell(k)/(\eta Q_\ell(k)))]$ gives
$[A,S]=\nu^{-1}K_{\rm nr}$, which we use to remove the
nonresonant transport term by the change of variables $v=e^Sz$.

\vskip0.05in
\Needspace{6\baselineskip}
In the next lemma, we estimate $S$ and its commutators with the heat and transport operators.
Here $s>7/2$ ensures both $H^2$ boundedness of $S_\eta(u)$ and
$L^2$ boundedness of $\partial_tS_\eta(u)=S_\eta(u_t)$:
the Fourier sums require $s>2+3/2$ and $s-2>3/2$, respectively.

\nopagebreak[4]
\begin{lemma}\label{mic:gauge}
Let $s>7/2$, $u\in H^s_\sigma$, and $0<\eta\le1$.
The operators in \eqref{mic:compressed-gauge} satisfy the following
estimates with $C_s$ independent of $u$, $\nu$, and $\eta$.
\begin{enumerate}
\renewcommand{\theenumi}{\roman{enumi}}
\renewcommand{\labelenumi}{(\theenumi)}
\item
The operator $S=S_\eta(u)$ is self-adjoint on $L^2_\sigma$,
maps real-valued vector fields to real-valued vector fields, and satisfies
\begin{equation}\label{mic:S-bounds}
\begin{aligned}
 \|S_\eta(u)\|_{\mathcal L(H^q_\sigma)}
 &\le\frac{C_s}{\nu\eta}\|u\|_s,
 &&q=0,1,2,\\
 \|S_\eta(w)\|_{\mathcal L(L^2_\sigma)}
 &\le\frac{C_s}{\nu\eta}\|w\|_{s-2},
 &&w\in H^{s-2}_\sigma.
\end{aligned}
\end{equation}
\item
On $H^2_\sigma$, we have
\begin{equation}\label{mic:gauge-commutators}
\begin{aligned}
 {}[A,S]&=\nu^{-1}K_{\rm nr},\\
 \|[K_u,S]\|_{\mathcal L(L^2_\sigma)}
 +\|[K_{\rm nr},S]\|_{\mathcal L(L^2_\sigma)}
 &\le\frac{C_s}{\nu\eta^2}\|u\|_s^2.
\end{aligned}
\end{equation}
The norms refer to the bounded extensions of the commutators
from $H^2_\sigma$ to $L^2_\sigma$.
\item
The operator $K_\eta:H^1_\sigma\to L^2_\sigma$ is skew-symmetric.
For every $v\in H^1_\sigma$,
\begin{equation}\label{mic:K-bounds}
 \|K_\eta v\|_0\le C_s\|u\|_s\|A^{1/2}v\|_0,
\end{equation}
and
\begin{equation}\label{mic:resonant-form}
 |\langle[A,K_\eta]v,v\rangle_0|
 \le C_s\eta\|u\|_s\|A^{1/2}v\|_0^2.
\end{equation}
We interpret the quadratic form in \eqref{mic:resonant-form}
by continuous extension from smooth fields to $H^1_\sigma$.
\end{enumerate}
\end{lemma}
\begin{proof}
We first work with finite Fourier sums. Throughout the proof,
we use Young's convolution inequality and the Cauchy--Schwarz bound
\begin{equation}\label{mic:fourier-summability}
 \sum_\ell\langle\ell\rangle^q|f_\ell|
 \le C\|f\|_\sigma,
 \qquad \sigma>q+3/2.
\end{equation}
Here $C$ depends only on $\sigma$ and $q$.
In particular, this applies to $f=u$, $q=0,1,2$, and to
$f=w\in H^{s-2}_\sigma$, $q=0$. The bounds below permit passage
to the stated Sobolev spaces by density.

We first prove (i), estimating $S^0$ before applying $\mathbb P$.
Since $u_\ell\cdot\ell=0$, we have
$a_\ell(k)=-iu_\ell\cdot(k+\ell/2)$.
On the support of $1-\chi(d_\ell(k)/(\eta Q_\ell(k)))$,
$|d_\ell(k)|\ge\eta Q_\ell(k)$. By \eqref{mic:symbols}, we obtain
\begin{equation}\label{mic:symbol-zero}
 |a_\ell(k)|\le |u_\ell|Q_\ell(k),\qquad
 |s_\ell(k)|\le\frac{C|u_\ell|}{\nu\eta}.
\end{equation}
By convolution and \eqref{mic:fourier-summability}, we obtain
\[
 \|S^0\|_{\mathcal L(H^q)}
 \le\frac{C}{\nu\eta}\sum_\ell\langle\ell\rangle^q|u_\ell|
 \le\frac{C_s}{\nu\eta}\|u\|_s,\qquad q=0,1,2,
\]
Taking $q=0$ and replacing $u$ by $w$ gives the second bound
in \eqref{mic:S-bounds} for $S^0_\eta(w)$.
By \eqref{mic:compressed-gauge} and the contractivity of
$\mathbb P$ on $H^q$, we obtain \eqref{mic:S-bounds}.

By $u_{-\ell}=\overline{u_\ell}$, $u_\ell\cdot\ell=0$,
and the evenness of $\chi$, we obtain from \eqref{mic:symbols}
\[
 s_{-\ell}(k+\ell)=\overline{s_\ell(k)},\qquad
 s_{-\ell}(-k)=\overline{s_\ell(k)}.
\]
These identities give self-adjointness and preservation of real
fields, respectively. Since $a_\ell(0)=a_\ell(-\ell)=0$,
$S^0$ also preserves zero mean. Orthogonal compression by
$\mathbb P$ preserves these properties, proving (i).

We next prove (ii). By \eqref{mic:symbols}, we have
\[
 [A,S^0]_{k+\ell,k}=d_\ell(k)s_\ell(k)
 =\nu^{-1}(K^0_{\rm nr})_{k+\ell,k}.
\]
Using $[A,\mathbb P]=0$ and \eqref{mic:compressed-gauge},
we obtain $[A,S]=\nu^{-1}K_{\rm nr}$.
We estimate the scalar commutators first. Write
\[
 a_\ell^{\rm nr}(k)=a_\ell(k)
 [1-\chi(d_\ell(k)/(\eta Q_\ell(k)))],
 \qquad a^\sharp\in\{a,a^{\rm nr}\},
\]
and let $K^0_\sharp$ denote the corresponding operator
$K^0_u$ or $K^0_{\rm nr}$.
We set $h(\tau)=(1-\chi(\tau))/\tau$, with $h(0)=0$.
For $|k|\ge4\langle\ell\rangle$,  differentiating
$s_\ell= a_\ell h(d_\ell/(\eta Q_\ell))/(\nu\eta Q_\ell)$, we obtain
\begin{equation}\label{mic:symbol-derivatives}
 |\nabla_k s_\ell(k)|
 \le\frac{C|u_\ell|\langle\ell\rangle}{\nu\eta^2\langle k\rangle},
 \qquad
 |\nabla_k a_\ell^\sharp(k)|
 \le\frac{C|u_\ell|\langle\ell\rangle}{\eta},
\end{equation}
where $a^\sharp=a$ or $a^{\rm nr}$.
For $|k|\ge8(\langle\ell\rangle+\langle m\rangle)$, we write
\[
 \begin{aligned}
 &a^\sharp_\ell(k+m)s_m(k)-s_m(k+\ell)a^\sharp_\ell(k)\\
 &\quad=[a^\sharp_\ell(k+m)-a^\sharp_\ell(k)]s_m(k)
       +[s_m(k)-s_m(k+\ell)]a^\sharp_\ell(k).
 \end{aligned}
\]
By \eqref{mic:symbol-zero}, \eqref{mic:symbol-derivatives},
and the mean-value formula, we obtain
\[
 \bigl|a_\ell^\sharp(k+m)s_m(k)
       -s_m(k+\ell)a_\ell^\sharp(k)\bigr|
 \le \frac{C}{\nu\eta^2}
 |u_\ell||u_m|\langle\ell\rangle\langle m\rangle.
\]
For the remaining $k$, we use \eqref{mic:symbol-zero} and
$|a^\sharp_\ell(k+m)|+|a^\sharp_\ell(k)|
 \le C|u_\ell|(\langle\ell\rangle+\langle m\rangle)$
to obtain the same bound. Convolution and
\eqref{mic:fourier-summability} now give
\[
 \|[K^0_\sharp,S^0]\|_{\mathcal L(L^2)}
 \le\frac{C}{\nu\eta^2}
       \left(\sum_\ell\langle\ell\rangle|u_\ell|\right)^2
 \le\frac{C_s}{\nu\eta^2}\|u\|_s^2.
\]

We next control the Leray projection corrections.
For its Fourier symbol $P(k)=I-k\otimes k/|k|^2$, $k\ne0$,
and $P(0)=0$, we obtain
\begin{equation}\label{mic:leray-difference}
 |k|\,\|P(k+\ell)-P(k)\|_{\mathcal L(\mathbb R^3)}\le C|\ell|.
\end{equation}
For $|\ell|\ge|k|/2$, we use boundedness of $P$; otherwise, we
apply the mean-value formula and
$\|\nabla P(\xi)\|_{\mathcal L(\mathbb R^3,\mathcal L(\mathbb R^3))}
 \le C/|\xi|$ along the segment from $k$ to $k+\ell$.
By \eqref{mic:leray-difference} and $|a^\sharp_\ell(k)|\le|u_\ell||k|$,
we obtain
\begin{equation}\label{mic:leray-commutators}
 \|[\mathbb P,K^0_u]\|_{\mathcal L(L^2)}
 +\|[\mathbb P,K^0_{\rm nr}]\|_{\mathcal L(L^2)}
 \le C\sum_\ell|\ell||u_\ell|\le C_s\|u\|_s.
\end{equation}
We expand
\begin{align}\label{mic:compression-commutator}
 [\mathbb P K^0_\sharp\mathbb P,\mathbb P S^0\mathbb P]
 ={}&\mathbb P[K^0_\sharp,S^0]\mathbb P
 -\mathbb P K^0_\sharp(I-\mathbb P)S^0\mathbb P\nonumber\\
 &+\mathbb P S^0(I-\mathbb P)K^0_\sharp\mathbb P.
\end{align}
Using the bound for $S^0$ proved in (i) and
\eqref{mic:leray-commutators}, we bound the last two terms in
\eqref{mic:compression-commutator} by $C_s\|u\|_s^2/(\nu\eta)$.
Together with the scalar commutator bound, this proves
\eqref{mic:gauge-commutators}, including the bounded extensions
asserted in (ii).

We finally prove (iii). We denote the scalar coefficients of
$K^0_\eta$ by
$r_\ell(k)=a_\ell(k)\chi(d_\ell(k)/(\eta Q_\ell(k)))$.
By \eqref{mic:symbols}, $u_{-\ell}=\overline{u_\ell}$,
and the evenness of $\chi$,  we obtain
\[
 r_{-\ell}(k+\ell)=-\overline{r_\ell(k)},\qquad
 |r_\ell(k)|\le|u_\ell||k|.
\]
The first identity gives skew-symmetry of $K^0_\eta$.
The second, convolution, and \eqref{mic:fourier-summability} give
\[
 \|K^0_\eta v\|_0
 \le\sum_\ell|u_\ell|\,\|A^{1/2}v\|_0
 \le C_s\|u\|_s\|A^{1/2}v\|_0.
\]
Orthogonal compression gives the skew-symmetry of $K_\eta$
and \eqref{mic:K-bounds}.

To prove the quadratic-form bound, we use the resonant support.
If $r_\ell(k)\ne0$, both $k$ and $k+\ell$ are nonzero, and
\[
 |d_\ell(k)|\le2\eta Q_\ell(k),\qquad
 Q_\ell(k)\le C\sqrt{|k||k+\ell|}.
\]
Indeed, the first inequality and
$Q_\ell(k)\le1+(|k|+|k+\ell|)/2$ imply
$\big||k+\ell|-|k|\big|\le2\eta$.
Since $\eta\le1$ and $|k|,|k+\ell|\ge1$, we obtain
\[
 \max\{|k|,|k+\ell|\}\le3\min\{|k|,|k+\ell|\},
\]
which proves the stated bound for $Q_\ell(k)$.
By $u_\ell\cdot\ell=0$, we also have
$|a_\ell(k)|\le|u_\ell|\min\{|k|,|k+\ell|\}$.
Combining these bounds, we obtain
\[
 |d_\ell(k)r_\ell(k)|
 \le C\eta|u_\ell||k||k+\ell|.
\]
Cauchy--Schwarz for each frequency shift and
\eqref{mic:fourier-summability} therefore give
\[
 |\langle[A,K^0_\eta]v,v\rangle_0|
 \le C\eta\sum_\ell|u_\ell|\|A^{1/2}v\|_0^2
 \le C_s\eta\|u\|_s\|A^{1/2}v\|_0^2.
\]
For $v=\mathbb P v$, we use $[A,\mathbb P]=0$ and
\eqref{mic:compressed-gauge} to obtain
\[
 \langle[A,K_\eta]v,v\rangle_0
 =\langle[A,K^0_\eta]v,v\rangle_0.
\]
Density extends the skew-symmetry identity and quadratic-form
bound to $H^1_\sigma$, proving (iii). Approximation of $u$ in
$H^s$ (and of $w$ in $H^{s-2}$ in (i)) using the bounds above
completes the proof.
\end{proof}

We choose $B_{4r}(X)$ inside the common ball of
Sections~\ref{asy:section} and~\ref{geo:section}, with $r/\nu\le1$,
and set $\alpha=r/\nu$. For $u_0\in B_r(X)$,
Lemma~\ref{asy:small-flow} gives
\begin{equation}\label{mic:energy-input}
\begin{split}
 \sup_{t\ge0}\|S_tu_0\|_s&\le C_sr,\qquad
 \int_0^\infty\|S_tu_0\|_s\,dt\le C_sr/\nu,\\
 \int_0^\infty\bigl(\|S_tu_0\|_s^2+\|S_tu_0\|_{s+1}^2\bigr)\,dt
 &\le C_sr^2/\nu.
\end{split}
\end{equation}
We set $\mathcal B(t)=e^{S_\eta(u(t))}$ and define
\[
 \begin{aligned}
 \mathcal R_\eta(t)
 ={}&\mathcal B(t)^{-1}(-\nu A+K_{u(t)}+E(t))\mathcal B(t)\\
 &-\mathcal B(t)^{-1}\dot{\mathcal B}(t)+\nu A-K_\eta,
 \end{aligned}
\]
initially on $H^2_\sigma$ for almost every $t>0$.
Thus the substitution $v=\mathcal B(t)z$ in
\eqref{mic:normalized-equation} gives
\[
 z_t=(-\nu A+K_\eta+\mathcal R_\eta(t))z.
\]
We write
$\rho_\eta(t)=\|\mathcal R_\eta(t)\|_{\mathcal L(L^2_\sigma)}$
for the norm of the bounded extension established below.

\vskip0.05in
\Needspace{6\baselineskip}
The next lemma gives the transformed equation and bounds its
remainder and the change of variables. 
\begin{lemma}\label{mic:remainder}
The transformed equation and the change of variables satisfy the following bounds.
\begin{enumerate}
\renewcommand{\theenumi}{\roman{enumi}}
\renewcommand{\labelenumi}{(\theenumi)}
\item The transformed equation is
\begin{equation}\label{mic:transformed}
 z_t=(-\nu A+K_\eta+\mathcal R_\eta(t))z.
\end{equation}
The operator $\mathcal R_\eta$ is bounded on $L^2$ almost everywhere, and
\begin{equation}\label{mic:remainder-square}
 \int_0^\infty\rho_\eta(t)^2\,dt
 \le C_s(r^2/\nu)\eta^{-4}e^{C_s\alpha/\eta}.
\end{equation}
\item For either sign,
\begin{equation}\label{mic:exponential-bounds}
 \|e^{\pm S}\|_{\mathcal L(H^q)}\le e^{C_s\alpha/\eta}\quad(q=0,1,2),
 \qquad
 \|[A,e^{\pm S}]\|_{\mathcal L(H^1,L^2)}
 \le C_s\alpha e^{C_s\alpha/\eta}.
\end{equation}
\end{enumerate}
\end{lemma}
\begin{proof}
We first prove (ii), then derive an explicit formula for the
transformed remainder and estimate it using Lemmas~\ref{mic:normalized}
and~\ref{mic:gauge}. Write $S=S_\eta(u(t))$. By
\eqref{mic:S-bounds} and \eqref{mic:energy-input},
\[
 \|e^{\pm S}\|_{\mathcal L(H^q_\sigma)}
 \le e^{\|S\|_{\mathcal L(H^q_\sigma)}}
 \le e^{C_s\alpha/\eta},\qquad q=0,1,2.
\]
Since $K_{\rm nr}=K_u-K_\eta$, the transport bound in the proof of
Lemma~\ref{mic:normalized} and Lemma~\ref{mic:gauge}(iii) give
$\|K_{\rm nr}\|_{\mathcal L(H^1_\sigma,L^2_\sigma)}
\le C_s\|u(t)\|_s$.
Using \eqref{mic:gauge-commutators} and \eqref{mic:energy-input}, we obtain
\[
 \|[A,S]\|_{\mathcal L(H^1_\sigma,L^2_\sigma)}
 \le C_s\nu^{-1}\|u(t)\|_s\le C_s\alpha.
\]
We use the identity
\[
 [A,e^{\pm S}]
 =\pm\int_0^1e^{\pm(1-\theta)S}[A,S]e^{\pm\theta S}\,d\theta,
\]
initially on $H^2_\sigma$, and the preceding bounds to obtain
\eqref{mic:exponential-bounds} by density. This proves (ii).

We next derive \eqref{mic:transformed}. By
\eqref{mic:gauge-commutators} and \eqref{mic:compressed-gauge}, we have
$-\nu[A,S]=-K_{\rm nr}$ and $K_u-K_{\rm nr}=K_\eta$.
Integrating the derivative of
$e^{-\theta S}(-\nu A+K_u)e^{\theta S}$ and applying the same identity
to $K_{\rm nr}$, we obtain on $H^2_\sigma$
\[
 \begin{aligned}
 e^{-S}(-\nu A+K_u)e^S
 ={}&-\nu A+K_\eta\\
 &+\int_0^1e^{-\theta S}
 \bigl([K_u,S]-(1-\theta)[K_{\rm nr},S]\bigr)e^{\theta S}\,d\theta.
 \end{aligned}
\]
By \eqref{mic:velocity-bounds} and \eqref{mic:S-bounds},
$S(t)$ is locally absolutely continuous in $\mathcal L(L^2_\sigma)$,
with $\dot S=S_\eta(u_t)$ almost everywhere. Differentiating the
exponential series, we obtain
\[
 e^{-S}\frac d{dt}e^S
 =\int_0^1e^{-\theta S}\dot S e^{\theta S}\,d\theta.
\]
We substitute $v=e^Sz$ into \eqref{mic:normalized-equation} to obtain
\eqref{mic:transformed}, where
\begin{align}\label{mic:remainder-formula}
 \mathcal R_\eta={}&e^{-S}Ee^S
 -\int_0^1e^{-\theta S}\dot S e^{\theta S}\,d\theta\nonumber\\
 &+\int_0^1e^{-\theta S}
 \bigl([K_u,S]-(1-\theta)[K_{\rm nr},S]\bigr)e^{\theta S}\,d\theta.
\end{align}

We finally estimate \eqref{mic:remainder-formula} to complete (i).
By \eqref{mic:E-bound}, \eqref{mic:velocity-bounds},
\eqref{mic:S-bounds}, \eqref{mic:gauge-commutators}, and
\eqref{mic:exponential-bounds}, each term extends to a bounded
operator on $L^2_\sigma$ for almost every $t>0$, and we obtain
\[
 \rho_\eta(t)\le C_se^{C_s\alpha/\eta}
 \bigl(\|u(t)\|_{s+1}+\eta^{-1}\|u(t)\|_s
                 +\nu^{-1}\eta^{-2}\|u(t)\|_s^2\bigr).
\]
Since $\sup_{t\ge0}\|u(t)\|_s\le C_sr$, $\alpha=r/\nu\le1$,
and $0<\eta\le1$, this gives
\[
 \rho_\eta(t)\le C_s\eta^{-2}e^{C_s\alpha/\eta}\|u(t)\|_{s+1}.
\]
Squaring and integrating, we obtain from \eqref{mic:energy-input}
\[
 \int_0^\infty\rho_\eta(t)^2\,dt
 \le C_s\eta^{-4}e^{C_s\alpha/\eta}
       \int_0^\infty\|u(t)\|_{s+1}^2\,dt
 \le C_s\frac{r^2}{\nu}\eta^{-4}e^{C_s\alpha/\eta}.
\]
This proves \eqref{mic:remainder-square} and completes (i).
\end{proof}

\subsection{Dirichlet quotients and transformed spectral estimates}

We use terminal asymptotics to control Dirichlet quotients
and growth rates on the transformed normal and fast spaces,
then estimate their leakage beyond enlarged spectral cutoffs.
These localization estimates provide the basis for the block
interaction bounds needed in the conjugacy construction.

Let $\Lambda$ be the largest retained eigenvalue and $\mu$ the
first omitted eigenvalue, as in \eqref{geo:spectral-endpoints}.
We define
\begin{equation}\label{mic:transformed-spaces}
 \widetilde F_{>\Lambda}(t)=\mathcal B(t)^{-1}J_sF_{>\Lambda}(u(t)),\qquad
 \widetilde E_{\le \Lambda}(t)=\mathcal B(t)J_sE_{\le \Lambda}(u(t)).
\end{equation}
Since $\mathcal B(t)$ is self-adjoint, we have
$\widetilde E_{\le \Lambda}(t)=\widetilde F_{>\Lambda}(t)^\perp$.
For a normal adjoint solution $v(t)$ from
\eqref{geo:normal-continuation}, we set
$n(t)=\mathcal B(t)J_sv(t)$. Then
\begin{equation}\label{mic:dual-equation}
 n_t=(\nu A+K_\eta-\mathcal R_\eta^*)n.
\end{equation}

\vskip0.05in
The next lemma gives the estimates for the Dirichlet quotients on the transformed
normal and fast spaces and derives the corresponding growth
and decay estimates.
\nopagebreak[4]
\begin{lemma}\label{mic:dirichlet}
There exist $C_s>0$ and $D\ge0$, independent of $\Lambda$ and $t$, such that
\begin{equation}\label{mic:a-b-D}
  \int_0^\infty\rho_\eta(t)^2\,dt\le\nu D,\qquad
 D\le C_s\alpha^2\eta^{-4}e^{C_s\alpha/\eta}
\end{equation}
and the following estimates hold for every $t\ge0$.
\begin{enumerate}
\renewcommand{\theenumi}{\roman{enumi}}
\renewcommand{\labelenumi}{(\theenumi)}
\item For $n\in\widetilde E_{\le \Lambda}(t)$ and
$z\in\widetilde F_{>\Lambda}(t)\cap H^1_\sigma$, we have
\begin{equation}\label{mic:normal-quotient}
 \|A^{1/2}n\|_{0}^2\le(e^{C_s\eta\alpha}\Lambda+D)\|n\|_{0}^2,
\end{equation}
and
\begin{equation}\label{mic:fast-quotient}
 \|A^{1/2}z\|_{0}^2\ge(e^{-C_s\eta\alpha}\mu-D)\|z\|_{0}^2.
\end{equation}
\item Let $n(t)$ and $z(t)$ solve \eqref{mic:dual-equation} and
\eqref{mic:transformed}, respectively, with
$n(0)\in\widetilde E_{\le \Lambda}(0)$ and
$z(0)\in\widetilde F_{>\Lambda}(0)$. Then
\begin{equation}\label{mic:growth}
 \begin{aligned}
 \|n(t)\|_{0}&\le e^{1/4}e^{\nu(e^{C_s\eta\alpha}\Lambda+2D)t}\|n(0)\|_{0},\\
 \|z(t)\|_{0}&\le e^{1/4}e^{-\nu(e^{-C_s\eta\alpha}\mu-2D)t}\|z(0)\|_{0}.
\end{aligned}
\end{equation}
\end{enumerate}
\end{lemma}
\begin{proof}
We first derive the quotient inequalities, use terminal asymptotics
to prove (i), and then integrate the norm inequalities to prove (ii).
Initially let $u_0$ be smooth and consider nonzero solutions with
$z(0)\in\widetilde F_{>\Lambda}(0)\cap H^1_\sigma$.
Set $q(v)=\|A^{1/2}v\|_0^2/\|v\|_0^2$ for $v\ne0$.
Lemmas~\ref{asy:variational} and~\ref{mic:remainder} give the
positive-time regularity of $z$ and its continuity in $H^1_\sigma$.
Together with \eqref{geo:normal-continuation}, they also give
$n\in C([0,\infty);H^2_\sigma)\cap C^1((0,\infty);L^2_\sigma)$.
Thus \eqref{geo:quadratic-derivative} justifies differentiation
of both quotients. From \eqref{mic:transformed},
\[
 \frac d{dt}q(z)
 =-2\nu\frac{\|(A-q(z))z\|_0^2}{\|z\|_0^2}
 +\frac{\langle[A,K_\eta]z,z\rangle_0}{\|z\|_0^2}
 +2\frac{\langle(A-q(z))z,\mathcal R_\eta z\rangle_0}{\|z\|_0^2}.
\]
Applying \eqref{mic:resonant-form} and Young's inequality, and
repeating the calculation for \eqref{mic:dual-equation}, we obtain
\begin{equation}\label{mic:quotient-differential}
 \frac d{dt}q(z)\le\gamma(t)q(z)+\nu^{-1}\rho_\eta(t)^2,
 \qquad
 \frac d{dt}q(n)\ge-\gamma(t)q(n)-\nu^{-1}\rho_\eta(t)^2,
\end{equation}
where $\gamma(t)=C_s\eta\|S_tu_0\|_s$.

We next use the terminal spectral bounds to complete (i).
By \eqref{asy:normal-terminal}, \eqref{asy:tangent-asymptotic}, and
$\mathcal B(t)\to I$ in $\mathcal L(H^1_\sigma)$ from
\eqref{mic:S-bounds} and \eqref{asy:decay},
\begin{equation}\label{mic:quotient-terminal}
 \limsup_{t\to\infty}q(n(t))\le\Lambda,
 \qquad \liminf_{t\to\infty}q(z(t))\ge\mu.
\end{equation}
By \eqref{mic:energy-input}, we choose $C_s$ so that
$\int_0^\infty\gamma(t)\,dt\le C_s\eta\alpha$. We define the common constant
\[
 D=\frac{e^{C_s\eta\alpha}}\nu\sup_{u_0\in B_r}
           \int_0^\infty\rho_\eta(t;u_0)^2\,dt.
\]
By \eqref{mic:remainder-square}, $D$ satisfies
\eqref{mic:a-b-D}, after enlarging $C_s$ if necessary.
Integrating \eqref{mic:quotient-differential} on $[t,T]$
with integrating factors and using these uniform bounds, we obtain
\[
 q(n(t))\le e^{C_s\eta\alpha}q(n(T))+D,\qquad
 q(z(t))\ge e^{-C_s\eta\alpha}q(z(T))-D.
\]
Letting $T\to\infty$ and applying \eqref{mic:quotient-terminal}, we get
\[
 q(n(t))\le e^{C_s\eta\alpha}\Lambda+D,\qquad
 q(z(t))\ge e^{-C_s\eta\alpha}\mu-D.
\]
Starting the same argument at any time $t\ge0$, with the same
uniform bounds, proves \eqref{mic:normal-quotient}--\eqref{mic:fast-quotient}
for every vector in the corresponding spaces at smooth base points.

For arbitrary $u_0\in B_r(X)$, we choose smooth $u_0^{(m)}\to u_0$ in the common
ball. For $v\in E_{\le \Lambda}(u_0)$ and
$w\in F_{>\Lambda}(u_0)\cap H^{s+1}_\sigma$, we set
\[
 v_m=\mathcal P_{\le \Lambda}(u_0^{(m)})v,
 \qquad w_m=\mathcal P_{>\Lambda}(u_0^{(m)})w.
\]
By Proposition~\ref{asy:factor-flags}(iii), $v_m\to v$ and
$w_m\to w$ in $H^{s+1}$, while \eqref{mic:S-bounds} gives convergence
of the changes of variables on $H^1$. Passing to the limit proves
(i) at $t=0$; applying it at $S_tu_0$ gives the result for every $t\ge0$.

We finally prove (ii). Using \eqref{mic:transformed},
\eqref{mic:dual-equation}, the skew-symmetry in
Lemma~\ref{mic:gauge}(iii), and
\eqref{mic:normal-quotient}--\eqref{mic:fast-quotient}, we obtain
\[
 \begin{aligned}
 \frac d{dt}\log\|n(t)\|_0&\le\nu(e^{C_s\eta\alpha}\Lambda+D)+\rho_\eta(t),\\
 \frac d{dt}\log\|z(t)\|_0&\le-\nu(e^{-C_s\eta\alpha}\mu-D)+\rho_\eta(t).
 \end{aligned}
\]
By \eqref{mic:a-b-D} and Cauchy--Schwarz, we have
\[
 \int_0^t\rho_\eta(q)\,dq
 \le\sqrt{t\int_0^t\rho_\eta(q)^2\,dq}
 \le\sqrt{\nu Dt}\le\nu Dt+\tfrac14.
\]
We integrate the logarithmic inequalities to obtain \eqref{mic:growth}.
Projecting smooth approximations by $\mathcal P_{>\Lambda}(u_0)$
and using Proposition~\ref{asy:factor-flags}(iii) shows that
$F_{>\Lambda}(u_0)\cap H^{s+1}_\sigma$ is dense in $F_{>\Lambda}(u_0)$.
Lemma~\ref{asy:variational} and the bounded changes of variables
therefore extend \eqref{mic:growth} to all
$z(0)\in\widetilde F_{>\Lambda}(0)$ by continuity, completing (ii).
\end{proof}

Fix $0<\varepsilon\le1$ and set
\begin{equation}\label{mic:transformed-cutoffs}
 L_+=(e^{C_s\eta\alpha}+\varepsilon)\Lambda+3D,\qquad
 L_-=(e^{-C_s\eta\alpha}-\varepsilon)\mu-3D.
\end{equation}
We interpret a nonpositive lower cutoff as the zero projection.

\vskip0.05in

In the next lemma, we estimate the high-frequency components of
transformed normal vectors and the low-frequency components
of transformed fast vectors.

\nopagebreak[4]
\begin{lemma}\label{mic:transformed-buffers}
Let $0<\varepsilon\le1$, and let $L_\pm$ be given by
\eqref{mic:transformed-cutoffs}.
There exists $C_s>0$, independent of $\Lambda$, $t$, $\eta$,
and $\varepsilon$, such that, for every $t\ge0$,
\begin{enumerate}
\renewcommand{\theenumi}{\roman{enumi}}
\renewcommand{\labelenumi}{(\theenumi)}
\item For every $n\in\widetilde E_{\le \Lambda}(t)$, we have
\begin{equation}\label{mic:transformed-leak}
 \|A^{1/2}P_{>L_+}n\|_0
 \le C_s\left(\frac{\alpha}{\varepsilon}
             +\sqrt{\frac{D}{\varepsilon}}\right)\|n\|_0.
\end{equation}
\item For every $z\in\widetilde F_{>\Lambda}(t)$, we have
\begin{equation}\label{mic:transformed-fast-leak}
 \|P_{\le L_-}z\|_0
 \le C_s\left(\frac{\alpha}{\varepsilon}
             +\sqrt{\frac{D}{\varepsilon}}\right)
             \mu^{-1/2}\|z\|_0.
\end{equation}
\end{enumerate}
\end{lemma}
\begin{proof}
We use the growth and decay bounds of Lemma~\ref{mic:dirichlet}
to derive future integral representations, then estimate
the transport and remainder terms separately.

For (i), normalize $\|n(0)\|_0=1$ and set
$c=e^{C_s\eta\alpha}\Lambda+2D$ and
$A_Q=A|_{\operatorname{Ran}P_{>L_+}}$ with its natural domain.
Since $L_+-c=\varepsilon\Lambda+D>0$, \eqref{mic:growth} gives
$e^{-\nu A_QT}P_{>L_+}n(T)\to0$ as $T\to\infty$.
Integrating the high-frequency part of \eqref{mic:dual-equation}
and passing to this limit, we obtain
\begin{equation}\label{mic:normal-future}
 P_{>L_+}n(0)
 =-\int_0^\infty e^{-\nu A_Qt}P_{>L_+}
          (K_\eta-\mathcal R_\eta^*)n(t)\,dt.
\end{equation}
For the transport term, we use the following estimate for $0\le c<L$:
\begin{equation}\label{mic:heat-integral}
 \int_0^\infty
 \|A^{1/2}e^{-\nu(A-c)t}P_{>L}\|_{\mathcal L(L^2_\sigma)}\,dt
 \le\frac{C\sqrt L}{\nu(L-c)}.
\end{equation}
We split the spectrum at $2L$ and bound the integrand by
$C\sqrt L e^{-\nu(L-c)t}
 +C(\nu t)^{-1/2}e^{-\nu(L-c)t/2}$; integration gives
\eqref{mic:heat-integral}.
Equations~\eqref{mic:K-bounds}, \eqref{mic:normal-quotient}, and
\eqref{mic:growth} give
\[
 \|K_\eta n(t)\|_0
 \le C_s\|S_tu_0\|_s\sqrt{e^{C_s\eta\alpha}\Lambda+D}\,e^{\nu ct}.
\]
This bound, \eqref{mic:energy-input}, and \eqref{mic:heat-integral} yield
\[
 \left\|A^{1/2}\int_0^\infty e^{-\nu A_Qt}P_{>L_+}
                         K_\eta n(t)\,dt\right\|_0
 \le\frac{C_sr\sqrt{(e^{C_s\eta\alpha}\Lambda+D)L_+}}
      {\nu(\varepsilon\Lambda+D)}
 \le\frac{C_s\alpha}{\varepsilon}.
\]

For the remainder, the spectral theorem and Hilbert space duality give
\begin{equation}\label{mic:dual-heat}
 \left\|A^{1/2}\int_0^\infty
 e^{-\nu(A-c)t}P_{>L}f(t)\,dt\right\|_0
 \le\left(\sup_{\xi>L}\frac{\xi}{2\nu(\xi-c)}\right)^{1/2}
 \|f\|_{L^2(0,\infty;L^2_\sigma)}.
\end{equation}
We obtain this bound by integrating
$\|A^{1/2}e^{-\nu(A-c)t}P_{>L}w\|_0^2$ for the adjoint operator.
For $f(t)=e^{-\nu ct}\mathcal R_\eta(t)^*n(t)$,
\eqref{mic:growth} and \eqref{mic:a-b-D} give
\[
 \|f\|_{L^2(0,\infty;L^2_\sigma)}\le C\sqrt{\nu D},
 \qquad \frac{L_+}{L_+-c}\le\frac{C_s}{\varepsilon}.
\]
Hence \eqref{mic:dual-heat} yields
\[
 \left\|A^{1/2}\int_0^\infty e^{-\nu A_Qt}P_{>L_+}
                      \mathcal R_\eta(t)^*n(t)\,dt\right\|_0
 \le C_s\sqrt{D/\varepsilon}.
\]
We combine the two bounds in \eqref{mic:normal-future} and use
homogeneity to prove \eqref{mic:transformed-leak} at $t=0$.

For (ii), the case $L_-\le0$ is immediate since $P_{\le L_-}=0$.
Assume $L_->0$ and initially take
$z(0)\in\widetilde F_{>\Lambda}(0)\cap H^1_\sigma$ with $\|z(0)\|_0=1$.
Set
$c_-=e^{-C_s\eta\alpha}\mu-2D$ and
$A_P=A|_{\operatorname{Ran}P_{\le L_-}}$.
Since $c_--L_-=\varepsilon\mu+D>0$, \eqref{mic:growth} gives
$e^{\nu A_PT}P_{\le L_-}z(T)\to0$. Integrating the low-frequency part
of \eqref{mic:transformed} and letting $T\to\infty$, we obtain
\begin{equation}\label{mic:fast-future}
 P_{\le L_-}z(0)
 =-\int_0^\infty e^{\nu A_Pt}P_{\le L_-}
           (K_\eta+\mathcal R_\eta)z(t)\,dt.
\end{equation}
By the skew-symmetry in Lemma~\ref{mic:gauge}(iii) and
\eqref{mic:K-bounds}, we obtain
\[
 \|P_{\le L_-}K_\eta\|_{\mathcal L(L^2_\sigma)}
 =\|K_\eta P_{\le L_-}\|_{\mathcal L(L^2_\sigma)}
 \le C_s\|S_tu_0\|_s\sqrt{L_-}.
\]
We estimate \eqref{mic:fast-future} using \eqref{mic:energy-input},
\eqref{mic:growth}, \eqref{mic:a-b-D}, and Cauchy--Schwarz to obtain
\[
 \begin{aligned}
 \|P_{\le L_-}z(0)\|_0
 \le\frac{C_sr\sqrt{L_-}}{\nu(\varepsilon\mu+D)}
      +\frac{C\sqrt{\nu D}}{\sqrt{\nu(\varepsilon\mu+D)}}
 \le C_s\left(\frac\alpha\varepsilon
                     +\sqrt{\frac D\varepsilon}\right)\mu^{-1/2}.
 \end{aligned}
\]
Homogeneity and the density argument in the proof of
Lemma~\ref{mic:dirichlet} extend \eqref{mic:transformed-fast-leak}
to every $z(0)\in\widetilde F_{>\Lambda}(0)$.
Starting the same arguments at any time, with the uniform bounds
\eqref{mic:energy-input} and \eqref{mic:a-b-D}, proves (i) and (ii)
for every $t\ge0$.
\end{proof}

\subsection{Spectral estimates in the Sobolev inner product}

We transfer the preceding spectral localization estimates
through $\mathcal B$ and $\mathcal B^{-1}$ to the original
$H^s$ inner product, obtaining bounds for the normal and fast
spaces that apply to the minimum-norm projections and block
radial equations used in the conjugacy construction.

We retain the parameters $D,\eta,\varepsilon$ from
\eqref{mic:a-b-D} and \eqref{mic:transformed-cutoffs}.

\vskip0.05in
\Needspace{6\baselineskip}
The next proposition transfers the spectral localization estimates
to the original $H^s$ inner product and bounds the Rayleigh quotients
on the normal and fast spaces.
\nopagebreak[4]
\begin{proposition}\label{mic:physical-buffers}
Let $0<\eta,\varepsilon\le1$, $\alpha=r/\nu$, and let
$D$ be as in \eqref{mic:a-b-D}.
There exists $C_s>0$ such that the following estimates hold
for every $u_0\in B_r(X)$ and every spectral cut with endpoints
$\Lambda,\mu$ defined in \eqref{geo:spectral-endpoints}.
\begin{enumerate}
\renewcommand{\theenumi}{\roman{enumi}}
\renewcommand{\labelenumi}{(\theenumi)}
\item For $v\in E_{\le \Lambda}(u_0)$ and $w\in F_{>\Lambda}(u_0)$, we have
\begin{align}\label{mic:physical-leak}
 \|A^{1/2}P_{>(e^{C_s\eta\alpha}+2\varepsilon)\Lambda+4D}v\|_{s}
 &\le C_s\alpha\eta^{-2}\varepsilon^{-1}e^{C_s\alpha/\eta}\|v\|_s,\nonumber\\
 \|P_{\le (e^{-C_s\eta\alpha}-2\varepsilon)\mu-4D}w\|_{s}
 &\le C_s\alpha\eta^{-2}\varepsilon^{-1}e^{C_s\alpha/\eta}\mu^{-1/2}\|w\|_s.
\end{align}
\item There exists $C_{\eta,\varepsilon}\ge0$ satisfying
\begin{equation}\label{mic:C}
 C_{\eta,\varepsilon}
 \le C_s\alpha^2\eta^{-4}\varepsilon^{-2}e^{C_s\alpha/\eta}
\end{equation}
such that, for $v\in E_{\le \Lambda}(u_0)$ and
$w\in F_{>\Lambda}(u_0)\cap H^{s+1}_\sigma$, we have
\begin{align}\label{mic:physical-rayleigh}
 \|A^{1/2}v\|_{s}^2
 &\le((e^{C_s\eta\alpha}+2\varepsilon)\Lambda+C_{\eta,\varepsilon})
                    \|v\|_{s}^2,\nonumber\\
 \|A^{1/2}w\|_{s}^2
 &\ge((e^{-C_s\eta\alpha}-2\varepsilon)\mu-C_{\eta,\varepsilon})
                    \|w\|_{s}^2.
\end{align}

\end{enumerate}
The constants are independent of $u_0$ and the spectral cut,
and $C_s$ is independent of $\eta$ and $\varepsilon$.

\end{proposition}
\begin{proof}
We first control the mixing of separated spectral ranges by
the changes of variables, then transfer the estimates of
Lemma~\ref{mic:transformed-buffers} to the original Sobolev norm.
The Rayleigh quotient bounds follow by spectral orthogonality.

For $L'>L>0$, set $A_Q=A|_{\operatorname{Ran}P_{>L'}}$ and
$A_P=A|_{\operatorname{Ran}P_{\le L}}$, with their natural domains.
By Lemma~\ref{mic:remainder}(ii),
$e^{-A_QT}P_{>L'}e^{\pm S}P_{\le L}e^{A_PT}$ tends to zero in operator
norm at rate $e^{-(L'-L)T}$.
We differentiate $e^{-A_Qt}P_{>L'}e^{\pm S}P_{\le L}e^{A_Pt}$,
integrate on $[0,T]$, and let $T\to\infty$ to obtain
\begin{equation}\label{mic:sylvester}
 P_{>L'}e^{\pm S}P_{\le L}
 =\int_0^\infty e^{-A_Qt}P_{>L'}[A,e^{\pm S}]P_{\le L}e^{A_Pt}\,dt.
\end{equation}
Using \eqref{mic:exponential-bounds}, \eqref{mic:heat-integral}
with viscosity one, and
$\|P_{\le L}\|_{\mathcal L(L^2_\sigma,H^1_\sigma)}\le C\sqrt L$
on each nonempty spectral range, we obtain
\begin{equation}\label{mic:sylvester-bounds}
 \begin{aligned}
 \|P_{>L'}e^{\pm S}P_{\le L}\|_{\mathcal L(L^2_\sigma)}
 &\le\frac{C_s\alpha e^{C_s\alpha/\eta}\sqrt L}{L'-L},\\
 \|A^{1/2}P_{>L'}e^{\pm S}P_{\le L}\|_{\mathcal L(L^2_\sigma)}
 &\le\frac{C_s\alpha e^{C_s\alpha/\eta}\sqrt{LL'}}{L'-L}.
 \end{aligned}
\end{equation}
Here we use $A\ge I$. Self-adjointness of $e^{\pm S}$ gives
the same first bound for $P_{\le L}e^{\pm S}P_{>L'}$.

We now transfer the estimates of
Lemma~\ref{mic:transformed-buffers} to the original spaces.
For $v\in E_{\le \Lambda}(u_0)$, we set $n=\mathcal BJ_sv$ and
\[
 L=(e^{C_s\eta\alpha}+\varepsilon)\Lambda+3D,\qquad
 L'=(e^{C_s\eta\alpha}+2\varepsilon)\Lambda+4D.
\]
We write
$J_sv=\mathcal B^{-1}P_{\le L}n+\mathcal B^{-1}P_{>L}n$.
We control the first term by \eqref{mic:sylvester-bounds} and
the second by \eqref{mic:transformed-leak} and the $H^1$ bound
for $\mathcal B^{-1}$ in \eqref{mic:exponential-bounds}.
Since $L'-L=\varepsilon\Lambda+D$, these estimates give
\[
 \|A^{1/2}P_{>L'}J_sv\|_0
 \le C_se^{C_s\alpha/\eta}\left(\frac\alpha\varepsilon+\sqrt{\frac D\varepsilon}\right)\|J_sv\|_0.
\]
For $w\in F_{>\Lambda}(u_0)$, we set $z=\mathcal B^{-1}J_sw$ and
\[
 L=(e^{-C_s\eta\alpha}-\varepsilon)\mu-3D,\qquad
 L'=(e^{-C_s\eta\alpha}-2\varepsilon)\mu-4D.
\]
If $L'\le0$, the low-frequency estimate is immediate.
Otherwise, split $J_sw=\mathcal BP_{\le L}z+\mathcal BP_{>L}z$.
The first term is controlled by \eqref{mic:transformed-fast-leak}
and \eqref{mic:exponential-bounds}. For the second, the adjoint
estimate above gives
\[
 \|P_{\le L'}\mathcal BP_{>L}\|_{\mathcal L(L^2_\sigma)}
 \le\frac{C_s\alpha e^{C_s\alpha/\eta}\sqrt{L'}}{L-L'}.
\]
Using $L-L'=\varepsilon\mu+D$ and $L'\le\mu$, we obtain
\[
 \|P_{\le L'}J_sw\|_0
 \le C_se^{C_s\alpha/\eta}\left(\frac\alpha\varepsilon+\sqrt{\frac D\varepsilon}\right)\mu^{-1/2}\|J_sw\|_0.
\]
Since $J_s$ is an isometry and commutes with $A$, we obtain the
corresponding $H^s$ estimates with common coefficient $H$.
Using \eqref{mic:a-b-D}, we obtain
\begin{equation}\label{mic:H}
 H=C_se^{C_s\alpha/\eta}\left(\frac\alpha\varepsilon+\sqrt{\frac D\varepsilon}\right)
 \le C_s\alpha\eta^{-2}\varepsilon^{-1}e^{C_s\alpha/\eta},
\end{equation}
after enlarging $C_s$ in the exponent. This proves \eqref{mic:physical-leak} and completes (i).

For (ii), denote the upper and lower cutoffs in
\eqref{mic:physical-leak} by $L_+$ and $L_-$, respectively.
Spectral orthogonality and (i) give
\[
 \begin{aligned}
 \|A^{1/2}v\|_s^2
 &\le L_+\|v\|_s^2+\|A^{1/2}P_{>L_+}v\|_s^2
 \le(L_++H^2)\|v\|_s^2,\\
 \|A^{1/2}w\|_s^2
 &\ge L_-\bigl(\|w\|_s^2-\|P_{\le L_-}w\|_s^2\bigr)
 \ge(L_--H^2)\|w\|_s^2
 \qquad(L_->0).
 \end{aligned}
\]
The second estimate uses $L_-\le\mu$; for $L_-\le0$, its final
lower bound follows from positivity of $A$.
Thus \eqref{mic:physical-rayleigh} holds with
$C_{\eta,\varepsilon}=4D+H^2$, and \eqref{mic:a-b-D} and
\eqref{mic:H} give \eqref{mic:C}. This completes (ii).
\end{proof}

The gauge parameter may vary with the spectral cut when we apply
Proposition~\ref{mic:physical-buffers}. Each application estimates
the same physical spaces $E_{\le L}(u_0)$ and $F_{>L}(u_0)$.
We do not require uniform boundedness of the exponential gauges
as the cut tends to infinity.

\section{Spectral blocks and radial estimates}\label{app:estimates-section}
In this section, we choose spectral blocks and derive radial
inequalities for the Navier--Stokes semiflow in the coordinates
$z=\Psi(u)$. We obtain small normalized block couplings and
square-summable normalized source terms, providing uniform control
over all finite prefixes for the conjugacy construction in
Section~\ref{app:section}.

\vskip0.05in
We retain the Sobolev conventions of \eqref{main:norm},
the normal projections in \eqref{geo:orthogonal-projections},
and the spectral projections in \eqref{asy:spectral-projections}.
The block norms are induced by the $H^s$ inner product as in
\eqref{geo:block-spaces}.
We choose all balls below inside the common ball on which the results of
Sections~\ref{geo:section} and~\ref{mic:section} hold.  Constants may depend
on $s$ and on the fixed viscosity $\nu$.  They do not depend on a spectral
cut or on a point in the indicated ball.

\subsection{Complete spectral clusters and their interactions}
We choose spectral blocks consisting of complete eigenspaces
and estimate their interactions, providing the decay rates
and coupling bounds needed for the radial inequalities
and the subsequent blockwise construction of conjugacies.

For all sufficiently large integers $n$, set
\begin{equation}\label{app:cut-construction}
 k_n=\lfloor e^{\sqrt n/2}\rfloor,\qquad \Lambda_n=k_n^2.
\end{equation}
After discarding finitely many terms, we relabel this sequence by $j\ge1$
and write $g_j=\Lambda_j-\Lambda_{j-1}$, with $\Lambda_0=0$.
We use $\mu_L$ from \eqref{geo:spectral-endpoints}.

\vskip0.05in
\Needspace{6\baselineskip}
The next lemma establishes the spacing and summability properties
of the chosen spectral cutoffs.
\nopagebreak[4]
\begin{lemma}\label{app:cuts}
For the sequence defined by \eqref{app:cut-construction},
after deleting finitely many initial terms and relabeling,
the following assertions hold.
\begin{enumerate}
\renewcommand{\theenumi}{\roman{enumi}}
\renewcommand{\labelenumi}{(\theenumi)}
\item
$\Lambda_j\in\sigma(A)$ and $\Lambda_{j+1}>\Lambda_j$
for every $j\ge1$. Moreover,
\begin{equation}\label{app:cut-asymptotics}
 \log\Lambda_j\sim\sqrt j,\qquad
 g_j\sim\frac{\Lambda_j}{2\log\Lambda_j},\qquad
 \frac{g_{j+1}}{g_j}\to1
 \quad\text{as }j\to\infty.
\end{equation}

\item For every $j\ge1$,
\begin{equation}\label{app:next-eigenvalue}
 \Lambda_j<\mu_{\Lambda_j}
 \le\Lambda_j+2\sqrt{\Lambda_j}+1.
\end{equation}

\item For every $\theta>0$ and $m\ge0$,
\begin{equation}\label{app:summability}
 \sum_{j=1}^{\infty}
 \Lambda_j^{-\theta}(\log\Lambda_j)^m<\infty.
\end{equation}
\end{enumerate}
\end{lemma}

\begin{proof}
Every positive integer square belongs to $\sigma(A)$ by the
spectral description in Section~\ref{main:section}; in particular,
$\Lambda_n=k_n^2\in\sigma(A)$.
By \eqref{app:cut-construction},
\[
 \Lambda_n=e^{\sqrt n}+O(e^{\sqrt n/2}),\qquad
 e^{\sqrt n/2}=o\!\left(\frac{e^{\sqrt n}}{\sqrt n}\right).
\]
The mean value theorem therefore gives
\[
 \Lambda_n-\Lambda_{n-1}
 =e^{\sqrt n}-e^{\sqrt{n-1}}+O(e^{\sqrt n/2})
 \sim\frac{e^{\sqrt n}}{2\sqrt n}.
\]
Thus the sequence is eventually strictly increasing. These
asymptotics also give \eqref{app:cut-asymptotics} after deleting
finitely many terms and relabeling, which proves (i).
Since $(k_n+1)^2\in\sigma(A)$, we have
\[
 \Lambda_n<\mu_{\Lambda_n}\le(k_n+1)^2
 =\Lambda_n+2\sqrt{\Lambda_n}+1,
\]
which proves (ii).
By \eqref{app:cut-asymptotics}, for every $\theta>0$ and $m\ge0$,
we obtain
\[
 \Lambda_j^{-\theta}(\log\Lambda_j)^m
 \le Ce^{-\theta\sqrt j/2}j^{m/2}
\]
for all sufficiently large $j$, with $C$ depending only on $\theta,m$.
The integral test gives (iii).
\end{proof}

We use $\mathcal P_{\le 0}(u)=0$ and $\mathcal P_{>0}(u)=I$.
By nesting, the orthogonal projection onto each normal block is
\begin{equation}\label{app:shell-intersection}
 \begin{aligned}
 \mathcal P_{\le \Lambda_j}(u)\mathcal P_{>\Lambda_{j-1}}(u)
 &=\mathcal P_{\le \Lambda_j}(u)-\mathcal P_{\le \Lambda_{j-1}}(u),\\
 \operatorname{Ran}\bigl(\mathcal P_{\le \Lambda_j}(u)\mathcal P_{>\Lambda_{j-1}}(u)\bigr)
 &=E_{\le \Lambda_j}(u)\cap F_{>\Lambda_{j-1}}(u).
 \end{aligned}
\end{equation}
For $j\ge3$, define
\begin{equation}\label{app:far-coefficients}
 C_j(r)=\sup_{u\in B_r}
 \|\mathcal P_{\le \Lambda_j}(u)\mathcal P_{>\Lambda_{j-1}}(u)
       A\mathcal P_{\le \Lambda_{j-2}}(u)\|_{\mathcal L(X)}.
\end{equation}
We set $C_1=C_2=0$. All operator norms below act on $X$.

\vskip0.05in
\Needspace{6\baselineskip}
In the next proposition, we estimate the energy within each normal block
and control the interactions of the heat operator between blocks.
\nopagebreak[4]
\begin{proposition}\label{app:heat-interactions}
Let $\gamma\in(1/2,1)$ and $\varepsilon_*>0$.
There exist $r_*,D_0,C_\gamma>0$ and a tail of the sequence
$(\Lambda_j)$ such that, after relabeling this tail, the following
estimates hold.
\begin{enumerate}
\renewcommand{\theenumi}{\roman{enumi}}
\renewcommand{\labelenumi}{(\theenumi)}
\item For every $j\ge2$,
\begin{equation}\label{app:normal-block-energy}
 \sup_{u\in B_{r_*}(X)}\;
 \sup_{0\ne v\in E_{\le \Lambda_j}(u)\cap F_{>\Lambda_{j-1}}(u)}
 \frac{\bigl|\langle(A-\Lambda_j)v,v\rangle_s\bigr|}{\|v\|_s^2}
 \le D_0\frac{\Lambda_j}{\log\Lambda_j}.
\end{equation}
\item For every $j\ge2$,
\begin{equation}\label{app:adjacent-heat}
 \sup_{u\in B_{r_*}(X)}
 \|\mathcal P_{\le \Lambda_j}(u)\mathcal P_{>\Lambda_{j-1}}(u)A\mathcal P_{\le \Lambda_{j-1}}(u)\|_{\mathcal L(X)}
 \le\varepsilon_*g_j.
\end{equation}
\item With $C_j(r)$ defined by \eqref{app:far-coefficients},
\begin{equation}\label{app:far-growth}
 \begin{gathered}
 C_j(r)\le C_\gamma\Lambda_j^\gamma,
 \qquad j\ge3,\quad 0<r\le r_*,\\
 \lim_{r\downarrow0}\sum_{j=3}^\infty
 \frac{C_j(r)^2}{\Lambda_j^2}=0.
 \end{gathered}
\end{equation}
\end{enumerate}
The constants $D_0$ and $C_\gamma$ are independent of $u$, $j$,
and $r\in(0,r_*]$.
\end{proposition}

\begin{proof}
We first choose the localization parameters so that the spectral
errors are small relative to the block widths. We then estimate
the energy within each block and the interactions between
adjacent and separated blocks.

Write $\alpha=r_*/\nu$, where $r_*$ will be chosen below, and
let $D,H$ satisfy \eqref{mic:a-b-D} and \eqref{mic:H}.
Proposition~\ref{mic:physical-buffers} gives, uniformly in $L$
and $u\in B_{r_*}(X)$,
\begin{align}
 \|A^{1/2}P_{>(e^{C_s\eta\alpha}+2\epsilon)L+4D}\mathcal P_{\le L}(u)\|_{\mathcal L(X)}
 &\le H,\label{app:physical-upper}\\
 \|P_{\le (e^{-C_s\eta\alpha}-2\epsilon)\mu_L-4D}\mathcal P_{>L}(u)\|_{\mathcal L(X)}
 &\le H\mu_L^{-1/2}.\label{app:physical-lower}
\end{align}
By \eqref{mic:physical-rayleigh}, we also have
\begin{align}
 \sup_{0\ne v\in E_{\le L}(u)}
   \frac{\langle Av,v\rangle_{s}}{\|v\|_{s}^2}
 &\le(e^{C_s\eta\alpha}+2\epsilon)L+4D+H^2,\label{app:rayleigh-upper}\\
 \inf_{0\ne v\in F_{>L}(u)\cap H^{s+1}}
   \frac{\langle Av,v\rangle_{s}}{\|v\|_{s}^2}
 &\ge(e^{-C_s\eta\alpha}-2\epsilon)\mu_L-4D-H^2.\label{app:rayleigh-lower}
\end{align}
We fix $\kappa>0$ and $0<\sigma<\min\{1/16,\varepsilon_*/8\}$.
Choose $0<d<\sigma/8$, and then choose $r_*>0$ sufficiently small
that $C_s\alpha\kappa+2d<\sigma/4$ and
\begin{equation}\label{app:exponent-smallness}
 \frac{C_s\alpha}{\kappa}<\gamma-\frac12,
 \qquad \frac{C_s'\alpha}{\kappa}<1,
\end{equation}
where $C_s'$ is large enough for the bound below.
For large $j$, set
\begin{equation}\label{app:log-parameters}
 \eta=\frac{\kappa}{\log\Lambda_j},\qquad
 \epsilon=\frac{d}{\log\Lambda_j}.
\end{equation}
By \eqref{mic:a-b-D} and \eqref{mic:H}, we obtain
\begin{equation}\label{app:buffer-growth}
 H\le C(\log\Lambda_j)^3\Lambda_j^{C_s\alpha/\kappa},
 \qquad
 D+H^2\le C(\log\Lambda_j)^6
                   \Lambda_j^{C_s'\alpha/\kappa}.
\end{equation}
Thus \eqref{app:exponent-smallness} and \eqref{app:buffer-growth} give
\begin{equation}\label{app:buffer-growth-simplified}
 H\sqrt{\Lambda_j}\le C\Lambda_j^\gamma,
 \qquad D+H^2=o(\Lambda_j/\log\Lambda_j).
\end{equation}
Moreover, \eqref{app:log-parameters} gives
\[
 |e^{\pm C_s\eta\alpha}-1|+2\epsilon
 \le\frac{C_s\alpha\kappa+2d+o(1)}{\log\Lambda_j}.
\]
Since $\Lambda_{j-2}\sim\Lambda_{j-1}\sim\Lambda_j$, the same
parameters can be used at all three cuts.
By \eqref{app:cut-asymptotics}, $D+H^2=o(g_j)$, and each cutoff
displacement in \eqref{app:physical-upper}--\eqref{app:physical-lower}
is less than one quarter of the corresponding block width for large $j$.
We retain such a tail and relabel it.
Thus, with the parameters in \eqref{app:log-parameters},
\[
 H\sqrt{\Lambda_j}\le C\Lambda_j^\gamma=o(g_j),
 \qquad D+H^2=o(g_j).
\]
Although the bounds for the change of variables grow with the
cutoff, these errors remain smaller than the selected block widths.
The radius $r_*$ was chosen before taking the tail and is independent
of $j$. Only the localization parameters $\eta$ and $\epsilon$ vary
with the cutoff; the estimates concern the same physical normal
and fast spaces. Thus all the following estimates hold on one ball.

We now prove (i). For $0\ne v\in E_{\le \Lambda_j}(u)\cap F_{>\Lambda_{j-1}}(u)$, we use
\eqref{app:shell-intersection} to apply \eqref{app:rayleigh-upper}
at $L=\Lambda_j$ and \eqref{app:rayleigh-lower} at
$L=\Lambda_{j-1}$. By \eqref{app:cut-asymptotics},
\eqref{app:next-eigenvalue}, and \eqref{app:buffer-growth-simplified},
we obtain
\[
 \Lambda_j-D_0\frac{\Lambda_j}{\log\Lambda_j}
 \le\frac{\langle Av,v\rangle_s}{\|v\|_s^2}
 \le\Lambda_j+D_0\frac{\Lambda_j}{\log\Lambda_j}.
\]
This proves \eqref{app:normal-block-energy}.
By \eqref{app:rayleigh-upper}, we also obtain
\begin{equation}\label{app:rough-normal-energy}
 \|A^{1/2}v\|_s\le C\sqrt{\Lambda_j}\|v\|_s,
 \qquad v\in E_{\le \Lambda_j}(u).
\end{equation}

We next prove (ii) by splitting the spectrum near $\Lambda_{j-1}$.
We set $L=\Lambda_{j-1}$, take
$v\in E_{\le L}(u)$ and $w\in E_{\le \Lambda_j}(u)\cap F_{>\Lambda_{j-1}}(u)$, and use
$\langle v,w\rangle_{s}=0$ to write
\[
 \langle Av,w\rangle_{s}=\langle (A-L)v,w\rangle_{s}.
\]
By \eqref{app:log-parameters} and
\eqref{app:buffer-growth-simplified}, for large $j$ we choose
$0<h\le\sigma L/\log L$ such that
\[
 (e^{C_s\eta\alpha}+2\epsilon)L+4D<L+h,
 \qquad
 L-h<(e^{-C_s\eta\alpha}-2\epsilon)\mu_L-4D.
\]
We split the spectrum into $A\le L-h$, $L-h<A\le L+h$, and $A>L+h$.
On the middle interval, the operator $A-L$ has norm at most $h$.
The low-frequency estimate \eqref{app:physical-lower} controls
the lower interval. On the upper interval, $0\le A-L\le A$, so
\eqref{app:physical-upper} and \eqref{app:rough-normal-energy} apply.
Together they give
\[
 \begin{aligned}
 |\langle(A-L)P_{\le L-h}v,P_{\le L-h}w\rangle_s|
 &\le CH\sqrt L\|v\|_s\|w\|_s,\\
 |\langle(A-L)P_{>L+h}v,P_{>L+h}w\rangle_s|
 &\le CH\sqrt L\|v\|_s\|w\|_s.
 \end{aligned}
\]
Taking the supremum over unit vectors, we obtain
\begin{equation}\label{app:boundary-strip}
 \|\mathcal P_{\le \Lambda_j}(u)\mathcal P_{>\Lambda_{j-1}}(u)A\mathcal P_{\le \Lambda_{j-1}}(u)\|_{\mathcal L(X)}
 \le h+CH\sqrt L
 \le\sigma\frac L{\log L}+C L^\gamma.
\end{equation}
Here $C$ may also depend on $\sigma$, but is independent of $u,r,j$.
By \eqref{app:cut-asymptotics}, $L/\log L\sim2g_j$ and
$L^\gamma=o(g_j)$. Our choice of $\sigma$ therefore gives
\eqref{app:adjacent-heat} for all large $j$.

For (iii), we first estimate the interactions between separated blocks. For $j\ge3$, take
$v\in E_{\le \Lambda_{j-2}}(u)$ and $w\in E_{\le \Lambda_j}(u)\cap F_{>\Lambda_{j-1}}(u)$.
By the parameter choice above and \eqref{app:cut-asymptotics},
we choose $L\asymp\Lambda_j$ between the upper cutoff for
$E_{\le \Lambda_{j-2}}(u)$ in \eqref{app:physical-upper} and the lower
cutoff for $F_{>\Lambda_{j-1}}(u)$ in \eqref{app:physical-lower}.
By \eqref{app:physical-upper}, \eqref{app:physical-lower}, and
\eqref{app:rough-normal-energy}, we obtain
\begin{align*}
 \abs{\langle Av,w\rangle_s}
 &\le\|A^{1/2}P_{>L}v\|_s\|A^{1/2}w\|_s
       +\|AP_{\le L}v\|_s\|P_{\le L}w\|_s\\
 &\le CH\sqrt{\Lambda_j}\|v\|_s\|w\|_s.
\end{align*}
Taking the supremum over unit vectors $v,w$ and using
\eqref{app:far-coefficients} and \eqref{app:buffer-growth-simplified},
we obtain $C_j(r)\le C_\gamma\Lambda_j^\gamma$ for $0<r\le r_*$.

For each fixed $L$, Proposition~\ref{asy:factor-flags}(iii) and
\eqref{asy:derivative-at-zero} give
\begin{equation}\label{app:projection-origin}
 \|A(\mathcal P_{\le L}(u)-P_{\le L})\|_{\mathcal L(X)}\to0
 \quad\text{as }u\to0.
\end{equation}
Since distinct physical spectral blocks are orthogonal and
invariant under $A$, \eqref{app:far-coefficients} then gives
$C_j(r)\to0$ as $r\downarrow0$ for each fixed $j$.
By the growth bound already proved and \eqref{app:summability},
\[
 \frac{C_j(r)^2}{\Lambda_j^2}
 \le C_\gamma^2\Lambda_j^{-2(1-\gamma)},\qquad
 \sum_{j=3}^\infty\Lambda_j^{-2(1-\gamma)}<\infty.
\]
Dominated convergence proves the limit in \eqref{app:far-growth}.

Finally, \eqref{app:projection-origin} also gives
$\|\mathcal P_{\le \Lambda_j}(u)\mathcal P_{>\Lambda_{j-1}}(u)A\mathcal P_{\le \Lambda_{j-1}}(u)\|_{\mathcal L(X)}\to0$
for each fixed $j$. A single further decrease of $r_*$ therefore
ensures \eqref{app:adjacent-heat} for the remaining finite set of
indices. All bounds already obtained remain valid on the smaller
ball, with the parameters and retained cuts fixed. This completes
the proof on one common ball.
\end{proof}

\subsection{Radial estimates for the triangular system}

We use Proposition~\ref{app:heat-interactions} to derive radial
inequalities for the triangular system in the coordinates
$z=\Psi(u)$, with small normalized block couplings and
square-summable normalized source terms. These estimates provide
uniform control for the exit-time
construction and the assembly of the conjugacies and their
inverses in Section~\ref{app:section}.

For a smooth real velocity field $u\in X$, we define the cubic functional
\begin{equation}\label{app:cubic-functional}
 T(u)=\langle u,B(u,u)\rangle_s,
 \qquad B(u,u)=\mathbb P((u\cdot\nabla)u).
\end{equation}

\vskip0.05in
The following lemma extends the cubic energy functional to $H^s_\sigma$
and gives a Lipschitz estimate on bounded balls.
\nopagebreak[4]
\begin{lemma}\label{app:cubic-energy}
There exists a unique bounded symmetric trilinear form
$\mathfrak T:X^3\to\mathbb R$ such that
\[
 T(u)=\mathfrak T(u,u,u)
\]
for every smooth $u\in X$. We define $T(u)=\mathfrak T(u,u,u)$
for all $u\in X$. Then
\begin{equation}\label{app:cubic-bound}
\begin{aligned}
 |T(u)|&\le C_s\|u\|_s^3,
 &&u\in X,\\
 |T(u)-T(\widetilde u)|&\le C_sr^2\|u-\widetilde u\|_s,
 &&u,\widetilde u\in\overline B_r(X),\quad r>0.
\end{aligned}
\end{equation}
\end{lemma}

\begin{proof}
By \eqref{asy:commutator-product}, the trilinear form
\[
 (f,g,h)\longmapsto
 \langle[J_s,f\cdot\nabla]g,J_sh\rangle_0
\]
is bounded on $X^3$. Its diagonal value equals $T(u)$ for smooth
$u$ by the transport cancellation in Lemma~\ref{asy:products}(i).
Symmetrization gives $\mathfrak T$ with
$|\mathfrak T(f,g,h)|\le C_s\|f\|_s\|g\|_s\|h\|_s$;
uniqueness follows from polarization and density.
Taking $f=g=h=u$ proves the first bound. For the second, write
\[
 T(u)-T(\widetilde u)=\mathfrak T(u-\widetilde u,u,u)+\mathfrak T(\widetilde u,u-\widetilde u,u)
                         +\mathfrak T(\widetilde u,\widetilde u,u-\widetilde u).
\]
Hence
\[
 |T(u)-T(\widetilde u)|
 \le C_s\bigl(\|u\|_s^2+\|u\|_s\|\widetilde u\|_s+\|\widetilde u\|_s^2\bigr)
       \|u-\widetilde u\|_s,
\]
which gives \eqref{app:cubic-bound} on $\overline B_r(X)$.
\end{proof}

For the retained cuts, we apply Subsection~\ref{geo:coordinates-section}
with $L_j=\Lambda_j$ and set $\Pi_{\le 0}=0$.
Whenever we retain a further tail, we use its full first prefix
$Z_1=P_{\le \Lambda_1}X$ in the same construction.
After the final choice of cuts, we fix $\Psi$ on a reference ball
containing $\overline B_{r_*}(X)$ and restrict this map when decreasing
$r\le r_*$.
By Theorem~\ref{geo:chart}, the resulting coordinates satisfy
\begin{equation}\label{app:canonical-radii}
 z=\Psi(u),\qquad
 \|z_j\|_{Z_{j}}^2=\|\Pi_{\le \Lambda_j}u\|_{s}^2-\|\Pi_{\le \Lambda_{j-1}}u\|_{s}^2.
\end{equation}
We denote its triangular vector field by $V_j(z_{\le j})$.
This notation refers to the smooth finite-prefix fields; we do not assert
that $V$ is a bounded vector field on all of $\X$.

\vskip0.05in
\Needspace{6\baselineskip}
The next lemma compares the orthogonal normal blocks with
the coordinate radii and controls the error terms needed
to derive the radial inequalities.

\nopagebreak[4]
\begin{lemma}\label{app:radius-comparison}
Let $j\ge2$, $0<r\le r_*$, and $u\in\overline B_r(X)$,
with $z=\Psi(u)$. The following estimates hold with a constant
$C>0$ independent of $u$, $r$, and $j$.
\begin{enumerate}
\renewcommand{\theenumi}{\roman{enumi}}
\renewcommand{\labelenumi}{(\theenumi)}
\item
\begin{align}
 \bigl|\|\mathcal P_{>\Lambda_{j-1}}(\Pi_{\le \Lambda_j}u)\Pi_{\le \Lambda_j}u\|_s^2
       -\|z_j\|_{Z_j}^2\bigr|
 &\le \frac{Cr}{\nu\sqrt{\Lambda_j}}\|z_j\|_{Z_j}^2,
 \label{app:q-radius}\\
 \|\mathcal P_{\le \Lambda_{j-1}}(\Pi_{\le \Lambda_j}u)\mathcal P_{>\Lambda_{j-2}}(\Pi_{\le \Lambda_j}u)\Pi_{\le \Lambda_j}u\|_s
 &\le C\|z_{j-1}\|_{Z_{j-1}}
   +C\sqrt{\frac r\nu}\Lambda_j^{-1/4}\|z_j\|_{Z_j}.
 \label{app:previous-radius}
\end{align}
\item
\begin{equation}\label{app:energy-offset}
 \begin{aligned}
 \bigl|\|A^{1/2}\mathcal P_{\le \Lambda_{j-1}}(\Pi_{\le \Lambda_j}u)\Pi_{\le \Lambda_j}u\|_s^2
       -\|A^{1/2}\Pi_{\le \Lambda_{j-1}}u\|_s^2\bigr|
 \le
 \frac{Cr^2\sqrt{\Lambda_j}}{\nu}\|z_j\|_{Z_j}.
 \end{aligned}
\end{equation}
\begin{equation}\label{app:cubic-offset}
 |T(\Pi_{\le \Lambda_j}u)-T(\Pi_{\le \Lambda_{j-1}}u)|
 \le Cr^2\|z_j\|_{Z_j}.
\end{equation}
\end{enumerate}
\end{lemma}

\begin{proof}
We combine the norm and energy comparisons of
Lemma~\ref{geo:projection-comparison} with orthogonality of the
nested normal projections and the Lipschitz bound of Lemma~\ref{app:cubic-energy}.
Initially assume $u\in B_r(X)$.

For (i), \eqref{geo:nested} and \eqref{geo:radial-normal-norm} give,
for $1\le k<j$,
\begin{equation}\label{app:prefix-norm-error}
 \begin{aligned}
 \bigl|\|\mathcal P_{\le \Lambda_k}(\Pi_{\le \Lambda_j}u)\Pi_{\le \Lambda_j}u\|_s^2-\|\Pi_{\le \Lambda_k}u\|_s^2\bigr|
 \le\frac{Cr}{\nu\sqrt{\mu_{\Lambda_k}}}
       \bigl(\|\Pi_{\le \Lambda_j}u\|_s^2-\|\Pi_{\le \Lambda_k}u\|_s^2\bigr).
 \end{aligned}
\end{equation}
By \eqref{geo:minimum-stationarity},
\[
 \mathcal P_{\le \Lambda_j}(\Pi_{\le \Lambda_j}u)\Pi_{\le \Lambda_j}u
 =\Pi_{\le \Lambda_j}u.
\]
The normal block component is therefore
$\mathcal P_{>\Lambda_{j-1}}(\Pi_{\le \Lambda_j}u)\Pi_{\le \Lambda_j}u$.
Orthogonality and \eqref{app:canonical-radii} therefore yield
\[
 \begin{aligned}
 \|\mathcal P_{>\Lambda_{j-1}}(\Pi_{\le \Lambda_j}u)\Pi_{\le \Lambda_j}u\|_s^2-\|z_j\|_{Z_j}^2
 &=\|\Pi_{\le \Lambda_{j-1}}u\|_s^2
          -\|\mathcal P_{\le \Lambda_{j-1}}(\Pi_{\le \Lambda_j}u)\Pi_{\le \Lambda_j}u\|_s^2,\\
 \|\mathcal P_{\le \Lambda_{j-1}}(\Pi_{\le \Lambda_j}u)\mathcal P_{>\Lambda_{j-2}}(\Pi_{\le \Lambda_j}u)\Pi_{\le \Lambda_j}u\|_s^2
 &=\|\mathcal P_{\le \Lambda_{j-1}}(\Pi_{\le \Lambda_j}u)\Pi_{\le \Lambda_j}u\|_s^2
          -\|\mathcal P_{\le \Lambda_{j-2}}(\Pi_{\le \Lambda_j}u)\Pi_{\le \Lambda_j}u\|_s^2.
 \end{aligned}
\]
Lemma~\ref{app:cuts} gives
$\mu_{\Lambda_{j-1}}\asymp\Lambda_j$ for $j\ge2$ and
$\mu_{\Lambda_{j-2}}\asymp\Lambda_j$ for $j\ge3$.
Applying \eqref{app:prefix-norm-error} with $k=j-1$ to the first
identity proves \eqref{app:q-radius}. Applying it with $k=j-1,j-2$
to the second and using \eqref{app:canonical-radii}, we obtain
\[
 \|\mathcal P_{\le \Lambda_{j-1}}(\Pi_{\le \Lambda_j}u)\mathcal P_{>\Lambda_{j-2}}(\Pi_{\le \Lambda_j}u)\Pi_{\le \Lambda_j}u\|_s^2
 \le\bigl(1+\frac{Cr}{\nu\sqrt{\Lambda_j}}\bigr)
       \|z_{j-1}\|_{Z_{j-1}}^2
       +\frac{Cr}{\nu\sqrt{\Lambda_j}}\|z_j\|_{Z_j}^2.
\]
For $j=2$, use $\mathcal P_{\le 0}(\Pi_{\le \Lambda_2}u)=0$ and $\Pi_{\le 0}=0$ in place of the estimate
with $k=j-2$. Taking square roots proves \eqref{app:previous-radius}.

For (ii), apply \eqref{geo:radial-normal-heat} with
$w=\Pi_{\le \Lambda_j}u$ and $L=\Lambda_{j-1}$. Equations~\eqref{geo:nested}
and \eqref{app:canonical-radii} identify
$v=\Pi_{\le \Lambda_{j-1}}u$ and $\rho=\|z_j\|_{Z_j}$.
Since $\Lambda_{j-1}\le\Lambda_j$, this gives \eqref{app:energy-offset}.
Finally, \eqref{app:cubic-bound}, \eqref{geo:pythagoras},
\eqref{geo:nested}, and \eqref{app:canonical-radii} give
\[
 |T(\Pi_{\le \Lambda_j}u)-T(\Pi_{\le \Lambda_{j-1}}u)|
 \le Cr^2\|\Pi_{\le \Lambda_j}u-\Pi_{\le \Lambda_{j-1}}u\|_s
 \le Cr^2\|z_j\|_{Z_j},
\]
which proves \eqref{app:cubic-offset}.
All estimates extend to $\overline B_r(X)$ by continuity.
\end{proof}

For the radial comparison, we use rates $\ell_j>0$ and
$0\le d_j<\ell_j$ to be chosen in the next proposition, and write
\begin{equation}\label{exit:rates}
 a_j=\ell_j-d_j,\qquad b_j=\ell_j+d_j,
\end{equation}
and set $z_0=0$.

\vskip0.05in
\Needspace{6\baselineskip}
The following proposition establishes radial inequalities for the
triangular system and the coefficient bounds needed for the
exit-time construction.
\nopagebreak[4]
\begin{proposition}\label{app:radial-inequality}
After retaining a tail of $(\Lambda_j)$ and relabeling,
we may decrease $r_*>0$ so that there exist sequences
$\ell_j>0$, $0\le d_j<\ell_j$, $\gamma_j\ge0$, and functions
$F_j:(0,r_*]\to[0,\infty)$ satisfying the following assertions.
The sequences $\ell_j,d_j,\gamma_j$ are independent of $r$.
\begin{enumerate}
\renewcommand{\theenumi}{\roman{enumi}}
\renewcommand{\labelenumi}{(\theenumi)}
\item For every $0<r\le r_*$, $j\ge1$, and
$z_{\le j}\in\overline B_r(\mathcal Z_{\le j})$,
\begin{equation}\label{app:radial-main}
 \begin{aligned}
 \bigl|\langle V_j(z_{\le j})+\ell_jz_j,z_j
          \rangle_{Z_j}\bigr|
 \le d_j\|z_j\|_{Z_j}^2
 +\|z_j\|_{Z_j}
 \bigl(\gamma_j\|z_{j-1}\|_{Z_{j-1}}+F_j(r)\bigr).
 \end{aligned}
\end{equation}
For $j=1$, we omit the preceding-block term and take
$\gamma_1=F_1(r)=0$.
\item For every $j\ge2$ and $0<r\le r_*$,
\begin{align}
 \ell_j&=\nu\Lambda_j,\qquad
 d_j=D\nu\frac{\Lambda_j}{\log\Lambda_j},
 \label{app:rates}\\
 \gamma_j&\le C_s\varepsilon_*\nu g_j,\qquad
 F_j(r)\le Cr^2\sqrt{\Lambda_j}
          +C\nu rC_j(r)+Cr^2.
 \label{app:forcing-size}
\end{align}
\item With $a_j,b_j$ defined by \eqref{exit:rates},
\begin{equation}\label{app:endpoint-gaps}
 a_j-a_{j-1}\ge c\nu g_j,\qquad
 b_j-b_{j-1}\ge c\nu g_j,
 \qquad j\ge2.
\end{equation}
\item
\begin{equation}\label{app:small-source}
 \lim_{r\downarrow0}\frac1r
 \big(\sum_{j=1}^{\infty}
       \frac{F_j(r)^2}{a_j^2}\big)^{1/2}=0.
\end{equation}
\end{enumerate}
The constants $D,C,c>0$ are independent of $r$ and $j$.
The constant $C_s$ in the bound for $\gamma_j$ depends only on $s$,
and we may take $c=1/2$ in \eqref{app:endpoint-gaps}.
The remaining constants and the retained spectral tail are fixed
after $\varepsilon_*$ is chosen.
\end{proposition}

\begin{proof}
We derive the radial inequalities from the block energy identity,
group the errors into diagonal, adjacent-block, and source terms,
and then choose the rates and verify the summability condition.

For a finite prefix $z_{\le j}\in\overline B_r(\mathcal Z_{\le j})$,
set $u=\Psi^{-1}(z_1,\ldots,z_j,0,\ldots)$.
Theorem~\ref{geo:chart} gives $\|u\|_s\le r$.
For $j\ge2$, \eqref{geo:block-radial} and \eqref{app:canonical-radii} yield
\begin{equation}\label{app:exact-radius}
 \begin{aligned}
 \langle V_j(z_{\le j}),z_j\rangle_{Z_j}
 =-\nu\bigl(\|A^{1/2}\Pi_{\le \Lambda_j}u\|_s^2
             -\|A^{1/2}\Pi_{\le \Lambda_{j-1}}u\|_s^2\bigr)-T(\Pi_{\le \Lambda_j}u)+T(\Pi_{\le \Lambda_{j-1}}u).
 \end{aligned}
\end{equation}
By \eqref{geo:minimum-stationarity}, we have
$\Pi_{\le \Lambda_j}u=\mathcal P_{\le \Lambda_{j-1}}(\Pi_{\le \Lambda_j}u)\Pi_{\le \Lambda_j}u+\mathcal P_{>\Lambda_{j-1}}(\Pi_{\le \Lambda_j}u)\Pi_{\le \Lambda_j}u$.
We expand the energy difference as
\begin{equation}\label{app:heat-expansion}
 \begin{aligned}
 &\|A^{1/2}\Pi_{\le \Lambda_j}u\|_s^2-\|A^{1/2}\Pi_{\le \Lambda_{j-1}}u\|_s^2\\
 &=\|A^{1/2}\mathcal P_{>\Lambda_{j-1}}(\Pi_{\le \Lambda_j}u)\Pi_{\le \Lambda_j}u\|_s^2
   +\|A^{1/2}\mathcal P_{\le \Lambda_{j-1}}(\Pi_{\le \Lambda_j}u)\Pi_{\le \Lambda_j}u\|_s^2
   -\|A^{1/2}\Pi_{\le \Lambda_{j-1}}u\|_s^2\\
 &\quad+2\langle A\mathcal P_{\le \Lambda_{j-1}}(\Pi_{\le \Lambda_j}u)\Pi_{\le \Lambda_j}u,
                    \mathcal P_{>\Lambda_{j-1}}(\Pi_{\le \Lambda_j}u)\Pi_{\le \Lambda_j}u\rangle_s.
 \end{aligned}
\end{equation}
The diagonal term is controlled by \eqref{app:normal-block-energy}
and \eqref{app:q-radius}:
\begin{equation}\label{app:diagonal-radius-error}
 \begin{aligned}
 \bigl|\|A^{1/2}\mathcal P_{>\Lambda_{j-1}}(\Pi_{\le \Lambda_j}u)\Pi_{\le \Lambda_j}u\|_s^2
                  -\Lambda_j\|z_j\|_{Z_j}^2\bigr|
\le C\left(\frac{\Lambda_j}{\log\Lambda_j}
               +\frac r\nu\sqrt{\Lambda_j}\right)\|z_j\|_{Z_j}^2.
 \end{aligned}
\end{equation}
For the mixed term, use the identity
\[
 \mathcal P_{\le \Lambda_{j-1}}(\Pi_{\le \Lambda_j}u)
 =\mathcal P_{\le \Lambda_{j-1}}(\Pi_{\le \Lambda_j}u)\mathcal P_{>\Lambda_{j-2}}(\Pi_{\le \Lambda_j}u)
  +\mathcal P_{\le \Lambda_{j-2}}(\Pi_{\le \Lambda_j}u)
\]
to separate adjacent and more distant blocks.
Equations~\eqref{app:adjacent-heat}, \eqref{app:far-coefficients},
\eqref{app:q-radius}, and \eqref{app:previous-radius} give
\begin{equation}\label{app:mixed-radius-error}
 \begin{aligned}
 &\bigl|\langle A\mathcal P_{\le \Lambda_{j-1}}(\Pi_{\le \Lambda_j}u)\Pi_{\le \Lambda_j}u,
                    \mathcal P_{>\Lambda_{j-1}}(\Pi_{\le \Lambda_j}u)\Pi_{\le \Lambda_j}u\rangle_s\bigr|\\
 &\qquad\le C_s\varepsilon_*g_j\|z_{j-1}\|_{Z_{j-1}}\|z_j\|_{Z_j}
 +C\varepsilon_*\sqrt{r/\nu}\,g_j\Lambda_j^{-1/4}\|z_j\|_{Z_j}^2
 +CrC_j(r)\|z_j\|_{Z_j}.
 \end{aligned}
\end{equation}
To make the dependence on $\varepsilon_*$ explicit, decrease $r_*$
if necessary so that the leading coefficients in the norm comparisons
for the current and preceding normal blocks are at most $2$;
this follows from \eqref{app:q-radius} and the squared norm estimate
in the proof of Lemma~\ref{app:radius-comparison}(i).
Combining these bounds with \eqref{app:adjacent-heat} shows that
the constant $C_s$ in the first term of \eqref{app:mixed-radius-error}
is independent of $\varepsilon_*$.
The remaining energy offset and cubic difference are bounded by
Lemma~\ref{app:radius-comparison}(ii). Substituting these bounds and
\eqref{app:heat-expansion}--\eqref{app:mixed-radius-error} into
\eqref{app:exact-radius}, we group the resulting errors as follows.
By \eqref{app:cut-asymptotics},
\begin{equation}\label{app:diagonal-absorption}
 \sqrt{\Lambda_j}=o(\Lambda_j/\log\Lambda_j),\qquad
 g_j\Lambda_j^{-1/4}=o(\Lambda_j/\log\Lambda_j).
\end{equation}
Since $r\le r_*$, all terms proportional to $\|z_j\|_{Z_j}^2$
are bounded by $D\nu\Lambda_j\|z_j\|_{Z_j}^2/\log\Lambda_j$
for sufficiently large $D$ and sufficiently large cuts, uniformly in $r$.
The adjacent-block term and the remaining source terms are controlled by
\[
 \gamma_j=C_s\varepsilon_*\nu g_j,\qquad
 F_j(r)=C\bigl(r^2\sqrt{\Lambda_j}+\nu rC_j(r)+r^2\bigr),
 \qquad j\ge2,
\]
with $C_s$ sufficiently large depending only on $s$, and with $C$
fixed after $\varepsilon_*$ is chosen.
Taking $\ell_j,d_j$ as in \eqref{app:rates} gives
\eqref{app:radial-main}--\eqref{app:forcing-size}, including when $z_j=0$.
We make the final choice of $D$ and the retained tail below.

The first block contains the full low-frequency prefix and
requires separate endpoints. By \eqref{geo:factor-radial},
\eqref{app:rayleigh-upper}, \eqref{app:cubic-bound}, and $A\ge I$,
\[
 \begin{aligned}
 -\left(\nu\Lambda_1+C\nu\frac{\Lambda_1}{\log\Lambda_1}
              +Cr_*\right)\|z_1\|_{Z_1}^2
 &\le\langle V_1(z_1),z_1\rangle_{Z_1}\\
 &\le-(\nu-Cr_*)\|z_1\|_{Z_1}^2.
 \end{aligned}
\]
We decrease $r_*$ so that $Cr_*\le\nu/2$.
Choose $D_h$ large enough to absorb the first-block upper rate
error. Then fix $D>2D_h+2$ large enough for the diagonal bounds
above and the block widths in \eqref{app:cut-asymptotics}.
Finally, retain and relabel a tail satisfying all preceding
large-cut conditions and
\begin{equation}\label{app:first-cut-threshold}
 \log\Lambda_1\ge4D+4.
\end{equation}
We define
\begin{equation}\label{app:first-endpoints}
 \begin{gathered}
 a_1=\nu/2,\qquad
 b_1=\nu\Lambda_1+D_h\nu\frac{\Lambda_1}{\log\Lambda_1},\\
 \ell_1=(a_1+b_1)/2,\qquad d_1=(b_1-a_1)/2.
 \end{gathered}
\end{equation}
These choices give \eqref{app:radial-main} for $j=1$ with
$\gamma_1=F_1(r)=0$ and $0\le d_1<\ell_1$.
By \eqref{app:rates} and \eqref{app:first-cut-threshold},
$d_j/\ell_j=D/\log\Lambda_j<1/4$ for $j\ge2$.
This completes (i) and (ii). For this final sequence of cuts,
we fix the coordinates as described before Lemma~\ref{app:radius-comparison}.
The rates and coefficients $\gamma_j$ remain fixed as $r$ decreases.

We now prove (iii). For $t\ge\Lambda_1$, we have
\[
 \frac d{dt}\left(t\pm D\frac t{\log t}\right)
 =1\pm D\left(\frac1{\log t}-\frac1{(\log t)^2}\right)
 \in[1/2,3/2]
\]
by \eqref{app:first-cut-threshold}.
Using \eqref{exit:rates}, \eqref{app:rates}, and the mean-value theorem,
we obtain \eqref{app:endpoint-gaps} for $j\ge3$.
For $j=2$, \eqref{app:first-endpoints} gives
\[
 a_2-a_1\ge\tfrac12\nu\Lambda_2\ge\tfrac12\nu g_2,\qquad
 b_2-b_1
 =\nu g_2+D\nu\frac{\Lambda_2}{\log\Lambda_2}
       -D_h\nu\frac{\Lambda_1}{\log\Lambda_1}
 \ge\nu g_2,
\]
by \eqref{app:first-cut-threshold}, $D>D_h$, and the monotonicity of
$t/\log t$ for $t\ge\Lambda_1$.
Together with the derivative bound above, this proves (iii) with
$c=1/2$, independently of $\varepsilon_*$.

For (iv), by \eqref{app:rates} and
\eqref{app:first-cut-threshold}, $a_j\asymp\nu\Lambda_j$ for $j\ge2$.
Using \eqref{app:forcing-size} and $F_1(r)=C_2(r)=0$, we obtain
\[
 \begin{aligned}
 \frac1r\big(\sum_{j\ge1}\frac{F_j(r)^2}{a_j^2}\big)^{1/2}
 \le C\frac r\nu
 \big(\sum_{j\ge2}(\Lambda_j^{-1}+\Lambda_j^{-2})\big)^{1/2}+C\big(\sum_{j\ge3}
              \frac{C_j(r)^2}{\Lambda_j^2}\big)^{1/2}\to0
 \quad\text{as }r\downarrow0,
 \end{aligned}
\]
by \eqref{app:summability} and \eqref{app:far-growth}.
This proves \eqref{app:small-source}.
\end{proof}

\section{Construction of the conjugacy}\label{app:section}\label{exit:section}

In this section, we construct a local conjugacy between the
Navier--Stokes semiflow and the heat semigroup.

\vskip0.05in
We first outline the main steps of the proof and explain how the
estimates yield a common neighborhood for the construction.

Proposition~\ref{asy:expansion} gives holomorphic asymptotic factors
on one complex ball and the estimate
\[
 \sup_{u\in B_{2r_s}(X_{\mathbb C})}
 \bigl(\|\mathcal R_N(t,u)\|_s+
       \|D_u\mathcal R_N(t,u)\|_{\mathcal L(X_{\mathbb C})}\bigr)
 \le C_Ne^{-\nu(N+1/4)t}.
\]
The constants may depend on $N$, but the induction does not reduce
the radius. Proposition~\ref{asy:factor-flags} then propagates
surjectivity of each finite factor from a neighborhood of zero
to a common invariant ball, using covariance and the dense range
of the variational evolution. The time needed may depend on the
cutoff; the ball does not.

The conjugacy construction has three steps. First, we construct a
homeomorphism $\Psi$ that puts the Navier--Stokes semiflow into
lower-triangular form, using the geometry of Section~\ref{geo:section}.
The level sets
of the asymptotic factors define invariant foliations. The nonlinear
projection $\Pi_{\le L}$ assigns to each point $u$ the unique point of minimal
$H^s$ norm on its local leaf, defined by
$W_{\le L}(\widetilde u)=W_{\le L}(u)$. The geometric estimates are uniform in the
spectral cutoff, and Proposition~\ref{geo:minima} gives
\[
 \Pi_{\le L}\Pi_{\le M}=\Pi_{\le M}\Pi_{\le L}=\Pi_{\le L}\quad(L\le M),\qquad
 \|u-\Pi_{\le L}u\|_s^2
 \le2\bigl(\|u\|_s^2-\|\Pi_{\le L}u\|_s^2\bigr).
\]
The first identity expresses compatibility across cutoffs; the second
controls the projection error by the decrease in squared norm.
We construct $\Psi$ from these minimum-norm projections using a
parameter-dependent Morse construction. Together with
$\Pi_{\le L}u\to u$ in $H^s$, the uniform estimates allow the block coordinates
to assemble into a homeomorphism in the strong $H^s$ topology.
By Theorem~\ref{geo:chart},
\[
 \|\Psi(u)\|_{\mathcal Z}=\|u\|_s,\qquad
 \Psi\Pi_{\le L_j}=\pi_{\le j}\Psi,\qquad
 \|u-\Pi_{\le L_j}u\|_s^2
 \le2\sum_{i>j}\|\Psi_i(u)\|_{Z_i}^2.
\]
In the coordinates $z=\Psi(u)$, the Navier--Stokes equation
takes the lower-triangular form
\[
 \dot z_j=V_j(z_1,\ldots,z_j),\qquad j\ge1.
\]
Thus each finite prefix evolves autonomously, allowing us to
construct compatible conjugacies block by block. We write
\[
 \Phi_t^{\rm NS}=\Psi S_t\Psi^{-1},\qquad
 \Phi_t^{\rm heat}=e^{-\nu At}.
\]

Second, we construct conjugacies for both semiflows with a common
diagonal target. The spectral estimates of Section~\ref{mic:section}
and Proposition~\ref{app:radial-inequality} supply radial inequalities
with separated endpoint rates, small normalized block couplings,
and square-summable normalized source terms.
Lemma~\ref{app:common-data} collects the uniform bounds for both
semiflows. Under these bounds, Lemma~\ref{exit:finite-homeomorphism}
uses exit times to construct compatible finite-dimensional conjugacies
with the same diagonal target
\begin{equation}\label{exit:diagonal-semigroup}
 (D_th)_j=e^{-\ell_jt}h_j.
\end{equation}
Lemma~\ref{exit:envelopes} gives scalar bounds for the conjugacies
and their inverses. Lemma~\ref{exit:power-continuity} provides the
$\ell^2$ continuity and tail control needed to use these bounds.
Together with finite-prefix continuity, they give strong continuity
of the assembled maps and their inverses. The inverse bounds also
ensure that both images contain one ball $B_\eta(\mathcal Z)$.

Finally, Proposition~\ref{exit:hilbert-conjugacies} gives
homeomorphisms onto this common ball satisfying
\[
 H_{\rm NS}\Phi_t^{\rm NS}=D_tH_{\rm NS},\qquad
 H_{\rm heat}\Phi_t^{\rm heat}=D_tH_{\rm heat}.
\]
Their domains may differ. The common target allows us to define
\[
 \mathcal H=H_{\rm heat}^{-1}\circ H_{\rm NS}\circ\Psi,
 \qquad \mathcal H S_t=e^{-\nu At}\mathcal H.
\]
The estimates in the next subsection explain why the finite
constructions use one radius independent of the truncation index.

\subsection{Preparation for the conjugacy construction}
\label{app:concrete-flows}
We choose common spectral blocks, comparison coefficients,
and a radius valid for every finite prefix of both semiflows,
providing uniform data for their conjugacy constructions
with the same diagonal target.

We use the spectral blocks $Z_j$ from
\eqref{geo:block-spaces}, with $L_j=\Lambda_j$, and the coordinate
map $\Psi$ from Theorem~\ref{geo:chart}. We identify
$X$ with $\mathcal Z=\bigoplus_{j\ge1}Z_j$ and write
$\mathcal Z_{\le N}=\pi_{\le N}\mathcal Z$.

The rates $\ell_j,d_j$, coupling coefficients $\gamma_j$, and
source bounds $F_j(r)$ are those of
Proposition~\ref{app:radial-inequality}; the endpoint rates
$a_j=\ell_j-d_j$ and $b_j=\ell_j+d_j$ are defined in
\eqref{exit:rates}. By \eqref{app:endpoint-gaps}, both endpoint
sequences are strictly increasing. We fix the spectral sequence,
the rates and coupling coefficients, and $\Psi$ on its reference ball
before decreasing the final radius. The source bounds $F_j(r)$
retain their dependence on $r$. For $0<2\rho<r_*$, we set
$F_j=F_j(2\rho)$ and restrict $\Psi$ to the smaller domain.

The construction below applies to both $\Phi_t^{\rm NS}$ and
$\Phi_t^{\rm heat}$. When treating them simultaneously, we write
$\Phi_t$ for either semiflow, and use $V^{(N)}$ and $\phi_t^{(N)}$
for its corresponding finite-prefix vector field and flow.
Both constructions use the target \eqref{exit:diagonal-semigroup}.

We set
\[
 \kappa=\nu/2,\qquad
 c_{jk}=\begin{cases}
 \gamma_j,&k=j-1,\\
 0,&k<j-1.
 \end{cases}
 \qquad 1\le k<j.
\]
We define the comparison matrices by
\begin{equation}\label{exit:matrices}
 (A_-)_{jk}=\begin{cases}
 c_{jk}/(a_j-a_k),&k<j,\\
 0,&k\ge j,
 \end{cases}
 \qquad
 (A_+)_{jk}=\begin{cases}
 c_{jk}/(b_j-b_k),&k<j,\\
 0,&k\ge j.
 \end{cases}
\end{equation}
We interpret $\exp(-c_0\ell_j/d_j)$ as zero when $d_j=0$.

The choice of a common radius rests on
\eqref{app:small-source} and the uniform matrix bounds proved below.
Once the cutoffs and coefficients are fixed, these bounds give
\[
 \|A_\pm\|_{\mathcal L(\ell^2)}\le\tfrac12,\qquad
 \Big\|(I-A_-)^{-1}
       \Big(\frac{F_j(2\rho)}{a_j}\Big)_{j\ge1}\Big\|_{\ell^2}
 \le2\Big\|\Big(\frac{F_j(2\rho)}{a_j}\Big)_{j\ge1}\Big\|_{\ell^2}
 =o(\rho).
\]
Thus the final smallness condition is imposed on the entire sequence
at once. Only finitely many reductions of the radius are needed,
each independent of the truncation index. The following lemma
makes these choices precise.

\begin{lemma}\label{app:common-data}
We can choose the spectral cutoffs $(\Lambda_j)_{j\ge1}$,
the coefficients $\ell_j,d_j,\gamma_j$ in
Proposition~\ref{app:radial-inequality}, and a radius
$\rho\in(0,r_*/2)$ such that the following assertions hold
for both $\Phi_t^{\rm NS}$ and $\Phi_t^{\rm heat}$.
We choose $\ell_j,d_j,\gamma_j$ independently of $\rho$
and set $F_j=F_j(2\rho)$.
\begin{enumerate}
\renewcommand{\theenumi}{\roman{enumi}}
\renewcommand{\labelenumi}{(\theenumi)}
\item For every $N\ge1$, the finite factors satisfy
\begin{equation}\label{exit:finite-factor}
 \pi_{\le N}\Phi_tz=\phi_t^{(N)}(\pi_{\le N}z),\qquad
 \dot z_{\le N}=V^{(N)}(z_{\le N}).
\end{equation}
The vector field $V^{(N)}$ is smooth on
$B_{2\rho}(\mathcal Z_{\le N})$, vanishes at zero, and satisfies
\begin{equation}\label{exit:inward}
 \langle V^{(N)}(u),u\rangle_{\mathcal Z}
 \le-\kappa\|u\|_{\mathcal Z}^2,
 \qquad u\in\overline B_\rho(\mathcal Z_{\le N}).
\end{equation}
\item For every $j\ge1$ and
$z_{\le j}\in\overline B_\rho(\mathcal Z_{\le j})$,
\begin{equation}\label{exit:radial}
 \begin{aligned}
 \bigl|\langle V_j(z_{\le j})+\ell_jz_j,z_j\rangle_{Z_j}\bigr|
 \le d_j\|z_j\|_{Z_j}^2 +\|z_j\|_{Z_j}
 \big(\sum_{k<j}c_{jk}\|z_k\|_{Z_k}+F_j\big).
 \end{aligned}
\end{equation}
\item We have
\begin{equation}\label{exit:delta}
 \sup_{j\ge1}\frac{d_j}{\ell_j}<1,\qquad
 \sum_{j=1}^{\infty}\exp\big(-c_0\frac{\ell_j}{d_j}\big)<\infty
 \quad\text{for every }c_0>0.
\end{equation}
\item We have
\begin{equation}\label{exit:matrix-small}
 \|A_-\|_{\mathcal L(\ell^2)}\le\tfrac12,\qquad
 \|A_+\|_{\mathcal L(\ell^2)}\le\tfrac12,
\end{equation}
and
\begin{equation}\label{exit:forcing-small}
 (F_j/a_j)_{j\ge1}\in\ell^2_+,\qquad
 \left\|(I-A_-)^{-1}(F_j/a_j)_{j\ge1}\right\|_{\ell^2}
 \le\rho/4.
\end{equation}
\end{enumerate}
\end{lemma}

\begin{proof}
We first fix the cutoffs and comparison coefficients, verify
the estimates for both semiflows, and then choose a common
radius using the smallness of the full source sequence.

The matrices in \eqref{exit:matrices} are weighted shifts.
By \eqref{app:forcing-size} and \eqref{app:endpoint-gaps},
\[
 \begin{aligned}
 \|A_-\|_{\mathcal L(\ell^2)}
 &=\sup_{j\ge2}\frac{\gamma_j}{a_j-a_{j-1}}\le C_s\varepsilon_*,\\
 \|A_+\|_{\mathcal L(\ell^2)}
 &=\sup_{j\ge2}\frac{\gamma_j}{b_j-b_{j-1}}\le C_s\varepsilon_*.
 \end{aligned}
\]
Here $C_s$ depends only on $s$, by
Proposition~\ref{app:radial-inequality}.
We first choose $\varepsilon_*$ sufficiently small, depending only on $s$,
for \eqref{exit:matrix-small}, and then fix the corresponding cutoffs
and coefficients.
Equations~\eqref{app:rates}, \eqref{app:first-endpoints}, and
\eqref{app:first-cut-threshold} give
\[
 \frac{d_j}{\ell_j}=\frac D{\log\Lambda_j}<\frac14\quad(j\ge2),
 \qquad \frac{d_1}{\ell_1}<1.
\]
For every $c_0>0$,
\[
 \sum_{j\ge2}e^{-c_0\ell_j/d_j}
 =\sum_{j\ge2}\Lambda_j^{-c_0/D}<\infty
\]
by Lemma~\ref{app:cuts}(iii). This proves (iii).

We next verify (i) and (ii) for every sufficiently small
$\rho\in(0,r_*/2)$, restricting the fixed coordinate map $\Psi$
as described above. For $\Phi_t^{\rm NS}$,
Theorem~\ref{geo:chart} gives smooth finite factors on
$B_{2\rho}(\mathcal Z_{\le N})$, satisfying \eqref{exit:finite-factor}
and $V^{(N)}(0)=0$.
The norm identity \eqref{geo:isometry-origin} and
Proposition~\ref{geo:radial-factor}(i), with $A\ge I$,
give \eqref{exit:inward} with $\kappa=\nu/2$, uniformly in $N$.
Proposition~\ref{app:radial-inequality}(i), applied at radius $2\rho$,
gives \eqref{exit:radial} with $F_j=F_j(2\rho)$.

For $\Phi_t^{\rm heat}$, the finite-prefix vector field is
$V^{(N)}=-\nu A|_{\mathcal Z_{\le N}}$, so
\eqref{exit:finite-factor} holds and $A\ge I$ gives
\[
 \langle V^{(N)}(u),u\rangle_{\mathcal Z}
 =-\nu\|A^{1/2}u\|_{\mathcal Z}^2
 \le-\nu\|u\|_{\mathcal Z}^2.
\]
The choice of $D$ in the proof of
Proposition~\ref{app:radial-inequality} ensures $\nu g_j\le d_j$
for $j\ge2$. Together with \eqref{geo:block-spaces} and
\eqref{app:first-endpoints}, this gives
\[
 \sigma(\nu A|_{Z_j})\subset[a_j,b_j],\qquad j\ge1.
\]
Consequently,
\[
 \bigl|\langle(-\nu A+\ell_j)z_j,z_j\rangle_{Z_j}\bigr|
 \le d_j\|z_j\|_{Z_j}^2,
\]
which implies \eqref{exit:radial} with the same nonnegative
$c_{jk}$ and $F_j$. Thus (i) and (ii) hold for both semiflows.

Finally, the Neumann series and \eqref{exit:matrix-small} give
$\|(I-A_-)^{-1}\|_{\mathcal L(\ell^2)}\le2$.
By \eqref{app:small-source},
\begin{equation}\label{app:fixed-radius-smallness}
 \Big\|(I-A_-)^{-1}
       \Big(\frac{F_j(2\rho)}{a_j}\Big)_{j\ge1}\Big\|_{\ell^2}
 \le2\Big\|\Big(\frac{F_j(2\rho)}{a_j}\Big)_{j\ge1}\Big\|_{\ell^2}
 =o(\rho)\qquad(\rho\to0).
\end{equation}
Choose $\rho>0$ small enough for the preceding semiflow estimates
and for this norm to be at most $\rho/4$, and set $F_j=F_j(2\rho)$.
This proves \eqref{exit:forcing-small}. Since the estimate controls
the entire source sequence, the same radius works for every
finite prefix of both semiflows, with all cutoffs, rates,
coupling coefficients, and matrices fixed.
\end{proof}

We fix these choices for the rest of the section.
By \eqref{exit:inward}, the finite closed balls are positively invariant,
and
\[
 \|\phi_t^{(N)}u\|_{\mathcal Z}\le e^{-\kappa t}\|u\|_{\mathcal Z},
 \qquad t\ge0,\quad u\in\overline B_\rho(\mathcal Z_{\le N}).
\]
We now construct the exit maps for either of the two semiflows.

\subsection{Finite-dimensional conjugacies via exit times}
We use transverse sections and exit times to construct
finite-dimensional conjugacies compatible across truncation
levels, preparing their assembly into a conjugacy on the
full Hilbert space.

For $0\ne u\in\overline B_\rho(\mathcal Z_{\le N})$, we define
\[
 \tau_N(u)=\inf\{t\ge0:
       \|\phi^{(N)}_{-t}u\|_{\mathcal Z}=\rho\}.
\]
By the finite-dimensional continuation theorem
\cite[Chapter~2]{HartmanODE} and \eqref{exit:inward},
the backward solution exists until it reaches the sphere, and
\begin{equation}\label{exit:back-growth}
 \|\phi^{(N)}_{-t}u\|_{\mathcal Z}
 \ge e^{\kappa t}\|u\|_{\mathcal Z},
 \qquad 0\le t\le\tau_N(u).
\end{equation}
Hence
\[
 0\le\tau_N(u)\le
 \kappa^{-1}\log\frac{\rho}{\|u\|_{\mathcal Z}}<\infty.
\]
We set $\tau_N(0)=\infty$ and define
\begin{equation}\label{exit:coordinate}
 H_j(z_{\le j})=
 e^{-\ell_j\tau_j(z_{\le j})}
 \bigl[\phi^{(j)}_{-\tau_j(z_{\le j})}(z_{\le j})\bigr]_j
 \quad(z_{\le j}\ne0),
 \qquad H_j(0)=0.
\end{equation}
We set
\[
 H^{(N)}(z_{\le N})
 =\bigl(H_j(z_{\le j})\bigr)_{j=1}^N.
\]
Since $H_j$ depends only on $z_{\le j}$, we have
\[
 \pi_{\le M}H^{(N)}
 =H^{(M)}\pi_{\le M},\qquad M\le N.
\]

\Needspace{6\baselineskip}
We prove that the finite exit maps are homeomorphisms and describe their images recursively.

\nopagebreak[4]
\begin{lemma}\label{exit:finite-homeomorphism}
For every $N\ge1$, the map $H^{(N)}$ with components defined by
\eqref{exit:coordinate} satisfies the following properties.
\begin{enumerate}
\renewcommand{\theenumi}{\roman{enumi}}
\renewcommand{\labelenumi}{(\theenumi)}
\item The map
\[
 H^{(N)}:\overline B_\rho(\mathcal Z_{\le N})
 \longrightarrow\mathcal Z_{\le N}
\]
is continuous and injective. The set
$\mathcal O_N=H^{(N)}(B_\rho(\mathcal Z_{\le N}))$
is open in $\mathcal Z_{\le N}$, and
\[
 H^{(N)}:B_\rho(\mathcal Z_{\le N})
 \longrightarrow\mathcal O_N
\]
is a homeomorphism.
\item With $\mathcal Z_{\le0}=\mathcal O_0=\{0\}$ and
$H^{(0)}(0)=0$, we have
\begin{equation}\label{exit:finite-image}
 \mathcal O_N=
 \left\{(h_{<N},h_N)\in\mathcal O_{N-1}\times Z_N:
 \|(H^{(N-1)})^{-1}(h_{<N})\|_{\mathcal Z}^2
 +\|h_N\|_{Z_N}^2<\rho^2\right\}.
\end{equation}
\item For every $u\in\overline B_\rho(\mathcal Z_{\le N})$
and $t\ge0$, we have
\begin{equation}\label{exit:finite-conjugacy}
 H^{(N)}(\phi_t^{(N)}u)=D_tH^{(N)}(u),
\end{equation}
where $D_t$ is given by \eqref{exit:diagonal-semigroup},
restricted to $\mathcal Z_{\le N}$.
\end{enumerate}
\end{lemma}

\begin{proof}
We first establish continuity and the conjugacy identity.
We then construct the inverse recursively by solving a strictly
monotone scalar equation, which also determines the image.

For $0\ne u\in\overline B_\rho(\mathcal Z_{\le j})$,
\eqref{exit:inward} gives
\[
 \left.\frac d{dt}\|\phi_{-t}^{(j)}u\|_{\mathcal Z}
 \right|_{t=\tau_j(u)}\ge\kappa\rho>0.
\]
The implicit-function theorem \cite[Chapter~4]{Deimling1985}
and smooth dependence for finite-dimensional ODEs
\cite[Chapter~5]{HartmanODE} therefore give continuity of
$\tau_j$ and the exit point away from zero, including on the
boundary since the vector field is defined on $B_{2\rho}$.
For each fixed $T>0$, continuous dependence at the equilibrium
keeps the backward orbit of sufficiently small $u$ inside $B_\rho$
on $[0,T]$. Thus $\tau_j(u)\to\infty$ as $u\to0$, and
\[
 \|H_j(u)\|_{Z_j}\le\rho e^{-\ell_j\tau_j(u)}\to0.
\]
Hence every $H_j$, and therefore $H^{(N)}$, is continuous.

Uniqueness of the flow and strict radial decrease give
$\tau_j(\phi_t^{(j)}u)=\tau_j(u)+t$ for $u\ne0$, with the same
exit point. Substituting into \eqref{exit:coordinate} yields
\[
 H_j(\phi_t^{(j)}u)=e^{-\ell_jt}H_j(u),\qquad t\ge0.
\]
This also holds at zero. Combining the component identities with
\eqref{exit:finite-factor} proves (iii).

We prove the remaining assertions by induction on $N$, starting
with $H^{(0)}(0)=0$ and the constant zeroth flow on $\{0\}$.
Assume (i) and (ii) for $N-1$.
Given $h_{<N}\in\mathcal O_{N-1}$, set
$u=(H^{(N-1)})^{-1}(h_{<N})$ and suppose
$\|u\|_{\mathcal Z}^2+\|h_N\|_{Z_N}^2<\rho^2$.
The zero pair has preimage zero. Otherwise, define
\begin{equation}\label{exit:root}
 G(T)=\|\phi_{-T}^{(N-1)}u\|_{\mathcal Z}^2
       +e^{2\ell_NT}\|h_N\|_{Z_N}^2.
\end{equation}
While the shorter backward orbit remains in the closed ball,
\eqref{exit:inward} gives
\[
 G'(T)\ge2\min\{\kappa,\ell_N\}G(T)>0.
\]
If $u\ne0$, then $G(0)<\rho^2\le G(\tau_{N-1}(u))$;
if $u=0$, then $h_N\ne0$ and $G(T)\to\infty$.
Thus there is a unique $T>0$ with $G(T)=\rho^2$. Set
\[
 v=\bigl(\phi_{-T}^{(N-1)}u,e^{\ell_NT}h_N\bigr),
 \qquad z=\phi_T^{(N)}v.
\]
By \eqref{exit:inward} and \eqref{exit:finite-factor},
\[
 \|v\|_{\mathcal Z}=\rho,\qquad
 z\in B_\rho(\mathcal Z_{\le N}),\qquad z_{<N}=u,
 \qquad\tau_N(z)=T.
\]
Equation~\eqref{exit:coordinate} and the induction hypothesis
then give $H^{(N)}(z)=h$.

Conversely, suppose $h=H^{(N)}(z)$ for
$z\in\overline B_\rho(\mathcal Z_{\le N})$.
By injectivity at level $N-1$, $h_{<N}$ uniquely determines
$u=z_{<N}$. If $z\ne0$, its exit time $T=\tau_N(z)$ satisfies
\[
 \phi_{-T}^{(N)}z
 =\bigl(\phi_{-T}^{(N-1)}u,e^{\ell_NT}h_N\bigr),
 \qquad G(T)=\rho^2,
\]
by \eqref{exit:coordinate} and \eqref{exit:finite-factor}.
If $h=0$, then $u=0$ and this exit point would be zero,
a contradiction. Thus zero has no other preimage.
For $h\ne0$, the same strictly increasing function $G$ uniquely
determines $T\ge0$ and hence $z=\phi_T^{(N)}v$.
This proves injectivity on the closed ball.
For nonzero $z$, an interior point has $T>0$ and therefore
$G(0)<\rho^2$, whereas a boundary point has $T=0$ and
$G(0)=\rho^2$. The strict inequality also holds at $z=0$.
Together with the inverse construction, this proves
\eqref{exit:finite-image}.

The finite-dimensional closed ball is compact, so the continuous
injection $H^{(N)}$ has a continuous inverse on its image.
By the induction hypothesis, the strict inequality in
\eqref{exit:finite-image} defines an open set $\mathcal O_N$.
This completes the induction and proves (i) and (ii).
\end{proof}

For $0\ne z_{\le j}\in\overline B_\rho(\mathcal Z_{\le j})$,
we use \eqref{exit:finite-factor} to obtain
\[
 \|\phi^{(k)}_{-t}(z_{\le k})\|_{\mathcal Z}
 \le \|\phi^{(j)}_{-t}(z_{\le j})\|_{\mathcal Z}
 \le\rho,
 \qquad 0\le t\le\tau_j(z_{\le j}),\quad k<j.
\]
By \eqref{exit:back-growth}, we conclude that
\begin{equation}\label{exit:nested-times}
 \tau_j(z_{\le j})\le\tau_k(z_{\le k}),\qquad k<j.
\end{equation}
The same inequality holds for $z_{\le j}=0$, since $\tau_k(0)=\infty$.

\subsection{Two scalar domination estimates}

We derive scalar bounds for the finite-dimensional conjugacies
and their inverses, together with uniform tail estimates along
convergent sequences. These bounds allow the compatible maps
to assemble into a homeomorphism in the strong $H^s$ topology.

For $z\in\mathcal Z$, we write
$|z|=(\|z_j\|_{Z_j})_{j\ge1}\in\ell^2_+$.
We extend finite block vectors and scalar sequences by zero.
For the power estimates, we set
\begin{equation}\label{exit:power-exponents}
 \delta_j=\frac{d_j}{\ell_j},\qquad p_j=(1+\delta_j)^{-1},
\end{equation}
and define coordinatewise maps on $\ell^2_+$ by
\begin{equation}\label{exit:powers}
 (\mathcal P_-u)_j=\rho^{\delta_j}u_j^{1-\delta_j},\qquad
 (\mathcal P_+u)_j=\rho^{1-p_j}u_j^{p_j}.
\end{equation}
We also set
\begin{equation}\label{exit:resolvents}
 R_\pm=(I-A_\pm)^{-1}.
\end{equation}
We prove continuity of the power maps and establish the tail property
of $R_\pm$.

\nopagebreak[4]
\begin{lemma}\label{exit:power-continuity}
\begin{enumerate}
\renewcommand{\theenumi}{\roman{enumi}}
\renewcommand{\labelenumi}{(\theenumi)}
\item The maps $\mathcal P_\pm:\ell^2_+\to\ell^2_+$
defined by \eqref{exit:powers} are continuous and satisfy
$\mathcal P_\pm(0)=0$.

\item The operators $R_\pm\in\mathcal L(\ell^2)$ are positive and satisfy
\[
 \pi_{\le N}R_\pm
 =\pi_{\le N}R_\pm\pi_{\le N},
 \qquad N\ge1.
\]
If $u^{(n)}\to u$ in $\ell^2$, then
\[
 \lim_{N\to\infty}\sup_{n\ge1}
 \bigl\|(I-\pi_{\le N})R_\pm u^{(n)}\bigr\|_{\ell^2}=0.
\]
\end{enumerate}
\end{lemma}

\begin{proof}
A summable scalar majorant gives continuity of the power maps.
The assertions for the resolvents follow from their Neumann
series and compactness of convergent sequences.

For $0<\beta<1$ and $K>1$, maximizing $\xi^{1-\beta}-K\xi$
and taking $\xi=t^2/\rho^2$ gives
\begin{equation}\label{exit:power-majorant}
 \rho^{2\beta}t^{2(1-\beta)}
 \le Kt^2+K\rho^2e^{-\log K/\beta},\qquad t\ge0.
\end{equation}
For $\beta=0$, the same inequality holds with the last term set to zero.
Fix $c_0>0$ and set $K=e^{c_0}$. Applying
\eqref{exit:power-majorant} with $\beta=\delta_j$ and
$\beta=1-p_j=\delta_j/(1+\delta_j)\le\delta_j$ gives the common bound
\[
 |(\mathcal P_\pm u)_j|^2
 \le K|u_j|^2+K\rho^2e^{-c_0/\delta_j}.
\]
Summing and using \eqref{exit:delta} shows that
$\mathcal P_\pm u\in\ell^2_+$ for every $u\in\ell^2_+$.
If $u^{(n)}\to u$ in $\ell^2_+$, the same bound gives
\[
 \sup_{n\ge1}\sum_{j>M}|(\mathcal P_\pm u^{(n)})_j|^2
 \le K\sup_{n\ge1}\sum_{j>M}|u_j^{(n)}|^2
      +K\rho^2\sum_{j>M}e^{-c_0/\delta_j}
 \to0\quad\text{as }M\to\infty,
\]
by the uniform tail property of convergent sequences and
\eqref{exit:delta}. Scalar continuity on each finite prefix,
together with this uniform tail bound, yields
$\mathcal P_\pm u^{(n)}\to\mathcal P_\pm u$ in $\ell^2$.
Since \eqref{exit:powers} gives $\mathcal P_\pm(0)=0$, this proves (i).

For (ii), the Neumann-series argument in the proof of
Lemma~\ref{app:common-data} gives
$R_\pm=\sum_{q\ge0}A_\pm^q$ in operator norm and $\|R_\pm\|\le2$.
Each term is positive and lower triangular, so the limit has
the stated positivity and prefix identity.
If $u^{(n)}\to u$, then $R_\pm u^{(n)}\to R_\pm u$;
these sequences together with their limits form compact sets.
The coordinate projections converge uniformly to the identity
on each such set, proving the asserted uniform tail estimate.
\end{proof}

The next lemma controls the distortion of the finite conjugacies
discussed in the introduction by bounding block norms in both
directions, uniformly in $N$. Together with
Lemma~\ref{exit:power-continuity}, these bounds preserve square
summability and give uniformly small tails along convergent input
sequences, as used in Proposition~\ref{exit:hilbert-conjugacies}.
We write
\begin{equation}\label{exit:normalized-source}
 f=(f_j)_{j\ge1},\qquad f_j=F_j/a_j.
\end{equation}

\nopagebreak[4]
\begin{lemma}\label{exit:envelopes}
Let $N\ge1$, $z\in\overline B_\rho(\mathcal Z_{\le N})$,
and $h=H^{(N)}(z)$. We extend the sequences
\[
 |z|=(\|z_j\|_{Z_j})_{j=1}^N,\qquad
 |h|=(\|h_j\|_{Z_j})_{j=1}^N
\]
by zero for $j>N$. Then, for every $1\le j\le N$,
\begin{enumerate}
\renewcommand{\theenumi}{\roman{enumi}}
\renewcommand{\labelenumi}{(\theenumi)}
\item
\begin{equation}\label{exit:inverse-envelope}
 \|z_j\|_{Z_j}
 \le \bigl[R_-(\mathcal P_-|h|+f)\bigr]_j.
\end{equation}
\item
\begin{equation}\label{exit:forward-envelope}
 \|h_j\|_{Z_j}
 \le \bigl[\mathcal P_+\bigl(R_+(|z|+f)\bigr)\bigr]_j.
\end{equation}
\end{enumerate}
\end{lemma}

\begin{proof}
We derive scalar inequalities for the block norms and use
induction along backward prefixes to obtain the two bounds.
The endpoint rate gaps give the comparison matrices, while
the exit sphere gives the power factors.

Along a forward solution, set $r_j(t)=\|z_j(t)\|_{Z_j}$.
By \eqref{exit:radial} and \eqref{exit:rates}, we obtain
\begin{equation}\label{exit:scalar-forward}
 -b_jr_j-\sum_{k<j}c_{jk}r_k-F_j
 \le r_j'
 \le-a_jr_j+\sum_{k<j}c_{jk}r_k+F_j
 \quad\text{a.e.}
\end{equation}
We divide \eqref{exit:radial} by $r_j$ where $r_j>0$;
on $\{r_j=0\}$, we use $r_j'=0$ almost everywhere by absolute continuity.

For both parts, the case $z_{\le j}=0$ is immediate.
For a nonzero prefix, set $T=\tau_j(z_{\le j})$ and use the backward
block norms
\[
 r_k(t)=\bigl\|[\phi_{-t}^{(j)}(z_{\le j})]_k\bigr\|_{Z_k},
 \qquad 0\le t\le T,\quad k\le j.
\]
By \eqref{exit:nested-times} and \eqref{exit:finite-factor},
all shorter prefixes follow their own backward flows in the
closed ball throughout $[0,T]$.

For (i), we induct on $j$, proving the bound for every point
in the corresponding closed ball. Integrating the upper
inequality in \eqref{exit:scalar-forward} from the exit point
to $z_{\le j}$ gives
\begin{equation}\label{exit:integral-inverse}
 \|z_j\|_{Z_j}\le e^{-a_jT}r_j(T)
 +\int_0^T e^{-a_jt}
       \Big(\sum_{k<j}c_{jk}r_k(t)+F_j\Big)\,dt.
\end{equation}
By \eqref{exit:coordinate}, we have
$\|h_j\|_{Z_j}=e^{-\ell_jT}r_j(T)$ and $r_j(T)\le\rho$.
If $r_j(T)>0$, \eqref{exit:rates} and \eqref{exit:powers} give
\[
 e^{-a_jT}r_j(T)
 =\|h_j\|_{Z_j}^{1-\delta_j}r_j(T)^{\delta_j}
 \le(\mathcal P_-|h|)_j.
\]
When $r_j(T)=0$, the inequality holds directly.

We set $\omega=R_-(\mathcal P_-|h|+f)$.
For $k<j$, the shorter backward prefix has conjugacy coordinates
$e^{\ell_it}h_i$ for $i\le k$ by \eqref{exit:finite-conjugacy}.
Apply the induction hypothesis at that point. The power formula
and positivity and prefix identity of $R_-$ from
Lemma~\ref{exit:power-continuity}(ii) give
\[
 r_k(t)\le\sum_{i\le k}(R_-)_{ki}
       \bigl(e^{a_it}(\mathcal P_-|h|)_i+f_i\bigr)
 \le e^{a_kt}\omega_k,
\]
where we use $0<a_i\le a_k$ and $(R_-)_{ki}\ge0$.
Substituting into \eqref{exit:integral-inverse} and using
\eqref{exit:matrices}, \eqref{exit:resolvents}, and
\eqref{exit:normalized-source}, we obtain
\[
 \|z_j\|_{Z_j}
 \le(\mathcal P_-|h|)_j+f_j
       +\sum_{k<j}\frac{c_{jk}}{a_j-a_k}\omega_k
 =(\mathcal P_-|h|+f+A_-\omega)_j=\omega_j.
\]
This proves (i).

For (ii), set $Y=R_+(|z|+f)$ and prove
$r_j(t)\le e^{b_jt}Y_j$ along each backward prefix by induction on $j$.
Reversing time in the lower inequality of \eqref{exit:scalar-forward},
we obtain
\[
 r_j'\le b_jr_j+\sum_{k<j}c_{jk}r_k+F_j
 \quad\text{a.e. on }[0,T].
\]
The shorter-prefix induction bounds hold throughout $[0,T]$.
Multiplying by $e^{-b_jt}$ and integrating, we obtain
\begin{align*}
 e^{-b_jt}r_j(t)
 &\le\|z_j\|_{Z_j}
   +\sum_{k<j}c_{jk}Y_k\int_0^t e^{-(b_j-b_k)s}\,ds
   +F_j\int_0^t e^{-b_js}\,ds\\
 &\le\|z_j\|_{Z_j}+\sum_{k<j}(A_+)_{jk}Y_k+f_j
 =Y_j,
\end{align*}
by \eqref{exit:matrices}, $F_j/b_j\le F_j/a_j=f_j$, and
$(I-A_+)Y=|z|+f$.
By \eqref{exit:rates} and \eqref{exit:power-exponents},
we have $p_jb_j=\ell_j$. If $r_j(T)>0$, we use
\eqref{exit:coordinate}, $r_j(T)\le\rho$, and \eqref{exit:powers}
to obtain
\[
 \|h_j\|_{Z_j}
 =\bigl(e^{-b_jT}r_j(T)\bigr)^{p_j}r_j(T)^{1-p_j}
 \le\rho^{1-p_j}Y_j^{p_j}
 =(\mathcal P_+Y)_j.
\]
If $r_j(T)=0$, then $h_j=0$ and the same bound holds.
In both inductions, the sums are empty for $j=1$, and the
integral estimates remain valid when $T=0$. This completes the proof.
\end{proof}

\subsection{Construction of the Hilbert-space conjugacies}

We assemble the compatible finite-dimensional conjugacies
into homeomorphisms in the strong $H^s$ topology, both onto
a common ball. Composing one with the inverse of the other
and the coordinate map $\Psi$ yields the desired conjugacy
between the Navier--Stokes semiflow and the heat semigroup.

\begin{proposition}\label{exit:hilbert-conjugacies}
There exist $\eta>0$ and open neighborhoods
$U_{\rm NS},U_{\rm heat}\subset B_\rho(\mathcal Z)$ of zero
such that the maps defined componentwise by
\eqref{exit:coordinate} for the respective semiflows
satisfy the following assertions.
\begin{enumerate}
\renewcommand{\theenumi}{\roman{enumi}}
\renewcommand{\labelenumi}{(\theenumi)}
\item The maps
\[
 H_{\rm NS}:U_{\rm NS}\to B_\eta(\mathcal Z),
 \qquad
 H_{\rm heat}:U_{\rm heat}\to B_\eta(\mathcal Z)
\]
are homeomorphisms in the $\mathcal Z$ norm and fix zero.
\item For every $t\ge0$, we have
\[
 \Phi_t^{\rm NS}(U_{\rm NS})\subset U_{\rm NS},
 \qquad
 \Phi_t^{\rm heat}(U_{\rm heat})\subset U_{\rm heat},
\]
and
\[
\begin{aligned}
 H_{\rm NS}(\Phi_t^{\rm NS}z)&=D_tH_{\rm NS}(z),
 &&z\in U_{\rm NS},\\
 H_{\rm heat}(\Phi_t^{\rm heat}z)&=D_tH_{\rm heat}(z),
 &&z\in U_{\rm heat},
\end{aligned}
\]
where $D_t$ is defined by \eqref{exit:diagonal-semigroup}.
\end{enumerate}
\end{proposition}

\begin{proof}
We treat either semiflow $\Phi_t$ in
Subsection~\ref{app:concrete-flows}, using the common coefficients
of Lemma~\ref{app:common-data}. We assemble the finite maps,
construct a common ball in their images, and prove continuity
of the inverses.

We first define
$H(z)=(H_j(z_{\le j}))_{j\ge1}$ for $z\in B_\rho(\mathcal Z)$
by \eqref{exit:coordinate}.
Applying Lemma~\ref{exit:envelopes} to each finite prefix and
using the prefix identity in Lemma~\ref{exit:power-continuity}(ii),
we obtain
\begin{equation}\label{exit:infinite-envelopes}
 |H(z)|\le\mathcal P_+\bigl(R_+(|z|+f)\bigr),\qquad
 |z|\le R_-\bigl(\mathcal P_-|H(z)|+f\bigr).
\end{equation}
The first bound and Lemma~\ref{exit:power-continuity} give
$H(z)\in\mathcal Z$, so the second bound is also well defined.
By Lemma~\ref{exit:finite-homeomorphism}(i), $H$ is injective,
fixes zero, and satisfies
$\pi_{\le N}H=H^{(N)}\pi_{\le N}$.
Equations~\eqref{exit:inward} and \eqref{exit:finite-factor} give
$\|\Phi_tz\|_{\mathcal Z}\le e^{-\kappa t}\|z\|_{\mathcal Z}$.
The finite conjugacy identity \eqref{exit:finite-conjugacy}
therefore yields
\begin{equation}\label{exit:global-identity}
 H(\Phi_tz)=D_tH(z),\qquad z\in B_\rho(\mathcal Z),\quad t\ge0.
\end{equation}

To prove continuity, let $z^{(n)}\to z$ in $B_\rho(\mathcal Z)$.
Since $|z^{(n)}|\to|z|$ in $\ell^2$, Lemma~\ref{exit:power-continuity}
gives
\[
 \mathcal P_+\bigl(R_+(|z^{(n)}|+f)\bigr)
 \to\mathcal P_+\bigl(R_+(|z|+f)\bigr)
 \quad\text{in }\ell^2.
\]
Convergent sequences in $\ell^2$ have uniformly small tails;
hence \eqref{exit:infinite-envelopes} implies
\[
 \lim_{M\to\infty}\sup_{n\ge1}
 \sum_{j>M}\|H_j(z^{(n)}_{\le j})\|_{Z_j}^2=0.
\]
For each fixed $M$, Lemma~\ref{exit:finite-homeomorphism}(i) gives
$\pi_{\le M}H(z^{(n)})\to\pi_{\le M}H(z)$.
Together these facts yield $H(z^{(n)})\to H(z)$ in $\mathcal Z$.
This argument also proves continuity at zero: the fixed source
$f$ need not vanish, but the envelope controls tails and
finite-prefix continuity controls the remaining coordinates.

We next construct a ball contained in the image of $H$.
By \eqref{exit:forcing-small}, $\|R_-f\|_{\ell^2}\le\rho/4$.
Since $\mathcal P_-$ is continuous and vanishes at zero,
Lemma~\ref{exit:power-continuity} allows us to choose
$0<\eta\le\rho/4$ such that
\begin{equation}\label{exit:eta-choice}
 \|h\|_{\mathcal Z}<\eta
 \quad\Longrightarrow\quad
 \|R_-\mathcal P_-|h|\|_{\ell^2}<\rho/4.
\end{equation}
This choice uses only the common coefficients and $\rho$.
Fix $h\in B_\eta(\mathcal Z)$ and set
\[
 \omega=R_-(\mathcal P_-|h|+f),\qquad \|\omega\|_{\ell^2}<\rho/2.
\]
We construct the finite inverses by induction, starting with the
empty prefix. If $u=(H^{(N-1)})^{-1}(h_{<N})$ exists,
Lemmas~\ref{exit:envelopes}(i) and \ref{exit:power-continuity}(ii)
give $|u|\le\pi_{<N}\omega$. Thus
\[
 \|u\|_{\mathcal Z}^2+\|h_N\|_{Z_N}^2
 <\frac{\rho^2}{4}+\frac{\rho^2}{16}<\rho^2.
\]
The image description \eqref{exit:finite-image} gives the unique
inverse $z^{[N]}=(H^{(N)})^{-1}(h_{\le N})$
in $B_\rho(\mathcal Z_{\le N})$.
By Lemma~\ref{exit:finite-homeomorphism} and the inverse envelope,
these inverses are compatible and satisfy
\begin{equation}\label{exit:inverse-dominated}
 \pi_{<N}z^{[N]}=z^{[N-1]},\qquad
 \|z_j^{[N]}\|_{Z_j}\le \omega_j\quad(j\le N).
\end{equation}
Extending them by zero, we obtain, for $N>M$,
\[
 \|z^{[N]}-z^{[M]}\|_{\mathcal Z}^2
 \le\sum_{j>M}\omega_j^2\to0\quad\text{as }M\to\infty.
\]
Their limit $z$ satisfies $\|z\|_{\mathcal Z}\le\|\omega\|_{\ell^2}<\rho/2$,
and compatibility gives $H_j(z_{\le j})=h_j$ for every $j$.
Thus $H(z)=h$ and $B_\eta(\mathcal Z)\subset H(B_\rho(\mathcal Z))$.

We finally prove continuity of the inverse and restrict the maps
to the required neighborhoods.
Set $U=H^{-1}(B_\eta(\mathcal Z))$. By continuity and injectivity
of $H$, this is an open neighborhood of zero and
$H:U\to B_\eta(\mathcal Z)$ is a bijection.
Let $h^{(n)}\to h$ in $B_\eta(\mathcal Z)$ and write
$z^{(n)}=H^{-1}(h^{(n)})$, $z=H^{-1}(h)$.
Lemma~\ref{exit:power-continuity} gives
\[
 R_-(\mathcal P_-|h^{(n)}|+f)
 \to R_-(\mathcal P_-|h|+f)\quad\text{in }\ell^2.
\]
The inverse bound in \eqref{exit:infinite-envelopes} therefore
gives uniformly small tails for $z^{(n)}$.
For each fixed $M$, prefix compatibility and
Lemma~\ref{exit:finite-homeomorphism}(i) give
$\pi_{\le M}z^{(n)}\to\pi_{\le M}z$.
The finite-prefix and tail argument above yields
$z^{(n)}\to z$ in $\mathcal Z$, proving (i).
Since $D_tB_\eta(\mathcal Z)\subset B_\eta(\mathcal Z)$,
\eqref{exit:global-identity} implies $\Phi_tU\subset U$
and proves (ii).
Applying the construction to $\Phi_t^{\rm NS}$ and
$\Phi_t^{\rm heat}$ gives $H_{\rm NS}$ and $H_{\rm heat}$
with the same target ball $B_\eta(\mathcal Z)$ and diagonal
semigroup $D_t$.
\end{proof}

\subsection{Completion of the proof}

\begin{proof}[Completion of the proof of Theorem~\ref{app:main}]
By Proposition~\ref{exit:hilbert-conjugacies}, the homeomorphisms
$H_{\rm NS}:U_{\rm NS}\to B_\eta(\mathcal Z)$ and
$H_{\rm heat}:U_{\rm heat}\to B_\eta(\mathcal Z)$
conjugate the two semiflows to the same $D_t$.
We define
\begin{equation}\label{app:final-conjugacy}
\begin{gathered}
 U=\Psi^{-1}(U_{\rm NS}),\qquad V=U_{\rm heat},
 \\
 \mathcal H=H_{\rm heat}^{-1}\circ H_{\rm NS}\circ\Psi.
\end{gathered}
\end{equation}
The sets in \eqref{app:final-conjugacy} are open neighborhoods of zero, and each map and its inverse
is strongly continuous.  Since $D_tB_\eta\subset B_\eta$, both neighborhoods
are positively invariant.  Finally,
\[
 \mathcal H S_t
 =H_{\rm heat}^{-1}D_tH_{\rm NS}\Psi
 =e^{-\nu At}\mathcal H
\]
on $U$ for every $t\ge0$.  This proves \eqref{app:conjugacy}.
\end{proof}

\section*{Acknowledgements}
The authors used OpenAI's ChatGPT to assist with language editing, reference checking, and reviewing the proofs for clarity and internal consistency. The authors take full responsibility for the
mathematical arguments, references, and final manuscript.

\section*{Statements and declarations}
\noindent\textbf{Data availability.}
No datasets were generated or analysed in this study.


\end{document}